\documentclass[11pt]{amsart}
\usepackage{amsthm,amssymb,mathrsfs,mathtools,empheq,enumitem}

\usepackage[T1]{fontenc}

\usepackage[notref,notcite,final]{showkeys}
\mathtoolsset{showonlyrefs}
\usepackage{fullpage,color}

\newtheorem{mainthm}{Theorem}
\newtheorem{theorem}{Theorem}[section]
\newtheorem*{theorem*}{Theorem}
\newtheorem{corollary}[theorem]{Corollary}

\newtheorem{lemma}[theorem]{Lemma}
\newtheorem{proposition}[theorem]{Proposition}
\newtheorem*{proposition*}{Proposition}

\newtheorem*{conjecture*}{Conjecture}

\theoremstyle{definition}

\newtheorem{remark}[theorem]{Remark}

\numberwithin{equation}{section}

\def\bN {\mathbb{N}}

\def\bR {\mathbb{R}}
\def\bS {\mathbb{S}}

\def\bZ {\mathbb{Z}}

\def\cA {\mathcal{A}}

\def\cE {\mathcal{E}}

\def\cH {\mathcal{H}}

\def\cY {\mathcal{Y}}
\def\cZ {\mathcal{Z}}

\def\scrL{\mathscr{L}}

\def\grad {{\nabla}}

\def\rstr {{\big |}}

\def\la {\langle}
\def\ra {\rangle}

\newcommand{\tx}[1]{\mathrm{#1}}
\newcommand{\wto}{\rightharpoonup}
\newcommand{\sto}[1]{\xrightarrow[#1]{}}

\newcommand{\wt}[1]{\widetilde{#1}}
\newcommand{\bs}[1]{\boldsymbol{#1}}

\newcommand{\supp}{\operatorname{supp}}

\newcommand{\dist}{\operatorname{dist}}

\newcommand{\Id}{\operatorname{Id}}

\newcommand{\eee}{\mathrm e}

\DeclareMathOperator{\sech}{sech}

\newcommand{\ud}{\mathrm{\,d}}
\newcommand{\vd}{\mathrm{d}}
\newcommand{\udr}{\,r\vd r}
\newcommand{\vD}{\mathrm{D}}
\newcommand{\dd}[1]{{\frac{\vd}{\vd{#1}}}}
\newcommand{\pd}[1]{{\frac{\partial}{\partial{#1}}}}

\newcommand{\uln}[1]{{\underline{ #1 }}}
\newcommand{\lin}{_{\textsc{l}}}

\newcommand{\yp}{{\cY}^+}
\newcommand{\ym}{{\cY}^-}
\newcommand{\ypl}{{\cY}^+_{\lambda}}
\newcommand{\yml}{{\cY}^-_{\lambda}}

\title{Construction of two-bubble solutions \\ for energy-critical wave equations}
\author{Jacek Jendrej}
\address{CMLS, {\'E}cole polytechnique, CNRS, Universit{\'e} Paris-Saclay, 91128 Palaiseau Cedex, France}
\email{jacek.jendrej@polytechnique.edu}

\begin{document}

\begin{abstract}
  We construct pure two-bubbles for some energy-critical wave equations, that is solutions which in one time direction
  approach a superposition of two stationary states both centered at the origin, but asymptotically decoupled in scale.
  Our solution exists globally, with one bubble at a fixed scale and the other concentrating in infinite time,
  with an error tending to $0$ in the energy space. We treat the cases of the power nonlinearity in space dimension $6$,
  the radial Yang-Mills equation and the equivariant wave map equation with equivariance class $k \geq 3$.
  The concentration speed of the second bubble is exponential for the first two models and a power function in the last case.
\end{abstract}

\maketitle

\section{Introduction}
\label{sec:intro}
\subsection{Energy critical NLW}
We consider the energy critical wave equation in space dimension $N = 6$:
\begin{equation}
\label{eq:nlw0}
  \left\{
    \begin{aligned}
      &(\partial_t^2 -\Delta)u(t, x) = |u(t, x)|\cdot u(t, x), \qquad (t, x) \in \bR\times \bR^6, \\
      &(u(t_0, x), \partial_t u(t_0, x)) = (u_0(x), \dot u_0(x)).
    \end{aligned}\right.
\end{equation}
The \emph{energy functional} associated with this equation is defined for $\bs u_0 = (u_0, \dot u_0) \in \cE := \dot H^1(\bR^6) \times L^2(\bR^6)$ by the formula
\begin{equation*}
  E(\bs u_0) := \int\frac 12|\dot u_0|^2 + \frac 12|\grad u_0|^2 - F(u_0)\ud x,
\end{equation*}
where $F(u_0) := \frac 13 |u_0|^3$. Note that $E(\bs u_0)$ is well-defined due to the Sobolev Embedding Theorem.
The differential of $E$ is $\vD E(\bs u_0) = (-\Delta u_0 - f(u_0), \dot u_0)$, where $f(u_0) = |u_0|\cdot u_0$,
hence we can rewrite equation \eqref{eq:nlw0} as
\begin{equation}
  \label{eq:nlw}
  \bigg\{
    \begin{aligned}
      \partial_t \bs u(t) &= J\circ \vD E(\bs u(t)), \\
      \bs u(t_0) &= \bs u_0 \in \cE.
    \end{aligned}
\end{equation}
  Here, $J := \begin{pmatrix}0 & \Id \\ -\Id & 0 \end{pmatrix}$ is the natural symplectic structure.

  Equation \eqref{eq:nlw0} is locally well-posed in the space $\cE$, see for example Ginibre, Soffer and Velo \cite{GSV92}, Shatah and Struwe \cite{ShSt94} (the defocusing case),
  as well as a complete review of the Cauchy theory in Kenig and Merle \cite{KeMe08} (for $N \in \{3, 4, 5\}$) and Bulut, Czubak, Li, Pavlovi\'c and Zhang~\cite{BCLPZ13} (for $N \geq 6$).
  By ``well-posed'' we mean that for any initial data $\bs u_0 \in \cE$ there exists $\tau > 0$ and a unique solution in some subspace of $C([t_0 - \tau, t_0 + \tau]; \cE)$,
  and that this solution is continuous with respect to the inital data. By standard arguments, there exists a maximal time of existence $(T_-, T_+)$, $-\infty \leq T_- < t_0 < T_+ \leq +\infty$,
  and a unique solution $\bs u \in C((T_-, T_+); \cE)$. If $T_+ < +\infty$, then $\bs u(t)$ leaves every compact subset of $\cE$ as $t$ approaches $T_+$.
  A crucial property of the solutions of \eqref{eq:nlw0} is that the energy $E$ is a conservation law.

  In this paper we always assume that the initial data are radially symmetric. This symmetry is preserved by the flow.

  For functions $v \in \dot H^1$, $\dot v \in L^2$, $\bs v = (v, \dot v)\in \cE$ and $\lambda > 0$, we denote
  \begin{equation*}
    v_\lambda(x) := \frac{1}{\lambda^2} v\big(\frac{x}{\lambda}\big), \qquad \dot v_\uln\lambda(x) := \frac{1}{\lambda^3} \dot v\big(\frac{x}{\lambda}\big),\qquad\bs v_\lambda(x) := \big(v_\lambda, \dot v_\uln\lambda\big).
  \end{equation*}

  A change of variables shows that
  \begin{equation*}
    E\big((\bs u_0)_\lambda\big) = E(\bs u_0).
  \end{equation*}
  Equation~\eqref{eq:nlw0} is invariant under the same scaling: if $\bs u(t) = (u(t), \dot u(t))$ is a solution of \eqref{eq:nlw0} and $\lambda > 0$, then
  $
  t \mapsto \bs u\big(t_0 + t/\lambda\big)_\lambda
  $ is also a solution
  with initial data $(\bs u_0)_\lambda$ at time $t = 0$.
  This is why equation~\eqref{eq:nlw0} is called \emph{energy-critical}.

  A fundamental object in the study of \eqref{eq:nlw0} is the family of stationary solutions $\bs u(t) \equiv \pm\bs W_\lambda = (\pm W_\lambda, 0)$, where
  \begin{equation*}
    W(x) = \Big(1 + \frac{|x|^2}{24}\Big)^{-2}.
  \end{equation*}
  The functions $\pm W_\lambda$ are called \emph{ground states} or \emph{bubbles} (of energy). They are the only radially symmetric solutions
  of the critical elliptic problem
  \begin{equation}
    \label{eq:elliptic}
    -\Delta u - f(u) = 0.
  \end{equation}
  The functions $W_\lambda$ are, up to translation, the only positive solutions of \eqref{eq:elliptic}.
  The ground states achieve the optimal constant in the critical Sobolev inequality, which was proved by Aubin \cite{Aubin76} and Talenti \cite{Talenti76}.
  They are the ``mountain passes'' for the potential energy.

  Kenig and Merle \cite{KeMe08} exhibited the special role of the ground states $\bs W_\lambda$ as the \emph{threshold elements} for nonlinear dynamics of the solutions of \eqref{eq:nlw0}
  in space dimensions $N = 3, 4, 5$, which is believed to be a general feature of dispersive equations (the so-called \emph{Threshold Conjecture}).
  Another major problem in the field is the \emph{Soliton Resolution Conjecture}, which predicts that a bounded (in an appropriate sense) solution
  decomposes asymptotically into a sum of energy bubbles at different scales and a radiation term (a solution of the linear wave equation).
  This was proved for the radial energy-critical wave equation in dimension $N = 3$
  by Duyckaerts, Kenig and Merle \cite{DKM4}, following the earlier work of the same authors \cite{DKM3}, where such a decomposition was proved only for a sequence of times
  (this last result was generalized to any odd dimension by Rodriguez \cite{Rod14}).
  
  It is natural to examine the dynamics of the solutions of \eqref{eq:nlw0} in a neighborhood (in the energy space) of the family of the ground states.
  In dimension $N = 3$ the was done by Krieger, Schlag and Tataru \cite{KrScTa09}, who showed that such solutions can blow up in finite time (by concentration of the bubble),
  see also \cite{DonKr13}, \cite{KrSc14}, \cite{DHKS14}, \cite{HiRa12}, \cite{moi15p-1} for related results.

  In view of the rich dynamics in a neighborhood of one bubble, it was expected that solutions behaving asymptotically as a superposition of many (at least two) bubbles exist,
  in other words that the result of \cite{DKM4} is essentially optimal. We prove that it is the case when $N = 6$:

\begin{mainthm}
  \label{thm:deux-bulles}
There exists a solution $\bs u: (-\infty, T_0] \to \cE$ of \eqref{eq:nlw0} such that
  \begin{equation}
    \label{eq:mainthm}
    \lim_{t\to -\infty}\|\bs u(t) - (\bs W + \bs W_{\frac{1}{\kappa}\eee^{-\kappa|t|}})\|_\cE = 0, \qquad \text{with }\kappa := \sqrt{\frac 54}.
  \end{equation}
\end{mainthm}
\begin{remark}
  More precisely, we will prove that
  \begin{equation*}
    \big\|\bs u(t) - \big(\bs W + \bs W_{\frac{1}{\kappa}\eee^{-\kappa|t|}}\big)\big\|_\cE \leq C_1 \cdot \eee^{-\frac 12 \kappa|t|}
  \end{equation*}
  for some constant $C_1 > 0$.
\end{remark}
\begin{remark}
We construct here \emph{pure} two-bubbles, that is the solution approaches a superposition of two stationary states, with no energy transformed into radiation.
By the conservation of energy and the decoupling of the two bubbles, we necessarily have $E(\bs u(t)) = 2E(\bs W)$.
Pure one-bubbles cannot concentrate and are completely classified, see \cite{DM08}.
\end{remark}
\begin{remark}
  \label{rem:nlw-signs}
  It was proved in \cite{moi15p-3}, in any dimension $N \geq 3$, that there exist no solutions $\bs u(t): [t_0, T_+) \to \cE$
    of \eqref{eq:nlw0} such that $\|\bs u(t) - (\bs W_{\mu(t)} - \bs W_{\lambda(t)})\|_\cE \to 0$ with $\lambda(t) \ll \mu(t)$
  as $t \to T_+ \leq +\infty$.
\end{remark}
\begin{remark}
  \label{rem:dim7}
  In any dimension $N > 6$ one can expect an analogous result with concentration rate $\lambda(t) \sim |t|^{-\frac{4}{N-6}}$.
\end{remark}
\begin{remark}
  In the context of the harmonic map heat flow, Topping \cite{Topping99} proved the existence of towers of bubbles
  for a well chosen target manifold, see also a non-existence result of van der Hout \cite{vdHout03}.
\end{remark}

Let us resume the overall strategy of the proof, which is based on the previous paper of the author \cite{moi15p-1}.

In Section~\ref{sec:approx} we construct an appropriate approximate solution $\varphi(t)$. We present first a formal computation
which allows to predict the concentration rate and explains why the proof fails in dimension $N \in \{3, 4, 5\}$.
It highlights also the role of the \emph{strong interaction} between the two bubbles
(by ``strong'' we mean ``significantly altering the dynamics''; \cite{MaRa15p} provides an example
of this phenomenon in a different context).
Then we give a precise definition of the approximate solution and prove bounds on its error.

In Section~\ref{sec:bootstrap} we build a sequence $\bs u_n: [T_n, T_0] \to \cE$ of solutions of \eqref{eq:nlw0}
with $T_n \to -\infty$ and $\bs u_n(t)$ close to a two-bubble solution for $t \in [T_n, T_0]$.
Taking a weak limit finishes the proof. This type of argument goes back to the works of Merle~\cite{Merle90} and Martel~\cite{Martel05}.
The heart of the analysis is to obtain uniform energy bounds for the sequence $\bs u_n$.
To this end we follow the approach of Rapha\"el and Szeftel \cite{RaSz11},
that is we prove bootstrap estimates involving an energy functional with a virial-type correction term.
This correction is designed to cancel some terms related to the concentration of the bubble $\bs W_{\lambda(t)}$.
It has to be localized in an appropriate way, so that it does not ``see'' the other bubble.
Finally, in order to deal with the linear instabilities of the flow, we use a classical topological (``shooting'') argument.

\subsection{Critical wave maps}
We consider the wave map equation from the $2+1$-dimensional Minkowski space (the energy-critical case) to $\bS^2$.
We will consider solutions with $k$-equivariant symmetry, in which case the problem is reduced to the following scalar equation:
\begin{equation}
\label{eq:wmap0}
  \left\{
    \begin{aligned}
      &\partial_t^2 u(t, r) = \partial_r^2 u(t, r) +\frac{1}{r}\partial_r u(t, r) - \frac{k^2}{2r^2}\sin(2u(t, r)), \\
      &(u(t_0, r), \partial_t u(t_0, r)) = (u_0(r), \dot u_0(r)), \qquad t, t_0 \in \bR,\ r\in (0, +\infty).
    \end{aligned}\right.
\end{equation}
For a presentation of the geometric content of this equation, one can consult \cite{ShSt00}.
Here we will regard \eqref{eq:wmap0} as a scalar semilinear problem.

We define the space $\cH$ as the completion of $C_0^\infty((0, +\infty))$ for the norm
\begin{equation}
  \label{eq:H-norm}
  \|v\|_{\cH}^2 := 2\pi \int_0^{+\infty}\big(|\partial_r v(r)|^2 + |\frac kr v(r)|^2 \big)\udr.
\end{equation}
We will work in the energy space $\cE := \cH \times L^2$.
Equation \eqref{eq:wmap0} can be written in the form \eqref{eq:nlw} with the energy functional $E$ defined for $\bs u_0 = (u_0, \dot u_0) \in \cE$ by the formula
\begin{equation}
  \label{eq:energy-wmap}
  E(\bs u_0) := \pi \int_0^{+\infty}\big((\dot u_0)^2 + (\partial_r u_0)^2 + \frac{k^2}{r^2}(\sin(u_0))^2\big)\udr.
\end{equation}
The Cauchy theory in the energy space has been established by Shatah and Tahvildar-Zadeh \cite{ST94}.
Note that $u_0 \in \cH$ forces $\lim_{r\to +\infty}u_0(r) = 0$, but we could just as well consider states of finite energy such that $\lim_{r\to +\infty}u_0(r) \in \pi\bZ$,
see \cite{CKM08, CKLS15} for details.

The stationary solutions $W_\lambda(r) := 2\arctan\big(\big(\frac{r}{\lambda}\big)^k\big)$ play a fundamental role in the study of \eqref{eq:wmap0}.
They are the harmonic maps of topological degree $k$.
We will write $W(r) := W_1(r) = 2\arctan(r^k)$ and $\Lambda W(r) := -\pd \lambda W_\lambda\rstr_{\lambda = 1} = \frac{2k}{r^k + r^{-k}}$.
Note that $W \notin \cH$ precisely because of the fact that $W(r) \to \pi$ as $r \to +\infty$.

The possibility of concentration of a harmonic map at the origin was first observed numerically by Bizo{\'n}, Chmaj and Tabor \cite{BCT01}. Struwe \cite{Struwe03} proved that if the blow-up occurs, then $\bs W$ is the blow-up profile (for a sequence of times).
The dynamics in a neighborhood of a harmonic map was studied by Krieger, Schlag and Tataru \cite{KrScTa08}, who constructed blow-up solutions in the energy space with the concentration rate $\lambda(t) \sim (T_+ - t)^{1+\nu}$ for all $\nu > \frac 12$. This behavior is expected to be highly unstable.
Rodnianski and Sterbenz \cite{RoSt10} constructed stable blow-up solutions, giving the first (partial) rigorous explanation of
the surprising numerical results mentioned above.
In the case $k = 1$, C\^ote \cite{Cote15} proved that any solution decomposes, for a sequence of times tending to the final (finite or inifinite) time of existence,
as a sum of a finite number of harmonic maps at different scales and a radiation term. A generalization of this result, including all the cases considered in this paper,
was recently obtained by Jia and Kenig~\cite{JiKe15p}. Motivated by these works, we prove the following result.
\begin{mainthm}
  \label{thm:deux-bulles-wmap}
Fix $k > 2$. There exists a solution $\bs u: (-\infty, T_0] \to \cE$ of \eqref{eq:wmap0} such that
  \begin{equation}
    \label{eq:mainthm-wmap}
    \lim_{t\to -\infty}\|\bs u(t) - (-\bs W + \bs W_{\frac{k-2}{2\kappa}(\kappa|t|)^{-\frac{2}{k-2}}})\|_\cE = 0, \qquad \text{with }\kappa := \frac{k-2}{2}\big(\frac{8k}{\pi}\sin\big(\frac{\pi}{k}\big)\big)^\frac 1k.
  \end{equation}
\end{mainthm}
\begin{remark}
  More precisely, we will prove that
  \begin{equation*}
    \big\|\bs u(t) - \big({-}\bs W + \bs W_{\frac{k-2}{2\kappa}(\kappa|t|)^{-\frac{2}{k-2}}}\big)\big\|_\cE \leq C_1 \cdot |t|^{-\frac{1}{2(k-2)}}.
  \end{equation*}
\end{remark}
\begin{remark}
  The constructed solution is a \emph{pure} two-bubble, hence by the conservation of energy $E(\bs u(t)) = 2E(\bs W)$,
  and it is clear that it has the homotopy degree $0$.
  In the case of equivariant class $k = 1$, C\^ote, Kenig, Lawrie and Schlag \cite{CKLS15} showed that any degree $0$ initial data of energy $<2E(\bs W)$ leads to dispersion
  (the proof is expected to generalize to all equivariance classes).
  Theorem~\ref{thm:deux-bulles-ym} gives the first example of a non-dispersive solution at the threshold energy.

  Note that pure two-bubbles of homotopy degree $2k$ (hence of type bubble-bubble and not bubble-antibubble)
  do not exist because the energy of such a map has to be $> 2E(\bs W)$.
  This is similar to the case of opposite signs for \eqref{eq:nlw0}, see Remark~\ref{rem:nlw-signs}.
\end{remark}
\begin{remark}
  \label{rem:general-wmap}
  I believe that the proof can be adapted to deal with a more general equation $\partial_t^2 u = \partial_r^2 u + \frac 1r \partial_r u - \frac{1}{r^2}(g g')(u)$ with
  $g$ satisfying the assumptions of \cite{CKM08} and $g'(0) \in\{3, 4, 5, \ldots\}$.
\end{remark}

\subsection{Critical Yang-Mills}
Finally, we consider the radial Yang-Mills equation in dimension 4 (which is the energy-critical case):
\begin{equation}
\label{eq:ym0}
  \left\{
    \begin{aligned}
      &\partial_t^2 u(t, r) = \partial_r^2 u(t, r) +\frac{1}{r}\partial_r u(t, r) - \frac{4}{r^2}u(t, r)(1-u(t, r))\big(1-\frac 12 u(t, r)\big), \\
      &(u(t_0, r), \partial_t u(t_0, r)) = (u_0(r), \dot u_0(r)), \qquad t, t_0 \in \bR,\ r\in (0, +\infty).
    \end{aligned}\right.
\end{equation}
For a derivation of this equation and further comments, see for instance \cite{CST98}.
Equation \eqref{eq:ym0} can be written in the form \eqref{eq:nlw} with the energy functional $E$ defined for $\bs u_0 = (u_0, \dot u_0) \in \cE$ by the formula
\begin{equation}
  \label{eq:energy-ym}
  E(\bs u_0) := \pi \int_0^{+\infty}\big((\dot u_0)^2 + (\partial_r u_0)^2 + \frac{1}{r^2}(u_0(2-u_0))^2\big)\udr.
\end{equation}

The stationary solutions of \eqref{eq:ym0} are $W_\lambda(r) := \frac{2 r^2}{\lambda^2 + r^2}$. We denote $W(r) := W_1(r) = \frac{2r^2}{1+r^2}$
and $\Lambda W(r) := -\pd \lambda W_\lambda\rstr_{\lambda = 1} = \frac{4}{(r+r^{-1})^2}$.
\begin{mainthm}
  \label{thm:deux-bulles-ym}
There exists a solution $\bs u: (-\infty, T_0] \to \cE$ of \eqref{eq:ym0} such that
  \begin{equation}
    \label{eq:mainthm-ym}
    \lim_{t\to -\infty}\|\bs u(t) - (-\bs W + \bs W_{\frac{1}{\kappa}\eee^{-\kappa|t|}})\|_\cE = 0, \qquad \text{with }\kappa := 2\sqrt{3}.
  \end{equation}
\end{mainthm}
\begin{remark}
  More precisely, we will prove that
  \begin{equation*}
    \big\|\bs u(t) - \big({-}\bs W + \bs W_{\frac{1}{\kappa}\eee^{-\kappa|t|}}\big)\big\|_\cE \leq C_1 \cdot \eee^{-\frac 12 \kappa|t|},
  \end{equation*}
  where $\Lambda W := - \pd \lambda W_\lambda\rstr_{\lambda = 1}$ and $C_1 > 0$ is a constant.
\end{remark}
\begin{remark}
  The case of wave maps in the equivariance class $k = 2$ should be very similar.
\end{remark}
\begin{remark}
The energy $2E(\bs W)$ is the threshold energy for a non-dispersive behavior for solutions with topological degree $0$, see \cite{LaOh16}.
\end{remark}
\subsection{Structure of the paper}
In Sections~\ref{sec:approx} and \ref{sec:bootstrap} we give a detailed proof of Theorem~\ref{thm:deux-bulles}.
In Section~\ref{sec:ym} we treat the case of the Yang-Mills equation. We skip these parts of the proof where the arguments of Sections~\ref{sec:approx} and \ref{sec:bootstrap}
are directly applicable. Section~\ref{sec:wmap} is devoted to the wave maps equation. The main difference with respect to Section~\ref{sec:ym}
is that the characteristic length of the concentrating bubble is now a power of $|t|$ and not an exponential. Nevertheless, large parts of the previous proofs extend to this case and are skipped.
It is conceivable that one could propose a unified, more general framework of the proof, encompassing all the cases under consideration.
Appendix~\ref{sec:cauchy} is devoted to some elements of the local Cauchy theory needed in the proofs.



\subsection{Notation}
\label{ssec:notation}
The bracket $\la\cdot, \cdot\ra$ denotes the distributional pairing and the scalar product in the spaces $L^2$ and $L^2 \times L^2$.

For positive quantities $m_1$ and $m_2$ we write $m_1 \lesssim m_2$ if $m_1 \leq Cm_2$ for some constant $C > 0$ and $m_1 \sim m_2$ if $m_1 \lesssim m_2 \lesssim m_1$.

We denote $\chi$ a standard $C^\infty$ cut-off function, that is $\chi(x) = 1$ for $|x| \leq 1$, $\chi(x) = 0$ for $|x| \geq 2$ and $0 \leq \chi(x) \leq 1$ for $1\leq |x|\leq 2$.

\section{Construction of an approximate solution -- the NLW case}
\label{sec:approx}
\subsection{Inverting the linearized operator}
\label{ssec:linearise}
Linearizing \eqref{eq:nlw0} around $\bs W$, $\bs u = \bs W + \bs h$, one obtains
\begin{equation*}
  \partial_t \bs h = J\circ\vD^2 E(\bs W)\bs h = \begin{pmatrix} 0 & \Id \\ -L & 0\end{pmatrix} \bs h,
\end{equation*}
where $L$ is the Schr\"odinger operator
$$Lh := (-\Delta - f'(W))h = (-\Delta - 2W)h.$$
We introduce the following notation for the generators of the $\dot H^1$-critical and the $L^2$-critical scale change:
\begin{equation}
\Lambda := 2+x\cdot \grad,\qquad \Lambda_0 := 3 + x\cdot \grad.
\end{equation}
This is coherent with the definition of $\Lambda W$.
Notice that $L(\Lambda W) = \pd\lambda\rstr_{\lambda = 1}\big({-}\Delta W_\lambda - f(W_\lambda)\big) = 0$.

We fix $\cZ\in C_0^\infty$ such that
\begin{equation}
  \label{eq:Z}
  \la \cZ, \Lambda W\ra > 0, \qquad \la \cZ, \cY\ra = 0.
\end{equation}
We will use this function to define appropriate orthogonality conditions.

We denote also
\begin{equation}
  \label{eq:kappa}
  \kappa := \Big(-\frac{\la \Lambda W, f'(W)\ra}{\la \Lambda W, \Lambda W\ra}\Big)^{\frac 12} = \sqrt{\frac 54}.
\end{equation}
\begin{lemma}
  \label{lem:antecedent}
  There exist radial rational functions $P(x), Q(x) \in C^\infty(\bR^6)$ such that
  \begin{gather}
    LP = \kappa^2 \Lambda W + f'(W), \qquad LQ = -\Lambda_0 \Lambda W, \label{eq:profil-eq} \\
    \la \cZ, P\ra = \la \cZ, Q\ra = 0, \label{eq:profil-orth} \\
    P(x) \sim |x|^{-2}, \qquad Q(x)\sim |x|^{-2}\qquad\text{as }|x| \to +\infty. \label{eq:profil-asympt}
  \end{gather}
  \end{lemma}
  \begin{proof}
  By a direct computation one checks that the functions
  \begin{align}
    \wt P(x) &:= \big(1+\frac{|x|^2}{24}\big)^{-3}\cdot\big(1-10\cdot\frac{|x|^2}{24}-3\cdot\big(\frac{|x|^2}{24}\big)^2\big),  \label{eq:A} \\
    \wt Q(x) &:= \big(1+\frac{|x|^2}{24}\big)^{-3}\cdot\big(1+11\cdot\frac{|x|^2}{24}-12\cdot\big(\frac{|x|^2}{24}\big)^2\big)  \label{eq:B}
  \end{align}
  satisfy \eqref{eq:profil-eq}. Adding suitable multiples of $\Lambda W$ to both functions we obtain $P$ and $Q$ satisfying \eqref{eq:profil-orth}.
  The formulas defining $\wt P$ and $\wt Q$ directly imply \eqref{eq:profil-asympt}.
\end{proof}
\begin{remark}
  Note that \eqref{eq:profil-asympt} is closely related to the Fredholm conditions $\la \Lambda W, \kappa^2 \Lambda W + f'(W)\ra = 0$ and $\la \Lambda W, -\Lambda_0 \Lambda W\ra = 0$,
  see Lemma~\ref{lem:fredholm-wmap} or \cite[Proposition 2.1]{moi15p-1} for a more systematic presentation.
\end{remark}

\subsection{Formal computation}
\label{ssec:formal}

The usual method of performing a formal analysis of blow-up solutions is to search a series expansion with respect to a small scalar parameter
depending on time and converging to $0$ at blow-up. In our case the blow-up time is $-\infty$.
If $u(t) \simeq W + W_{\lambda(t)}$, then $\partial_t u(t) \simeq -\lambda'(t) \Lambda W_\uln{\lambda(t)}$, hence
\begin{equation}
  \bs u(t) \simeq (W + W_{\lambda(t)}, 0) - \lambda'(t)\cdot(0, \Lambda W_\uln{\lambda(t)}) = \bs W + \bs U^{(0)}_{\lambda(t)} + b(t)\cdot \bs U^{(1)}_{\lambda(t)},
\end{equation}
with $b(t) := \lambda'(t)$, $\bs U^{(0)} := (W, 0)$ and $\bs U^{(1)} := (0, -\Lambda W)$.
This suggests considering $b(t) = \lambda'(t)$ as the small parameter with respect to which
the formal expansion should be sought. Hence, we make the ansatz
\begin{equation}
  \bs u(t) = \bs W + \bs U^{(0)}_{\lambda(t)} + b(t)\cdot \bs U^{(1)}_{\lambda(t)} + b(t)^2 \cdot\bs U^{(2)}_{\lambda(t)},
\end{equation}
and try to find the conditions under which a satisfactory candidate for $\bs U^{(2)} = (U^{(2)}, \dot U^{(2)})$ can be proposed.
Neglecting irrelevant terms and replacing $\lambda'(t)$ by $b(t)$, we compute
\begin{equation*}
  \partial_{t}^2u(t) = -b'(t)(\Lambda W)_{\uln{\lambda(t)}} + \frac{b(t)^2}{\lambda(t)}(\Lambda_0\Lambda W)_{\uln{\lambda(t)}} + \text{lot}.
\end{equation*}
  On the other hand, using the fact that $f(W + W_\lambda) = f(W) + f(W_\lambda) + f'(W_\lambda)W \simeq f(W) + f(W_\lambda) + f'(W_\lambda)$ for $\lambda \ll 1$
  and $f'(W_\lambda) = \lambda f'(W)_\uln\lambda$, we get
  \begin{equation*}
    \Delta u(t) + f(u(t)) = -\frac{b(t)^2}{\lambda(t)}(LU^{(2)})_{\uln{\lambda(t)}} + \lambda(t)f'(W)_\uln{\lambda(t)} + \text{lot}.
  \end{equation*}
  We discover that, formally at least, we should have
  \begin{equation}
    \label{eq:profilT}
    LU^{(2)} = -\Lambda_0\Lambda W + \frac{\lambda(t)}{b(t)^2}\big(b'(t)\cdot \Lambda W + \lambda(t)\cdot f'(W)\big).
  \end{equation}
  Lemma~\ref{lem:antecedent} shows that if $b'(t) = \kappa^2 \lambda(t)$,
  then equation \eqref{eq:profilT} has a decaying regular solution $U^{(2)} = Q + \frac{\lambda(t)^{2}}{b(t)^2}P$.
  The \emph{formal parameter equations}
  \begin{equation}
    \label{eq:formal-param}
    \lambda'(t) = b(t), \qquad b'(t) = \kappa^2 \lambda(t)
  \end{equation}
    have a solution
  \begin{equation}
    \label{eq:param-sol}
    (\lambda_\tx{app}(t), b_\tx{app}(t)) = \big(\frac{1}{\kappa}\eee^{-\kappa|t|}, \eee^{-\kappa|t|}\big),\qquad t \leq T_0 < 0.
  \end{equation}

  In any space dimension $N$, ignoring the problems related to slow decay of $W$, a similar analysis would yield $b'(t) = \kappa^2 \lambda(t)^\frac{N-4}{2}$.
  For $N < 6$ this leads to a finite time blow-up, which was studied in \cite{moi15p-1} for $N = 5$.
  For $N > 6$, we obtain a global solution $\lambda(t) \sim |t|^{-\frac{4}{N-6}}$, see Remark~\ref{rem:dim7}.
  
\subsection{Bounds on the error of the ansatz}
\label{ssec:error}
Let $I = [T, T_0]$ be a time interval, with $T \leq T_0 < 0$ and $|T_0|$ large.
Let $\lambda(t)$ and $\mu(t)$ be  $C^1$ functions on $[T, T_0]$ such that
\begin{gather}
  \lambda(T) = \frac{1}{\kappa} \eee^{-\kappa|T|},\quad \mu(T) = 1, \label{eq:lambda-init0} \\
  \frac{8}{9\kappa}\eee^{-\kappa|t|}\leq \lambda(t) \leq \frac{9}{8\kappa}\eee^{-\kappa|t|},\qquad \frac 89 \leq \mu(t)\leq \frac 98. \label{eq:lambda-bound0}
\end{gather}
We define the \emph{approximate solution} $\bs \varphi(t) = (\varphi(t), \dot \varphi(t)): [T, T_0] \to \cE$ by the formula
\begin{equation}
  \label{eq:phi-def}
  \begin{aligned}
    \varphi(t) &:= W_{\mu(t)} + W_{\lambda(t)} + S(t), \\
    \dot \varphi(t) &:= -b(t)\Lambda W_\uln{\lambda(t)},
\end{aligned}
\end{equation}
where
\begin{align}
  \label{eq:b-def}
  b(t) &:= \eee^{-\kappa|T|} + \kappa^2 \int_T^t \frac{\lambda(\tau)}{\mu(\tau)^2}\ud \tau,\qquad&\text{for }t \in [T, T_0], \\
  S(t) &:= \chi\cdot \Big(\frac{\lambda(t)^2}{\mu(t)^2} P_{\lambda(t)} + b(t)^2 Q_{\lambda(t)}\Big),\qquad&\text{for }t\in [T, T_0]. \label{eq:S-def}
\end{align}
From \eqref{eq:lambda-bound0} we get $\big(\frac 89\big)^3 \frac{1}{\kappa}\eee^{-\kappa|t|} \leq \frac{\lambda(t)}{\mu(t)^2} \leq \big(\frac 98\big)^3 \frac{1}{\kappa}\eee^{-\kappa|t|}$.
Integrating we get the following bound for $b(t)$, $t\in[T, T_0]$:
\begin{equation}
  \label{eq:b-bound0}
  \begin{aligned}
  \Big(\frac 89\Big)^3 \eee^{-\kappa|t|} &< \eee^{-\kappa|T|} +\Big(\frac 89\Big)^3 (\eee^{-\kappa|t|}-\eee^{-\kappa|T|}) \\ &
  \leq b(t) \leq \eee^{-\kappa|T|}+ \Big(\frac 98\Big)^3(\eee^{-\kappa|t|}-\eee^{-\kappa|T|}) < \Big(\frac 98\Big)^3\eee^{-\kappa|t|}.
\end{aligned}
\end{equation}
From \eqref{eq:profil-asympt} we obtain
\begin{equation}
  \label{eq:taille-P}
  \begin{aligned}
    &\|\chi\cdot P_\uln\lambda\|_{\dot H^1} \simeq \|\partial_r(\chi\cdot P_\uln\lambda)\|_{L^2(r^5\ud r)} \lesssim \|P_\uln\lambda\|_{L^2(0\leq r\leq 2)} + \|\frac{1}{\lambda}(\partial_r P)_\uln\lambda\|_{L^2(0\leq r\leq 2)} \\
    &= \|P\|_{L^2(0 \leq r\leq \frac{2}{\lambda})} + \frac{1}{\lambda}\|\partial_r P\|_{L^2(0\leq r\leq \frac{2}{\lambda})} \\
    &\lesssim \Big(\int_0^{2/\lambda} (1+r^2)^{-2} r^5\ud r\Big)^\frac 12+ \frac{1}{\lambda}\cdot \Big(\int_0^{2/\lambda} (1+r^3)^{-2} r^5\ud r\Big)^\frac 12
    \lesssim \frac{1}{\lambda}|\log \lambda|^\frac 12,
\end{aligned}
\end{equation}
and analogously $\|\chi\cdot Q_\uln\lambda\|_{\dot H^1} \lesssim \frac{1}{\lambda}|\log \lambda|^\frac 12$, hence
$$
\|S(t)\|_{\dot H^1} \leq \frac{\lambda^3}{\mu^2}\|\chi P_\uln\lambda\|_{\dot H^1} + \lambda b^2 \|\chi Q_\uln\lambda\|_{\dot H^1} \lesssim \eee^{-3\kappa|t|}\cdot \frac{1}{\lambda}|\log \lambda|^\frac 12\lesssim \sqrt{|t|}\cdot \eee^{-2\kappa|t|}.
$$
Thus for any $c > 0$ there exists $T_0$ such that if $T < T_0$ then
\begin{equation}
  \label{eq:taille-correct}
  \|S(t)\|_{\dot H^1} \leq c\cdot \eee^{-\frac 32 \kappa|t|}, \qquad \text{for }t \in [T, T_0].
\end{equation}
Note also that $\|P_\lambda\|_{L^\infty} + \|Q_\lambda\|_{L^\infty} \lesssim \lambda^{-2}$, hence $S(t)$ is bounded in $L^\infty$.

Since $\cZ$ has compact support, for sufficiently small $\lambda$, \eqref{eq:profil-orth} implies
\begin{equation}
  \label{eq:S-orth}
  \la \cZ_\uln{\lambda(t)}, S(t)\ra = 0.
\end{equation}

We denote
\begin{equation}
  \label{eq:psi}
\begin{aligned}
  \bs \psi(t) &= (\psi(t), \dot \psi(t)) := \partial_t  \bs \varphi(t) - \vD E(\bs \varphi(t)) \\ &= \big(\partial_t \varphi(t) - \dot \varphi(t), \partial_t \dot \varphi(t) - (\Delta \varphi(t) + f(\varphi(t)))\big).
\end{aligned}
\end{equation}
This function describes how much $\bs \varphi(t)$ fails to be an exact solution of \eqref{eq:nlw0}. Before we prove bounds on $\bs \psi(t)$,
we gather in the next elementary lemma pointwise inequalities used in various places in the text.
\begin{lemma}
  \label{lem:pointwise}
  Let $k, l, m \in \bR$. Then
  \begin{gather}
    |f'(k+l) - f'(k)| \leq f'(l), \label{eq:point-fp2} \\
    |f(k+l) - f(k) - f'(k)l| \leq 5|f(l)|, \label{eq:point-f2} \\
    |F(k+l) - F(k) - f(k)l - \frac 12 f'(k)l^2| \leq 5 F(l). \label{eq:point-F2}
  \end{gather}
\end{lemma}
\begin{proof}
  Inequality \eqref{eq:point-fp2} is well-known.
  Bounds \eqref{eq:point-f2} holds for $k = 0$, hence (by homogeneity) we may assume that $k = 1$. For $|l| \leq 1$ we have
  $|f(1+l) - f(1) - f'(1)l| = |(1+l)^2 - 1 - 2l| = 2l^2 \leq 5|f(l)|$ and for $|l|\geq 1$ we find
  $|f(1+l) - f(1) - f'(1)l| \leq (1+l)^2 + 1 + 2|l| \leq 5|f(l)|$.
  Bound \eqref{eq:point-F2} follows by integrating \eqref{eq:point-f2}.
\end{proof}
\begin{lemma}
  \label{lem:psi}
  Suppose that for $t \in [T, T_0]$ there holds $|\lambda'(t)| \lesssim \eee^{-\kappa |t|}$ and $|\mu'(t)| \lesssim \eee^{-\kappa|t|}$. Then
  \begin{align}
    \|\psi(t) + \mu'(t)\frac{1}{\mu(t)}\Lambda W_{\mu(t)}+ (\lambda'(t) - b(t))\frac{1}{\lambda(t)}\Lambda W_{\lambda(t)} \|_{\dot H^1} &\lesssim \eee^{-\frac 32 \kappa |t|}
    , \label{eq:psi0} \\
    \|\dot \psi(t) - \frac{b(t)}{\lambda(t)}(\lambda'(t) - b(t))\Lambda_0 \Lambda W_\uln{\lambda(t)}\|_{L^2} &\lesssim \eee^{-\frac 32 \kappa|t|}, \label{eq:psi1} \\
    \|(-\Delta - f'(\varphi(t)))\psi(t)\|_{\dot H^{-1}} &\lesssim \eee^{-\frac 32 \kappa|t|}. \label{eq:psi-lin}
  \end{align}
\end{lemma}
\begin{proof}
  Using the definitions of $\bs \varphi$ and $\bs \psi$ we find
\begin{equation*}
  \begin{aligned}
    &\psi + \mu'\Lambda W_\uln{\mu} +(\lambda' - b)\Lambda W_\uln{\lambda} = 
    \partial_t \varphi - \dot \varphi + \mu'\Lambda W_\uln\mu + (\lambda' - b)\Lambda W_\uln\lambda \\
    &= -\mu' \Lambda W_\uln\mu - \lambda'\Lambda W_\uln\lambda + \partial_t S + b\Lambda W_\uln\lambda + (\lambda' - b)\Lambda W_\uln\lambda \\
    &= \chi\cdot \big(-2\mu'\frac{\lambda^3}{\mu^3}P_\uln\lambda + 2\lambda'\frac{\lambda^2}{\mu^2} P_\uln\lambda -\lambda'\frac{\lambda^2}{\mu^2}\Lambda P_\uln\lambda + 2b'b\lambda Q_\uln\lambda -\lambda'b^2 \Lambda Q_\uln\lambda\big).
  \end{aligned}
\end{equation*}
Since $\Lambda P$ and $\Lambda Q$ are rational functions decaying like $r^{-2}$, we have $\|\chi\cdot \Lambda P_\uln\lambda\|_{\dot H^1} \lesssim \sqrt{|t|}\cdot \eee^{\kappa|t|}$ and
$\|\chi\cdot \Lambda Q_\uln\lambda\|_{\dot H^1} \lesssim \sqrt{|t|}\cdot \eee^{\kappa|t|}$, see \eqref{eq:taille-P}.
This implies \eqref{eq:psi0} because $|\lambda|, |b|, |\lambda'|, |b'|, |\mu'| \lesssim \eee^{-\kappa|t|}$.

In order to prove \eqref{eq:psi1}, we consider separately the regions $|x| \leq \sqrt\lambda$ and $|x| \geq \sqrt\lambda$.
The first step is to treat the nonlinearity, that is to show that
\begin{equation}
  \label{eq:psi1-nonlin}
\begin{aligned}
  \|f(\varphi) &- f(W_\mu) - f(W_\lambda) - f'(W_\lambda)W_\mu \\&-\frac{\lambda^2}{\mu^2}f'(W_\lambda)P_\lambda - b^2 f'(W_\lambda)Q_\lambda\|_{L^2(|x|\leq \sqrt\lambda)} \lesssim \eee^{-\frac 32 \kappa|t|}.
\end{aligned}
\end{equation}
Applying \eqref{eq:point-f2} with $k= W_\lambda$ and $l = W_\mu + S$ we get
\begin{equation*}
  \big|f(\varphi) - f(W_\lambda) - f'(W_\lambda)(W_\mu + \lambda^2\mu^{-2} \cdot P_\lambda + b^2 \cdot Q_\lambda)\big| \lesssim |f(W_\mu)| + |f(\lambda^2 \mu^{-2}P_\lambda + b^2 Q_\lambda)|,
\end{equation*}
which is bounded in $L^\infty$, hence bounded by $\lambda^\frac 32 \sim \eee^{-\frac 32 \kappa|t|}$ in $L^2(|x| \leq \sqrt\lambda)$. This proves \eqref{eq:psi1-nonlin}.
Now we check that
\begin{equation}
  \label{eq:WmWl}
  \big\|f'(W_\lambda)W_\mu - \frac{1}{\mu^2}f'(W_\lambda)\big\|_{L^2} \lesssim \eee^{-\frac 32 \kappa|t|}.
\end{equation}
Indeed, for $|x| \leq \sqrt\lambda$ we have $|W_\mu(x) - \frac{1}{\mu^2}| = |W_\mu(x) - W_\mu(0)| \lesssim |x|^2 \leq \lambda \sim \eee^{-\kappa|t|}$, hence
\begin{equation*}
  \big\|f'(W_\lambda)W_\mu - \frac{1}{\mu^2}f'(W_\lambda)\big\|_{L^2(|x|\leq \sqrt\lambda)} \lesssim \big\|W_\mu - \frac{1}{\mu^2}\big\|_{L^\infty(|x| \leq \sqrt\lambda)}\cdot \|W_\lambda\|_{L^2} \lesssim \eee^{-\kappa|t|}\cdot \lambda \ll \eee^{-\frac 32 \kappa|t|}.
\end{equation*}
From \eqref{eq:psi1-nonlin} and \eqref{eq:WmWl} we obtain
\begin{equation}
  \label{eq:psi1-nonlin-fin}
  \|f(\varphi) - f(W_\mu) - f(W_\lambda) - \mu^{-2}f'(W_\lambda)-\lambda^2\mu^{-2}f'(W_\lambda)P_\lambda - b^2 f'(W_\lambda)Q_\lambda\|_{L^2} \lesssim \eee^{-\frac 32 \kappa|t|}.
\end{equation}
Since $\chi = 1$ in the region $|x| \leq \sqrt\lambda$, we have $\Delta \varphi = \Delta(W_\mu) + \Delta(W_\lambda) + \lambda^2\mu^{-2} \Delta(P_\lambda) + b^2 \Delta(Q_\lambda)$.
From this and \eqref{eq:psi1-nonlin-fin}, using the fact that $\Delta(W_\mu) + f(W_\mu) = \Delta(W_\lambda) + f(W_\lambda) = 0$, we get
\begin{equation}
  \label{eq:psi1-total-1}
  \begin{aligned}
    \big\|\Delta\varphi + f(\varphi) &- \mu^{-2}\big(\lambda^2 \Delta(P_\lambda) + \lambda^2 f'(W_\lambda)P_\lambda + f'(W_\lambda)\big) \\
    &- \big(b^2 \Delta(Q_\lambda) + b^2 f'(W_\lambda)Q_\lambda\big)\big\|_{L^2(|x| \leq \sqrt\lambda)} \lesssim \eee^{-\frac 32 \kappa|t|}.
  \end{aligned}
\end{equation}
But formula \eqref{eq:profil-eq} gives
\begin{equation}
  \begin{gathered}
  \lambda^2 \Delta(P_\lambda) + \lambda^2 f'(W_\lambda)P_\lambda = (-LP)_\lambda = -\kappa^2 \Lambda W_\lambda - f'(W_\lambda) = -\kappa^2 \lambda \Lambda W_\uln\lambda - f'(W_\lambda), \\
  b^2 \Delta(Q_\lambda) + b^2 f'(W_\lambda)Q_\lambda = \frac{b^2}{\lambda^2}(-LQ)_\lambda = \frac{b^2}{\lambda}\Lambda_0 \Lambda W_\uln\lambda,
\end{gathered}
\end{equation}
hence we can rewrite \eqref{eq:psi1-total-1} as
\begin{equation}
  \label{eq:psi1-total-2}
  \big\|\Delta \varphi + f(\varphi) + \frac{\kappa^2\lambda}{\mu^2}\Lambda W_\uln\lambda - \frac{b^2}{\lambda}\Lambda_0\Lambda W_\uln\lambda\big\|_{L^2(|x|\leq \sqrt\lambda)} \lesssim \eee^{-\frac 32 \kappa|t|}.
\end{equation}
We have $\partial_t \dot \varphi = -b'\Lambda W_\uln\lambda + \frac{\lambda'b}{\lambda}\Lambda_0 \Lambda W_\uln\lambda$, thus
\begin{equation}
  \label{eq:dot-psi}
  \begin{aligned}
  -\dot \psi + \frac{b}{\lambda}(\lambda'-b)\Lambda_0 \Lambda W_\uln\lambda &= \Delta \varphi + f(\varphi) + b'\Lambda W_\uln\lambda - \frac{b\lambda'}{\lambda}\Lambda_0 \Lambda W_\uln\lambda
  +\frac{b}{\lambda}(\lambda'-b)\Lambda_0 \Lambda W_\uln\lambda \\
  &= \Delta \varphi + f(\varphi) + \frac{\kappa^2 \lambda}{\mu^2}\Lambda W_\uln\lambda - \frac{b^2}{\lambda}\Lambda_0 \Lambda W_\uln\lambda,
\end{aligned}
\end{equation}
so \eqref{eq:psi1-total-2} yields \eqref{eq:psi1} in the region $|x| \leq \sqrt\lambda$.

Consider now the region $|x| \geq \sqrt\lambda$.
First we show that
\begin{equation}
  \label{eq:psi1-lapl}
  \|\Delta\varphi - \Delta(W_\mu)\|_{L^2(|x|\geq \sqrt\lambda)} \lesssim \eee^{-\frac 32 \kappa|t|}.
\end{equation}
To this end, we compute
\begin{equation}
  \label{eq:Wl-ext}
  \begin{aligned}
  \|\Delta(W_\lambda)\|_{L^2(|x| \geq \sqrt\lambda)} &= \|f(W_\lambda)\|_{L^2(|x| \geq \sqrt\lambda)} = \frac{1}{\lambda}\|f(W)\|_{L^2(|x| \geq 1/\sqrt\lambda)} \\
  &\lesssim \frac{1}{\lambda}\Big(\int_{1/\sqrt\lambda}^{+\infty}r^{-16}r^5\ud r\Big)^\frac 12 \lesssim \frac{1}{\lambda}\lambda^\frac 52 \lesssim \eee^{-\frac 32 \kappa|t|}.
\end{aligned}
\end{equation}
We need to show that $\|\Delta(\chi\cdot P_\lambda)\|_{L^2(|x|\geq \sqrt\lambda)} \lesssim \eee^{\frac 12 \kappa|t|}$
and $\|\Delta(\chi\cdot Q_\lambda)\|_{L^2(|x|\geq \sqrt\lambda)} \lesssim \eee^{\frac 12 \kappa|t|}$.
We will prove the first bound (the second is exactly the same). Notice that $|P_\lambda(x)| \lesssim \frac{1}{\lambda^2}\cdot\frac{\lambda^2}{|x|^2} = |x|^{-2}$
and similarly $|\grad(P_\lambda)(x)| \lesssim |x|^{-3}$, $|\grad ^2(P_\lambda)(x)| \lesssim |x|^{-4}$, hence we have a pointwise bound
\begin{equation}
  |\Delta(\chi P_\lambda)| \lesssim |\grad^2 \chi|\cdot|P_\lambda| + |\grad \chi|\cdot|\grad(P_\lambda)| + |\chi|\cdot |\grad^2 (P_\lambda)| \lesssim |\grad^2 \chi|\cdot|x|^{-2} + |\grad \chi|\cdot|x|^{-3} + |\chi|\cdot|x|^{-4}.
\end{equation}
Of course $\big\||\grad^2 \chi|\cdot|x|^{-2}\big\|_{L^2(|x|\geq \sqrt\lambda)} + \big\||\grad \chi|\cdot|x|^{-3}\big\|_{L^2(|x|\geq \sqrt\lambda)} \lesssim 1 \ll \eee^{\frac 12 \kappa|t|}$ and we are left with the last term. We compute
\begin{equation}
  \big\||\chi|\cdot |x|^{-4}\|_{L^2(|x|\geq \sqrt\lambda)} \lesssim \Big(\int_{\sqrt\lambda}^2 r^{-8}r^5\ud r\Big)^\frac 12 \lesssim \lambda^{-\frac 12} \lesssim \eee^{\frac 12 \kappa|t|}.
\end{equation}
This finishes the proof of \eqref{eq:psi1-lapl}.

Applying \eqref{eq:point-f2} with $k = W_\mu$ and $l = W_\lambda + S$ we get
\begin{equation}
|f(\varphi) - f(W_\mu)| \lesssim f'(W_\mu)\cdot (|W_\lambda| + |S|) + |f(W_\lambda)| + |f(S)| \lesssim |W_\lambda| + |S|,
\end{equation}
where the last estimate follows from the fact that $\|W_\lambda\|_{L^\infty} + \|S\|_{L^\infty} \lesssim 1$ for $|x| \geq \sqrt\lambda$.

We have $\|\chi\cdot P_\uln\lambda\|_{L^2} \lesssim \Big(\int_0^\frac{2}{\lambda}(r^{-2})^2r^5\ud r\Big)^\frac 12 \sim \lambda^{-1}$,
and similarly $\|\chi\cdot Q_\uln\lambda\|_{L^2} \lesssim \lambda^{-1}$, which implies $\|S\|_{L^2} \ll \eee^{-\frac 32 \kappa|t|}$.
There holds also
\begin{equation}
  \label{eq:W-ext}
\|W_\lambda\|_{L^2(|x|\geq \sqrt\lambda)} = \lambda\|W\|_{L^2(|x|\geq 1/\sqrt\lambda)} \lesssim \lambda\Big(\int_{1/\sqrt\lambda}^\infty r^{-8}r^5\ud r\Big)^\frac 12 \sim \lambda^\frac 32 \sim \eee^{-\frac 32 \kappa|t|},
\end{equation}
hence $\|f(\varphi) - f(\mu)\|_{L^2(|x| \geq \sqrt\lambda)} \lesssim \eee^{-\frac 32 \kappa|t|}$. Together with \eqref{eq:psi1-lapl} this yields
$$\|\Delta \varphi + f(\varphi)\|_{L^2(|x| \geq \sqrt\lambda)} \lesssim \eee^{-\frac 32 \kappa|t|}.$$
The same computation as in \eqref{eq:W-ext} gives
$$\|\Lambda W_\uln\lambda\|_{L^2(|x| \geq \sqrt\lambda)} + \|\Lambda_0 \Lambda W_\uln\lambda\|_{L^2(|x|\geq \sqrt\lambda)} \lesssim \eee^{-\frac 12 \kappa|t|},$$
hence \eqref{eq:dot-psi} implies that \eqref{eq:psi1} holds also in the region $|x| \geq \sqrt\lambda$.

We are left with \eqref{eq:psi-lin}. From \eqref{eq:psi0} it follows that it suffices to check that
\begin{equation}
  \label{eq:psi-lin-2}
  \Big\|(-\Delta - f'(\varphi))(\lambda' - b)\frac{1}{\lambda}\Lambda W_\lambda\Big\|_{\dot H^{-1}} \lesssim \eee^{-\frac 32 \kappa|t|}
\end{equation}
and
\begin{equation}
  \label{eq:psi-lin-3}
  \Big\|(-\Delta - f'(\varphi))\mu'\frac{1}{\mu}\Lambda W_\mu\Big\|_{\dot H^{-1}} \lesssim \eee^{-\frac 32 \kappa|t|}.
\end{equation}
We start with \eqref{eq:psi-lin-2}. Since $|\lambda' - b| \lesssim \eee^{-\kappa|t|} \lesssim \lambda$, we need to show that
\begin{equation}
  \label{eq:psi-lin-4}
  \|(-\Delta - f'(\varphi))\Lambda W_\lambda\|_{\dot H^{-1}} \lesssim \eee^{-\frac 32 \kappa|t|}.
\end{equation}
By H\"older inequality
\begin{equation}
  \label{eq:potential-holder}
\|f'(W_\mu)\Lambda W_\lambda\|_{L^\frac 32} \leq \|f'(W_\mu)\|_{L^{12}}\cdot \|\Lambda W_\lambda\|_{L^\frac{12}{7}} \lesssim \lambda^\frac 32 \lesssim \eee^{-\frac 32\kappa|t|}.
\end{equation}
Since $|f'(W_\lambda + W_\mu) - f'(W_\lambda)| \leq f'(W_\mu)$, we obtain
$$
\|(f'(W_\lambda + W_\mu) - f'(W_\lambda))\Lambda W_\lambda\|_{\dot H^{-1}} \lesssim \|(f'(W_\lambda + W_\mu) - f'(W_\lambda))\Lambda W_\lambda\|_{L^{\frac 32}} \lesssim \eee^{-\frac 32\kappa|t|}.
$$
As noted earlier $(-\Delta - f'(W_\lambda))\Lambda W_\lambda = 0$, hence
$$
\|(-\Delta - f'(W_\lambda + W_\mu))\Lambda W_\lambda\|_{\dot H^{-1}} \lesssim \eee^{-\frac 32 \kappa|t|}.
$$
From \eqref{eq:taille-correct} we have $\|f'(\varphi) - f'(W_\lambda + W_\mu)\|_{L^3} \lesssim \eee^{-\frac 32\kappa|t|}$. This implies \eqref{eq:psi-lin-4}.

The proof of \eqref{eq:psi-lin-3} is similar. It suffices to check that $\|f'(W_\lambda)\Lambda W_\mu\|_{\dot H^{-1}} \lesssim \eee^{-\frac 12 \kappa|t|}$,
and in fact we even have the bound $\lesssim \eee^{-\frac 32 \kappa|t|}$, with the same proof as in \eqref{eq:potential-holder}.
\end{proof}
\section{Bootstrap control of the error term -- the NLW case}
\label{sec:bootstrap}
In the preceding section we defined \emph{approximate} solutions of \eqref{eq:nlw0}.
In the present section we consider \emph{exact} solutions of \eqref{eq:nlw0}, with some specific initial data prescribed at $t = T$, with $T \to -\infty$.
Our goal is to control the evolution of this solution up to a time $T_0$ independent of $T$.

For technical reasons we will require the initial data to belong to the space $X^1\times H^1$, where $X^1 := \dot H^2 \cap \dot H^1$.
This regularity is preserved by the flow, see Proposition~\ref{prop:persistence}.
\subsection{Set-up of the bootstrap}
\label{eq:bootstrap-setup}
It is known that $L = -\Delta - f'(W)$ has exactly one strictly negative simple eigenvalue which we denote $-\nu^2$ (we take $\nu > 0$).
We denote the corresponding positive eigenfunction $\cY$, normalized so that $\|\cY\|_{L^2} = 1$.
By elliptic regularity $\cY$ is smooth and by Agmon estimates it decays exponentially.
Self-adjointness of $L$ implies that
\begin{equation}
  \label{eq:YLW}
  \la \cY, \Lambda W\ra = 0.
\end{equation}
Note that
\begin{equation}
  \label{eq:nu-l-1}
  \nu < 1.
\end{equation}
Indeed, it is well-known that $-\Delta - W \geq 0$, with a one-dimensional kernel generated by $W$.
Since $1 - W(x) > 0$ almost everywhere, for any $h \neq 0$ we have
\begin{equation}
  \la h, Lh\ra + \la h, h\ra = \la (-\Delta - 2W + 1)h, h\ra > \la (-\Delta - W)h, h\ra \geq 0.
\end{equation}

We define
\begin{equation}
  \label{eq:Y}
  \ym := \big(\frac 1\nu\cY, -\cY\big),\qquad \yp := \big(\frac 1\nu\cY, \cY\big),
\end{equation}
\begin{equation}
  \label{eq:a}
  \alpha^- := \frac{\nu}{2}J\cY^+ = \frac 12(\nu\cY, -\cY),\qquad \alpha^+ := -\frac{\nu}{2}J\cY^- =\frac 12(\nu\cY, \cY).
\end{equation}

We have $J\circ\vD^2 E(\bs W) = \begin{pmatrix} 0 & \Id \\ -L & 0\end{pmatrix}$. A short computation shows that
\begin{equation}
  \label{eq:eigenvect}
  J\circ\vD^2 E(\bs W)\ym = -\nu \ym,\qquad J\circ\vD^2 E(\bs W)\yp = \nu \yp
\end{equation}
and
\begin{equation}
  \label{eq:eigencovect}
  \la\alpha^-, J\circ\vD^2 E(\bs W)\bs h\ra = -\nu\la\alpha^-, \bs h\ra,\qquad \la\alpha^+, J\circ\vD^2 E(\bs W)\bs h\ra = \nu\la\alpha^+, \bs h\ra,\qquad \forall \bs h\in\cE.
\end{equation}
We will think of $\alpha^-$ and $\alpha^+$ as linear forms on $\cE$.
Notice that $\la \alpha^-, \ym\ra = \la \alpha^+, \yp\ra = 1$ and $\la \alpha^-, \yp\ra = \la \alpha^+, \ym\ra = 0$.

The rescaled versions of these objects are
\begin{equation}
  \label{eq:Yl}
  \yml := \big(\frac 1\nu\cY_\lambda, -\cY_\uln\lambda\big),\qquad \ypl := \big(\frac 1\nu\cY_\lambda, \cY_\uln\lambda\big),
\end{equation}
\begin{equation}
  \label{eq:al}
  \alpha^-_\lambda := \frac{\nu}{2\lambda}J\cY_\lambda^+ = \frac 12\big(\frac{\nu}{\lambda}\cY_\uln\lambda, -\cY_\uln\lambda\big),\qquad \alpha^+_\lambda := -\frac{\nu}{2\lambda}J\cY_\lambda^- = \frac 12\big(\frac{\nu}{\lambda}\cY_\uln\lambda, \cY_\uln\lambda\big).
\end{equation}
The scaling is chosen so that $\la\alpha_\lambda^-, \ym_\lambda\ra = \la\alpha_\lambda^+, \yp_\lambda\ra = 1$.
We have
\begin{equation}
  \label{eq:eigenvectl}
  J\circ\vD^2 E(\bs W_\lambda)\ym_\lambda = -\frac{\nu}{\lambda} \ym_\lambda,\qquad J\circ\vD^2 E(\bs W_\lambda)\yp_\lambda = \frac{\nu}{\lambda} \yp_\lambda
\end{equation}
and
\begin{equation}
  \label{eq:eigencovectl}
  \la\alpha_\lambda^-, J\circ\vD^2 E(\bs W_\lambda)\bs h\ra = -\frac{\nu}{\lambda}\la\alpha_\lambda^-, \bs h\ra,
  \qquad \la\alpha_\lambda^+, J\circ\vD^2 E(\bs W_\lambda)\bs h\ra = \frac{\nu}{\lambda}\la\alpha_\lambda^+, \bs h\ra,\qquad \forall \bs h\in\cE.
\end{equation}

We will need the following simple lemma in order to properly choose the initial data.
\begin{lemma}
  \label{lem:initial}
  There exist universal constants $\eta, C > 0$ such that if $0 < \lambda < \eta \cdot \mu$,
  then for all $a_0 \in \bR$ there exists $\bs h_0 \in X^1\times H^1$ satisfying the orthogonality conditions $\la \cZ_\mu, h_0\ra = \la \cZ_\lambda, h_0\ra = 0$
  and such that $\la \alpha_\mu^+, \bs h_0\ra = 0$, $\la \alpha_\mu^-, \bs h_0\ra = 0$, $\la \alpha_\lambda^+, \bs h_0\ra = a_0$,
  $\la \alpha_\lambda^-, \bs h_0\ra = 0$, $\|\bs h_0\|_\cE \leq C|a_0|$.
\end{lemma}
\begin{proof}
  We consider functions of the form:
  \begin{equation}
    \bs h_0 := a_2^+ \cY_\mu^+ + a_2^- \cY_\mu^- + b_2\Lambda \bs W_\mu + a_1^+ \cY_\lambda^+ + a_1^- \cY_\lambda^- + b_1\Lambda \bs W_\lambda, \qquad a_2^+, a_2^-, b_2, a_1^+, a_1^-, b_1 \in \bR.
  \end{equation}
Consider the linear map $\Phi: \bR^6 \to \bR^6$ defined as follows:
\begin{equation}
  \Phi(a_2^+, a_2^-, b_2, a_1^+, a_1^-, b_1) := \big(\la \alpha_\mu^+, \bs h_0\ra, \la \alpha_\mu^-, \bs h_0\ra, \la \frac{1}{\mu}\cZ_{\uln\mu}, h_0\ra,
  \la \alpha_\lambda^+, \bs h_0\ra, \la \alpha_\lambda^-, \bs h_0\ra, \la \frac{1}{\lambda}\cZ_\uln\lambda, h_0\ra\big)
\end{equation}
It is easy to check that the matrix of $\Phi$ is strictly diagonally dominant if $\eta$ is small enough.
\end{proof}
We consider the solution $\bs u(t) = \bs u(a_0; t): [T, T_+) \to \cE$ of \eqref{eq:nlw0}
  with the initial data
  \begin{equation}
    \label{eq:data-at-T}
    \bs u(T) = \big(W_{\frac{1}{\kappa}\eee^{-\kappa|T|}} + W + h_0, -\Lambda W_\uln{\frac{1}{\kappa}\eee^{-\kappa|T|}}\big),
\end{equation}
where $h_0$ is the function given by Lemma~\ref{lem:initial} with $\lambda = \frac{1}{\kappa}\eee^{-\kappa|T|}$, $\mu = 1$ and some $a_0$
chosen later, satisfying
\begin{equation}
  \label{eq:borne-instable-in}
  |a_0| \leq \eee^{-\frac 32 \kappa|T|}.
\end{equation}
Note that the initial data depend continuously on $a_0$.

For $t \geq T$ we define the functions $\lambda(t)$ and $\mu(t)$ as the solutions of the following system of ordinary differential equations
with the initial data $\mu(T) = 1$ and $\lambda(T) = \frac{1}{\kappa}\eee^{-\kappa|T|}$:
\begin{equation}
  \label{eq:lambda-mu-diff}
  \begin{pmatrix}
    \la \cZ_\uln{\lambda}, \Lambda W_\uln{\lambda}\ra - \la \frac{1}{\lambda}\Lambda_0 \cZ_\uln{\lambda}, h\ra & \la \cZ_\uln{\lambda}, \Lambda W_\uln{\mu}\ra \\
    \la \cZ_\uln{\mu}, \Lambda W_\uln{\lambda}\ra & \la \cZ_\uln{\mu}, \Lambda W_\uln{\mu}\ra - \la \frac{1}{\mu}\Lambda_0\cZ_\uln{\mu}, h\ra
  \end{pmatrix}
  \cdot \begin{pmatrix}
    \lambda' \\ \mu'
  \end{pmatrix} =
  \begin{pmatrix}
    -\la \cZ_\uln{\lambda}, \partial_t u\ra \\
    -\la \cZ_\uln{\mu}, \partial_t u\ra
  \end{pmatrix},
\end{equation}
where
\begin{equation}
  \label{eq:h-def}
h = h(t) := \begin{cases}
  u(t) - W_{\mu(t)} - W_{\lambda(t)},\ &\text{if }\|u(t) - W_{\mu(t)} - W_{\lambda(t)}\|_{\dot H^1} \leq \eta, \\
  \frac{\eta}{\|u(t) - W_{\mu(t)} - W_{\lambda(t)}\|_{\dot H^1}}\big(u(t) - W_{\mu(t)} - W_{\lambda(t)}\big),\ &\text{if }\|u(t) - W_{\mu(t)} - W_{\lambda(t)}\|_{\dot H^1} \geq \eta
\end{cases}
\end{equation}
with a small constant $\eta > 0$.
Notice that $\la \cZ_\uln\lambda, \Lambda W_\uln\lambda\ra = \la \cZ_\uln\mu, \Lambda W_\uln\mu\ra = \la \cZ, \Lambda W\ra > 0$,
$|\la \frac{1}{\lambda}\Lambda_0 \cZ_\uln{\lambda}, h\ra| + |\la \frac{1}{\mu}\Lambda_0\cZ_\uln{\mu}, h\ra| \lesssim \|h\|_{\dot H^1}$
and $|\la \cZ_\uln\lambda, \Lambda W_\uln\mu\ra| + |\la \cZ_\uln\mu, \Lambda W_\uln\lambda\ra| \lesssim \lambda/\mu$.
For $t < T_0$ bounds \eqref{eq:lambda-bound0} imply that $\lambda/\mu$ is small, hence equation \eqref{eq:lambda-mu-diff} defines a unique solution as long as \eqref{eq:lambda-bound0} holds.
\begin{remark}
  Actually the second case in the definition of $h(t)$ will never occur in our analysis, since the bootstrap assumptions imply that $\|h(t)\|_{\dot H^1}$ is small.
\end{remark}

Suppose that $\lambda(t)$ and $\mu(t)$ are well defined and satisfy \eqref{eq:lambda-bound0} for $t \in [T, T_1]$, where $T < T_1 \leq T_0$.
Suppose also that $\|h(t)\|_{\dot H^1} < \eta$ for $t \in [T, T_1]$, which implies that $h(t) = u(t)$.
Using \eqref{eq:lambda-mu-diff} we find $\dd t\la \cZ_\uln{\mu(t)}, h(t)\ra = 0$ and $\dd t \la \cZ_\uln{\lambda(t)}, h(t)\ra = 0$.
Since $\la \cZ_\uln{\mu(T)}, h(T)\ra = 0$ and $\la \cZ_\uln{\lambda(T)}, h(T)\ra = 0$, we obtain
\begin{equation}
  \label{eq:h-orth}
  \la \cZ_\uln{\mu(t)}, h(t)\ra = \la \cZ_\uln{\lambda(t)}, h(t)\ra = 0,\qquad \text{for }t \in [T, T_1].
\end{equation}
We denote $\bs h(t) := (h(t), \partial_t u(t))$, so that
\begin{equation}
  \label{eq:dth}
  \bigg\{
    \begin{aligned}
      \partial_t h &= \dot h + \mu' \Lambda W_\uln\mu + \lambda' \Lambda W_\uln\lambda, \\
      \partial_t \dot h &= \Delta h + f(W_\mu + W_\lambda + h) - f(W_\mu) - f(W_\lambda).
    \end{aligned}
  \end{equation}
We define the function $b(t): [T, T_1] \to \bR$ by formula \eqref{eq:b-def} and decompose
\begin{equation}
  \label{eq:g-def}
  \bs u(t) = \bs \varphi(t) + \bs g(t),\qquad t\in [T, T_1].
\end{equation}
By the definitions of $\bs g(t)$ and $\bs \psi(t)$, $\bs g(t)$ satisfies the differential equation
  \begin{equation}
    \label{eq:g-diff}
    \partial_t \bs g(t) = J\circ\vD E(\bs \varphi(t)+\bs g(t))-J\circ\vD E(\bs \varphi(t)) - \bs \psi(t).
  \end{equation}
Finally, we denote
\begin{equation}
    \begin{aligned}
    &a_1^+(t) := \la \alpha_{\lambda(t)}^+, \bs g(t)\ra, \qquad &a_1^-(t) := \la \alpha_{\lambda(t)}^-, \bs g(t)\ra, \\
    &a_2^+(t) := \la \alpha_{\mu(t)}^+, \bs g(t)\ra, \qquad &a_2^-(t) := \la \alpha_{\mu(t)}^-, \bs g(t)\ra.
  \end{aligned}
\end{equation}

The rest of this section is devoted to the proof of the following bootstrap estimate, which is the~heart of the whole construction.
\begin{proposition}
  \label{prop:bootstrap}
  There exist constants $C_0 > 0$ and $T_0 <0$ ($C_0$ and $|T_0|$ large) with the following property.
  Let $T < T_1 < T_0$ and suppose that $\bs u(t) = \bs \varphi(t) + \bs g(t) \in C([T, T_1]; X^1\times H^1)$
  is a solution of \eqref{eq:nlw0} with initial data \eqref{eq:data-at-T}
  such that for $t \in [T, T_1]$ condition \eqref{eq:lambda-bound0} is satisfied and
      \begin{align}
        \|\bs g(t)\|_\cE &\leq C_0 \cdot\eee^{-\frac 32 \kappa|t|}, \label{eq:assumption-g} \\
        |a_1^+(t)| &\leq \eee^{-\frac 32 \kappa|t|}. \label{eq:assumption-a}
      \end{align}
  Then for $t \in [T, T_1]$ there holds
  \begin{align}
    \|\bs g(t)\|_\cE &\leq \frac 12 C_0\eee^{-\frac 32 \kappa|t|}, \label{eq:bootstrap-g} \\
  \big|\lambda(t) - \frac{1}{\kappa}\eee^{-\kappa|t|}\big| + |\mu(t) - 1| &\lesssim C_0\eee^{-\frac 32 \kappa|t|}. \label{eq:bootstrap-l}
  \end{align}
\end{proposition}
\begin{remark}
  Notice that \eqref{eq:bootstrap-g} and \eqref{eq:bootstrap-l} are strictly stronger
  than \eqref{eq:assumption-g} and \eqref{eq:lambda-bound0} respectively,
  which will be crucial for closing the bootstrap in Subsection~\ref{ssec:limit}.
%
\end{remark}
\begin{remark}
  The same conclusion should be true without the assumption of $X^1\times H^1$
  regularity, by means of a standard approximation procedure
  (both the assumptions and the conclusion are continuous for the topology $\|\cdot\|_\cE$).
\end{remark}

\subsection{Modulation}
\label{ssec:mod}

\begin{lemma}
  \label{lem:mod}
  Under assumptions \eqref{eq:lambda-bound0} and \eqref{eq:assumption-g}, for $t \in [T, T_1]$ there holds
  \begin{gather}
    \label{eq:g-orth}
    \big\la \frac{1}{\lambda(t)}\cZ_\uln{\lambda(t)}, g(t)\big\ra = 0, \quad \big|\big\la \frac{1}{\mu(t)}\cZ_\uln{\mu(t)}, g(t)\big\ra\big| \lesssim c\cdot\eee^{-\frac 32\kappa|t|}, \\
    \label{eq:mod}
    |\lambda'(t)-b(t)| + |\mu'(t)| \lesssim \|\bs g(t)\|_\cE + c\cdot \eee^{-\frac 32 \kappa|t|},
  \end{gather}
  with a constant $c$ arbitrarily small.
\end{lemma}
\begin{proof}
  We have $g(t) = h(t) - S(t)$. Since $\cZ$ has compact support, \eqref{eq:S-orth} and \eqref{eq:h-orth} yield $\la \cZ_\uln{\lambda}, g\ra = 0$.
  From $\|\frac{1}{\mu}\cZ_\uln{\mu}\|_{L^\infty} \lesssim 1$ and $\|\chi\cdot \lambda^2 P_\lambda\|_{L^1} \leq \|\lambda^2 P_\lambda\|_{L^1(|x|\leq 2)} \lesssim \lambda^2 \sim \eee^{-2\kappa|t|}$
  (analogously $\|\chi\cdot b^2 Q_\lambda\|_{L^1} \lesssim \eee^{-2\kappa|t|}$) we obtain $|\la \frac{1}{\mu}\cZ_\uln\mu, g\ra| \lesssim \eee^{-2\kappa|t|}$.

  From \eqref{eq:taille-correct} we have
\begin{equation}
  \label{eq:taille-g-h}
  \|g(t) - h(t)\|_{\dot H^1} \leq c\cdot \eee^{-\frac 32 \kappa|t|},\qquad \text{with a small constant }c,
\end{equation}
Using this, \eqref{eq:lambda-mu-diff} yields
\begin{equation}
  \begin{aligned}
  &\begin{pmatrix}
    \la \cZ_\uln{\lambda}, \Lambda W_\uln{\lambda}\ra - \la \frac{1}{\lambda}\Lambda_0 \cZ_\uln{\lambda}, g\ra & \la \cZ_\uln{\lambda}, \Lambda W_\uln{\mu}\ra \\
    \la \cZ_\uln{\mu}, \Lambda W_\uln{\lambda}\ra & \la \cZ_\uln{\mu}, \Lambda W_\uln{\mu}\ra - \la \frac{1}{\mu}\Lambda_0\cZ_\uln{\mu}, g\ra
  \end{pmatrix}
  \cdot \begin{pmatrix}
    \lambda' - b \\ \mu'
  \end{pmatrix} \\
  &=
  \begin{pmatrix}
    -\la \cZ_\uln{\lambda}, \partial_t u + b\Lambda W_\uln\lambda\ra \\
    -\la \cZ_\uln{\mu}, \partial_t u + b\Lambda W_\uln\lambda\ra
  \end{pmatrix} + o_{\bR^2}(\eee^{-\frac 32 \kappa|t|}).
\end{aligned}
\end{equation}
Since by definition $\partial_t u + b\Lambda W_\uln\lambda = \dot g$, inverting the matrix we get \eqref{eq:mod}.
\end{proof}

\begin{lemma}
  \label{lem:eigen}
  Under assumptions \eqref{eq:lambda-bound0} and \eqref{eq:assumption-g}, for $t \in [T, T_1]$ the functions $a_1^\pm(t)$ and $a_2^\pm(t)$ satisfy
\begin{gather}
  \big|\dd t a_1^+(t) - \frac{\nu}{\lambda(t)}a_1^+(t)\big| \leq c\cdot\eee^{-\frac 12\kappa|t|},\qquad \text{with a small constant }c \label{eq:a-1p} \\
  |a_1^-(t)|\leq \eee^{-\frac 32 \kappa|t|}, \label{eq:a-1m} \\
  |a_2^\pm(t)| \leq \eee^{-\frac 32 \kappa|t|}. \label{eq:a-reste}
\end{gather}
\end{lemma}
\begin{proof}
  Using the definition of $a_1^+(t)$ we compute
  \begin{equation}
    \label{eq:a1p-0}
\begin{aligned}
    &\dd t a_1^+(t) = \dd t\la \alpha_{\lambda(t)}^+, \bs g(t)\ra  \\ &= \big\la {-}\frac{\lambda'}{\lambda}(\Lambda_{\cE^*}\alpha^+)_\lambda, \bs g\big\ra
    + \la \alpha_\lambda^+, J\circ \vD E(\bs \varphi + \bs g) - J\circ \vD E(\bs \varphi) - \bs \psi\ra,
\end{aligned}
  \end{equation}
  where $\Lambda_{\cE^*}\alpha^+ := -\pd \lambda\alpha_\lambda^+\vert_{\lambda=1}$.
  We have
  \begin{equation}
    \label{eq:a1p-1}
    |\la (\Lambda_{\cE^*}\alpha^+)_\lambda, \bs g\ra| \lesssim \|\bs g\|_\cE \ll \eee^{-\frac 12 \kappa|t|}.
  \end{equation}
  Since $\la\alpha_\lambda^+, (\Lambda W_\lambda, 0)\ra = 0$, using \eqref{eq:psi1} we obtain
  \begin{equation}
    \label{eq:a1p-2}
    |\la \alpha_\lambda^+, \bs \psi\ra| \lesssim \big\|\bs \psi - \frac{\lambda' - b}{\lambda}(\Lambda W_\lambda, 0)\big\|_\cE \ll \eee^{-\frac 12 \kappa|t|}.
  \end{equation}
  From \eqref{eq:point-f2} we obtain $\|f(\varphi + g) - f(\varphi) - f'(\varphi)g\|_{L^\frac 32} \lesssim \|g\|_{\dot H^1}^2$.
  From \eqref{eq:point-fp2} and \eqref{eq:taille-correct} we have $\|(f'(\varphi) - f'(W_\mu) - f'(W_\lambda))g\|_{L^\frac 32} \lesssim \|\varphi - W_\mu - W_\lambda\|_{L^3}\cdot\|g\|_{L^3}
  \lesssim \eee^{-\frac 32 \kappa|t|}\|g\|_{\dot H^1}$.
  Taking the sum we obtain
  \begin{equation}
    \label{eq:f-phi-g}
    \|f(\varphi + g) - f(\varphi) - (f'(W_\mu) + f'(W_\lambda))g\|_{L^\frac 32} \lesssim \|g\|_{\dot H^1}^2 + \eee^{-\frac 32 \kappa|t|}\|g\|_{\dot H^1}.
  \end{equation}
  But $ |\la \cY_\uln\lambda, f'(W_\mu)g\ra| \lesssim \|\cY_\uln\lambda\|_{L^\frac 32} \cdot \|f'(W_\mu)\|_{L^\infty}\cdot \|g\|_{L^3} \lesssim \lambda \|g\|_{\dot H^1}$, hence
  \begin{equation}
    \label{eq:a1p-3}
    |\la \cY_\uln\lambda, f(\varphi + g) - f(\varphi) - f'(W_\lambda)g\ra| \lesssim \frac{1}{\lambda}(\|g\|_{\dot H^1}^2 + \eee^{-\frac 32 \kappa|t|}\|g\|_{\dot H^1})\ll \eee^{-\frac 12 \kappa|t|}.
  \end{equation}
  Combining \eqref{eq:a1p-0} with \eqref{eq:a1p-1}, \eqref{eq:a1p-2} and \eqref{eq:a1p-3} we obtain
  \begin{equation}
    \dd t a_1^+(t) = \la \alpha_\lambda^+, J\circ \vD^2 E(\bs W_\lambda)\bs g\ra + o(\eee^{-\frac 12 \kappa|t|}) = \frac{\nu}{\lambda}a_1^+(t) + o(\eee^{-\frac 12 \kappa|t|}),
  \end{equation}
  where in the last step we use \eqref{eq:eigencovectl}. This proves \eqref{eq:a-1p}.
  
  Similarly, we have
  \begin{equation}
    \label{eq:a-1m-eq}
    \big|\dd t a_1^-(t) + \frac{\nu}{\lambda(t)}a_1^-(t)\big| \leq c\cdot\eee^{-\frac 12\kappa|t|}.
  \end{equation}
  Inequality \eqref{eq:taille-g-h} implies that \eqref{eq:a-1m} holds for $t$ in a neighborhood of $T$.
  Suppose that $T_2 \in (T, T_1)$ is the last time such that \eqref{eq:a-1m} holds for $t \in [T, T_2]$.
  But \eqref{eq:a-1m-eq} implies that $\dd t a_1^-(T_2)$ and $a_1^-(T_2)$ have opposite signs. Hence \eqref{eq:a-1m} cannot break down at $t = T_2$.
  The contradiction shows that \eqref{eq:a-1m} holds for $t \in [T, T_1]$.

  In order to prove \eqref{eq:a-reste}, it is more convenient to work with $\bs h(t)$, which was defined right before \eqref{eq:dth}, than with $\bs g(t)$.
  We will prove the bound for $a_2^+(t)$. The proof for $a_2^-(t)$ is exactly the same. Let $\wt a(t) := \la \alpha_{\mu(t)}^+, \bs h(t)\ra$.
  We have
  \begin{equation}
    \label{eq:a2-1}
\begin{aligned}
    &|\la \cY_\uln\mu, \Lambda W_\uln\lambda\ra| \lesssim \|\Lambda W_\uln\lambda\|_{L^1(|x|\leq 1)} + \|\Lambda W_\uln\lambda\|_{L^\infty(|x| \geq 1)} \\
    &\lesssim \lambda^3 \int_0^{\frac{1}{\lambda}}r^{-4}r^5\ud r + \frac{1}{\lambda^3}\cdot\big(\frac{1}{\lambda}\big)^{-4} \lesssim \lambda.
\end{aligned}
  \end{equation}
  Together with \eqref{eq:taille-correct} this yields
  $$
  |a_2^+(t) - \wt a(t)| = |\la \alpha_\mu^+, \bs g(t) - \bs h(t)\ra| \lesssim b|\la \cY\uln\mu, \Lambda W_\uln\lambda\ra| + \|S\|_{\dot H^1} \leq \frac 12 \eee^{-\frac 32 \kappa|t|},
  $$
  hence it suffices to show that
  \begin{equation}
    \label{eq:a2-3}
    |\wt a(t)|\leq \frac 12 \eee^{-\frac 32 \kappa|t|}.
  \end{equation}
  As in the case of $a_1^+(t)$, using \eqref{eq:dth} we obtain
  \begin{equation}
    \label{eq:a2-4}
    \begin{aligned}
    \dd t \wt a(t) &= \big\la {-}\frac{\mu'}{\mu}(\Lambda_{\cE^*}\alpha^+)_\mu, \bs h\big\ra + \la \alpha_\mu^+, J\circ\vD^2 E(\bs W_\mu)h\ra \\
    & +\la \alpha_\mu^+, (\mu'\Lambda W_\uln\mu + \lambda'\Lambda W_\uln\lambda, f(W_\mu + W_\lambda + h) - f(W_\mu) - f(W_\lambda) - f'(W_\mu)h)\ra  \end{aligned}
  \end{equation}
  But
  \begin{equation}
    \begin{gathered}
      \la \cY_\uln\mu, \Lambda W_\uln\mu\ra = 0, \\
      |\la \cY_\uln\mu, \Lambda W_\uln\lambda\ra| \lesssim \eee^{-\kappa|t|},\qquad \text{see \eqref{eq:a2-1},} \\
      |\la \cY_\uln\mu, f(W_\mu + W_\lambda + h) - f(W_\mu + W_\lambda) - f'(W_\mu + W_\lambda)h\ra| \lesssim \|h\|_{\dot H^1}^2 \lesssim \eee^{-2 \kappa|t|}, \\
\begin{aligned}
      |\la \cY_\uln\mu, f(W_\mu + W_\lambda) - f(W_\mu) - f(W_\lambda)\ra| &= 2|\la \cY_\uln\mu, W_\mu \cdot W_\lambda\ra| \\ &= 2|\la \cY_\uln\mu\cdot W_\mu, W_\lambda\ra| \lesssim \eee^{-2\kappa|t|},\qquad \text{see \eqref{eq:a2-1}},
\end{aligned} \\
      |\la \cY_\uln\mu, (f'(W_\mu + W_\lambda) - f'(W_\mu))h\ra| \lesssim \|\cY_\uln\mu\|_{L^6}\cdot \|W_\lambda\|_{L^2}\cdot \|h\|_{L^3} \lesssim \lambda\|h\|_{\dot H^1} \lesssim \eee^{-2\kappa|t|},
    \end{gathered}
  \end{equation}
  hence \eqref{eq:a2-4} yields
  \begin{equation}
    \big|\dd t \wt a(t) - \frac{\nu}{\mu(t)}\wt a(t)\big| \lesssim \eee^{-2 \kappa|t|} \leq c\cdot\eee^{-\frac 32\kappa|t|},
  \end{equation}
  with a constant $c$ arbitrarily small.
  Using \eqref{eq:lambda-bound0}, we get
  \begin{equation}
    \label{eq:a-2pm-eq}
    |\dd t \wt a(t)| \leq \frac{9\nu}{8}|\wt a(t)| + c\cdot\eee^{-\frac 32\kappa|t|}.
  \end{equation}
  As in the proof of \eqref{eq:a-1m-eq}, suppose that $T_2 \in (T, T_1)$ is the last time such that \eqref{eq:a2-3} holds for $t \in [T, T_2]$.
  This implies that $|\wt a(T_2)| = \frac 12 \eee^{-\frac 32 \kappa|T_2|}$ and $|\dd t \wt a(T_2)| \geq \frac 34\kappa\cdot \eee^{-\frac 32 \kappa|T_2|}$,
  thus \eqref{eq:a-2pm-eq} yields
  \begin{equation*}
    \frac 34 \kappa\cdot\eee^{-\frac 32 \kappa|T_2|} \leq \frac{9\nu}{16}\eee^{-\frac 32 \kappa|T_2|} + c\cdot \eee^{-\frac 32 \kappa|T_2|}.
  \end{equation*}
  But \eqref{eq:kappa} and \eqref{eq:nu-l-1} give $\frac 34 \kappa > \frac{9\nu}{16}$. Since $c$ is arbitrarily small, we obtain a contradiction.
\end{proof}
\subsection{Coercivity}
\label{ssec:prelim-coer}
Recall that we denote $L := -\Delta - f'(W)$.
\begin{lemma}
  \label{lem:loc-coer}
  There exist constants $c, C > 0$ such that
  \begin{itemize}[leftmargin=0.5cm]
    \item for all $g \in \dot H^1$ radially symmetric there holds
      \begin{equation}
        \label{eq:coer-lin-coer-1}
        \la g, Lg\ra = \int_{\bR^6}|\grad g|^2 \ud x - \int_{\bR^N}f'(W)|g|^2\ud x \geq c\int_{\bR^6}|\grad g|^2 \ud x -C\big(\la \cZ, g\ra^2 + \la \cY, g\ra^2\big),
      \end{equation}
    \item if $r_1 > 0$ is large enough, then for all $g \in \dot H^1_\tx{rad}$ there holds
      \begin{equation}
        \label{eq:coer-lin-coer-2}
        (1-2c)\int_{|x|\leq r_1}|\grad g|^2 \ud x + c\int_{|x|\geq r_1}|\grad g|^2\ud x - \int_{\bR^6}f'(W)|g|^2\ud x \geq -C\big(\la \cZ, g\ra^2 + \la \cY, g\ra^2\big),
      \end{equation}
    \item if $r_2 > 0$ is small enough, then for all $g \in \dot H^1_\tx{rad}$ there holds
      \begin{equation}
        \label{eq:coer-lin-coer-3}
        (1-2c)\int_{|x|\geq r_2}|\grad g|^2 \ud x + c\int_{|x|\leq r_2}|\grad g|^2\ud x - \int_{\bR^6}f'(W)|g|^2\ud x \geq -C\big(\la \cZ, g\ra^2 + \la \cY, g\ra^2\big).
      \end{equation}
  \end{itemize}
\end{lemma}
\begin{proof}
  This is exactly Lemma 2.1 in \cite{moi15p-3}, see also \cite[Lemma 2.1]{MaMe15p}.
\end{proof}

\begin{lemma}
  \label{lem:bulles-coer}
  There exists a constant $\eta > 0$ such that if $\frac{\lambda}{\mu} < \eta$ and $\|\bs U - (\bs W_{\mu} + \bs W_{\lambda})\|_\cE < \eta$,
  then for all $\bs g \in \cE$ there holds
  \begin{equation*}
    \frac 12 \la \vD^2 E(\bs U)\bs g, \bs g\ra +2\big(\la \alpha^-_{\lambda}, \bs g\ra^2 + \la \alpha^+_{\lambda}, \bs g\ra^2 + \la \frac{1}{\lambda}\cZ_\uln\lambda, g\ra^2
    + \la \alpha^-_{\mu}, \bs g\ra^2 + \la \alpha^+_{\mu}, \bs g\ra^2 + \la \frac{1}{\mu}\cZ_\uln\mu, g\ra^2\big) \gtrsim \|\bs g\|_\cE^2.
  \end{equation*}
\end{lemma}
\begin{proof}
  We will repeat with minor changes the proof of \cite[Lemma 3.5]{moi15p-3}.
  \paragraph{\textbf{Step 1}}
  Without loss of generality we can assume that $\mu = 1$.
  Consider the operator $\bs L_\lambda$ defined by the following formula:
  \begin{equation*}
    \bs L_\lambda := \begin{pmatrix} -\Delta - f'(W_{\lambda}) - f'(W) & 0 \\ 0 & \Id\end{pmatrix}.
  \end{equation*}
  From the fact that $\|f'(U) - f'(W) - f'(W_\lambda)\|_{L^3} \lesssim \|U - (W + W_\lambda)\|_{L^3}$ we obtain
  \begin{equation}
    \label{eq:bulles-coer-approx}
    |\la \vD^2 E(\bs U)\bs g, \bs g\ra - \la \bs L_{\lambda}\bs g, \bs g\ra| \leq c\|\bs g\|_\cE^2,\qquad \forall \bs g\in\cE,
  \end{equation}
  with $c > 0$ small when $\eta$ and $\lambda_0$ are small.
  \paragraph{\textbf{Step 2.}}
  In view of \eqref{eq:bulles-coer-approx}, it suffices to prove that if $\lambda < \lambda_0$, then
  \begin{equation*}
    \frac 12 \la \bs L_\lambda\bs g, \bs g\ra +2\big(\la \alpha^-_{\lambda_1}, \bs g\ra^2 + \la \alpha^+_{\lambda_1}, \bs g\ra^2 + \la \alpha^-_{\lambda_2}, \bs g\ra^2 + \la \alpha^+_{\lambda_2}, \bs g\ra^2 + \la \frac{1}{\lambda}\cZ_\uln\lambda, g\ra^2 + \la \cZ, g\ra^2\big) \gtrsim \|\bs g\|_\cE^2.
  \end{equation*}
  Let $a_1^- :=\la \alpha_\lambda^-, \bs g\ra$, $a_1^+ := \la \alpha_\lambda^+, \bs g\ra$, $a_2^- := \la\alpha^-, \bs g\ra$, $a_2^+ := \la\alpha^+, \bs g\ra$,
  $b_1 := \la \cZ, \Lambda W\ra^{-1}\cdot \la \frac{1}{\lambda}\cZ_{\uln\lambda}, g\ra$, $b_2 := \la \cZ, \Lambda W\ra^{-1}\cdot \la \cZ, g\ra$
  and decompose
  $$\bs g = a_1^-\cY_\lambda^- + a_1^+\cY_\lambda^+ + a_2^-\ym + a_2^+\yp + b_1 \Lambda \bs W_\lambda + b_2\Lambda \bs W + \bs k.$$
  Using the fact that
  $$
  \begin{aligned}
    |\la\alpha^\pm, \cY^\pm_\lambda\ra| + |\la \alpha_\lambda^\pm, \cY^\pm\ra| + |\la \frac{1}{\lambda}\cZ_\uln\lambda, \cY\ra| + |\la \cZ, \cY_\lambda\ra| &\lesssim \lambda^2, \\
    |a_1^-| + |a_1^+| + |a_2^-| + |a_2^+| + |b_1| + |b_2| &\lesssim \|\bs g\|_\cE, \\
    \la \alpha^-, \cY^+\ra = \la \alpha^+, \cY^-\ra = \la \cZ, \cY \ra = \la \cY, \Lambda W\ra &= 0
  \end{aligned}
  $$ we obtain
  \begin{equation}
    \label{eq:bulles-coer-eigendir}
    \la\alpha^-, \bs k\ra^2 + \la\alpha^+, \bs k\ra^2 + \la \alpha_\lambda^-, \bs k\ra^2 + \la \alpha_\lambda^+, \bs k\ra^2 + \la \cZ, k\ra^2 + \la \frac{1}{\lambda}\cZ_\uln\lambda, k\ra^2 \lesssim \lambda^4\cdot\|\bs g\|_\cE^2.
  \end{equation}

  Since $\bs L_\lambda$ is self-adjoint, we can write
  \begin{equation}
    \label{eq:bulles-coer-expansion}
    \begin{aligned}
      \frac 12 \la \bs L_\lambda \bs g, \bs g\ra &= \frac 12 \la \bs L_\lambda\bs k, \bs k\ra  \\
      &+ \la \bs L_\lambda(a_2^-\cY^- + a_2^+\cY^+ + b_2 \Lambda \bs W), \bs k\ra + \la \bs L_\lambda(a_1^-\cY_\lambda^- + a_1^+\cY_\lambda^+ + b_1 \Lambda \bs W_\lambda), \bs k\ra \\
      &+ \frac 12 \la \bs L_\lambda(a_2^-\cY^- + a_2^+\cY^+ + b_2 \Lambda \bs W), a_2^-\cY^- + a_2^+\cY^+ + b_2 \Lambda \bs W\ra \\
      &+ \frac 12 \la \bs L_\lambda(a_1^-\cY_\lambda^- + a_1^+\cY_\lambda^+ + b_1 \Lambda \bs W_\lambda), a_1^-\cY_\lambda^- + a_1^+\cY_\lambda^+ + b_1 \Lambda \bs W_\lambda\ra \\
      &+ \la \bs L_\lambda(a_2^-\cY^- + a_2^+\cY^+ + b_2 \Lambda \bs W), a_1^-\cY_\lambda^- + a_1^+\cY_\lambda^+ + b_1 \Lambda \bs W_\lambda\ra.
    \end{aligned}
  \end{equation}
  It is easy to see that $\|f'(W)\cY_\lambda\|_{L^\frac 32} \to 0$, $\|f'(W) \Lambda W_\lambda\|_{L^\frac 32} \to 0$, $\|f'(W_\lambda)\cY\|_{L^\frac 32} \to 0$ and $\|f'(W_\lambda)\Lambda W\|_{L^\frac 32} \to 0$ as $\lambda \to 0$. This and \eqref{eq:al}, \eqref{eq:eigenvectl} imply
  $$
  \begin{gathered}
  \|\bs L_\lambda \cY^- + 2\alpha^+\|_{\cE^*} + \|\bs L_\lambda \cY^+ + 2\alpha^-\|_{\cE^*} + \|\bs L_\lambda \Lambda \bs W\|_{\cE^*} \\
  + \|\bs L_\lambda \cY_\lambda^- + 2\alpha_\lambda^+\|_{\cE^*} + \|\bs L_\lambda \cY_\lambda^+ + 2\alpha_\lambda^-\|_{\cE^*} + \|\bs L_\lambda \Lambda \bs W_\lambda\|_{\cE^*} \sto{\lambda \to 0} 0.
\end{gathered}
  $$
  Plugging this into \eqref{eq:bulles-coer-expansion} and using \eqref{eq:bulles-coer-eigendir} we obtain
  \begin{equation}
    \label{eq:bulles-coer-approx-6}
    \frac 12\la \bs L_\lambda \bs g, \bs g\ra \geq -2a_2^-a_2^+ - 2a_1^-a_1^+ + \frac 12\la \bs L_\lambda\bs k, \bs k\ra -\wt c \|\bs g\|_\cE^2,
  \end{equation}
  where $\wt c \to 0$ as $\lambda \to 0$.

  Applying \eqref{eq:coer-lin-coer-2} with $r_1 = \lambda^{-\frac 12}$, rescaling and using \eqref{eq:bulles-coer-eigendir} we get, for $\lambda$ small enough,
  \begin{equation}
    \label{eq:bulles-coer-approx-4}
    (1-2c)\int_{|x|\leq \sqrt\lambda}|\grad k|^2 \ud x + c\int_{|x|\geq \sqrt\lambda}|\grad k|^2\ud x - \int_{\bR^6}f'(W_\lambda)|k|^2\ud x \geq -\wt c \|\bs g\|_\cE^2.
  \end{equation}
  From \eqref{eq:coer-lin-coer-3} with $r_2 = \sqrt\lambda$ we have
  \begin{equation}
    \label{eq:bulles-coer-approx-5}
    (1-2c)\int_{|x|\geq \sqrt\lambda}|\grad k|^2 \ud x + c\int_{|x|\leq \sqrt\lambda}|\grad k|^2\ud x - \int_{\bR^6}f'(W)|k|^2\ud x \geq -\wt c \|\bs g\|_\cE^2.
  \end{equation}
  Taking the sum of \eqref{eq:bulles-coer-approx-4} and \eqref{eq:bulles-coer-approx-5}, and using \eqref{eq:bulles-coer-approx-6}, we obtain
  \begin{equation*}
    \frac 12\la \bs L_\lambda \bs g, \bs g\ra \geq -2a_2^-a_2^+ - 2a_1^-a_1^+ + c\|\bs k\|_\cE^2 - 2\wt c\|\bs g\|_\cE^2.
  \end{equation*}
  The conclusion follows if we take $\wt c$ small enough.
\end{proof}

\subsection{Definition of the mixed energy-virial functional}
\label{ssec:bootstrap}
\begin{lemma}
  \label{lem:fun-a}
  For any $c > 0$ and $R > 0$ there exists a radial function $q(x) = q_{c,R}(x) \in C^{3,1}(\bR^6)$ with the following properties:
  \begin{enumerate}[label=(P\arabic*)]
    \item $q(x) = \frac 12 |x|^2$ for $|x| \leq R$, \label{enum:approx}
    \item there exists $\wt R > 0$ (depending on $c$ and $R$) such that $q(x) \equiv \tx{const}$ for $|x| \geq \wt R$, \label{enum:support}
    \item $|\grad q(x)| \lesssim |x|$ and $|\Delta q(x)| \lesssim 1$ for all $x \in \bR^6$, with constants independent of $c$ and $R$, \label{enum:gradlap}
    \item $\sum_{1\leq i,j\leq 6} \big(\partial_{x_i x_j} q(x)\big) v_i v_j \geq -c\sum_{i=1}^6 v_i^2$, for all $x \in \bR^6, v_i \in \bR$, \label{enum:convex}
    \item $\Delta^2 q(x) \leq c\cdot|x|^{-2}$, for all $x \in \bR^6$. \label{enum:bilapl}
  \end{enumerate}
\end{lemma}
\begin{remark}
  We require $C^{3, 1}$ regularity in order not to worry about boundary terms in Pohozaev identities, see the proof of \eqref{eq:A-pohozaev}.
\end{remark}
\begin{proof}
  It suffices to prove the result for $R = 1$ since the function $q_R(x) := R^2 q(\frac xR)$ satisfies
  the listed properties if and only if $q(x)$ does.
  
  Let $r$ denote the radial coordinate. Define $q_0(x)$ by the formula
  \begin{equation}
    \label{eq:q0}
    q_0(r) := \bigg\{
      \begin{aligned}
        &{\textstyle \frac 12 \cdot r^2}\qquad & r\leq 1 \\
        &{\textstyle \frac 85 \cdot r - \frac 32+ \frac 12\cdot r^{-2} - \frac{1}{10}\cdot r^{-4}}\qquad & r \geq 1.
      \end{aligned}
    \end{equation}
    A direct computation shows that for $r > 1$ we have $q_0'(r) = \frac 85 - r^{-3} + \frac 25 r^{-5}$,
    $q_0''(r) = 3r^{-4} - 2r^{-6} > 0$ (so $q_0(x)$ is convex), $q_0'''(r) = 12(-r^{-5} + r^{-7})$ and $\Delta^2 q_0(r) = -24 r^{-3}$.
    Hence $q_0$ satisfies all the listed properties except for \ref{enum:support}. We correct it as follows.

    Let $e_j(r) := \frac{1}{j!}r^j\cdot \chi(r)$ for $j \in \{1, 2, 3\}$ and let $R_0 \gg 1$. We define
  \begin{equation}
    \label{eq:q}
    q(r) := \bigg\{
      \begin{aligned}
        &q_0(r)\qquad & r\leq R_0 \\
        &{\textstyle q_0(R_0) + \sum_{j=1}^{3}q_0^{(j)}(R_0)\cdot R_0^j\cdot e_j(-1 + R_0^{-1}r)}\qquad & r \geq R_0.
      \end{aligned}
    \end{equation}
    Note that $q_0'(R_0) \sim 1$, $q_0''(R_0) \sim R_0^{-4}$ and $q_0'''(R_0) \sim R_0^{-5}$.
    It is clear that $q(x) \in C^{3,1}(\bR^6)$.
    Property \ref{enum:approx} holds since $R_0 > 1$. By the definition of the functions $e_j$ we have $q(r) = q_0(R_0) = \tx{const}$ for $r \geq 3R_0$,
    hence \ref{enum:support} holds with $\wt R = 3R_0$. From the definition of $q(r)$ we get $|\partial_r q(r)| \lesssim 1$ and $|\partial_r^2 q(r)| \lesssim R_0^{-4}$ for $r \geq R_0$,
    with a constant independent of $R_0$, which implies \ref{enum:gradlap}. Similarly, $|\partial_{x_i x_j}q(x)| \lesssim R_0^{-1}$ for $|x| \geq R_0$, which implies \ref{enum:convex}
    if $R_0$ is large enough. Finally $|\Delta^2 q(x)| \lesssim R_0^{-3}$ for $|x| \geq R_0$ and $\Delta^2 q(x) = 0$ for $|x| \geq 3R_0$. This proves \ref{enum:bilapl} if $R_0$ is large enough.
\end{proof}
In the sequel $q(x)$ always denotes a function of class $C^{3, 1}(\bR^6)$ verifying \ref{enum:approx}--\ref{enum:bilapl}
with sufficiently small $c$ and sufficiently large $R$.

For $\lambda > 0$ we define the operators $A(\lambda)$ and $A_0(\lambda)$ as follows:
\begin{align}
  [A(\lambda)h](x) &:= \frac{1}{3\lambda}\Delta q(\frac{x}{\lambda})h(x) + \grad q(\frac{x}{\lambda})\cdot \grad h(x), \label{eq:opA} \\
  [A_0(\lambda)h](x) &:= \frac{1}{2\lambda}\Delta q(\frac{x}{\lambda})h(x) + \grad q(\frac{x}{\lambda})\cdot \grad h(x). \label{eq:opA0}
\end{align}
Combining these definitions with the fact that $q(x)$ is an approximation of $\frac 12 |x|^2$
we see that $A(\lambda)$ and $A_0(\lambda)$ are approximations (in a sense not yet precised)
of $\frac{1}{\lambda}\Lambda$ and $\frac{1}{\lambda}\Lambda_0$ respectively.
We will write $A$ and $A_0$ instead of $A(1)$ and $A_0(1)$ respectively. Note the following scale-change formulas, which follow directly from the definitions:
\begin{equation}
  \label{eq:A-rescale}
  \forall h\in \dot H^1:\qquad A(\lambda)(h_\lambda) = (Ah)_\uln\lambda,\quad A_0(\lambda)(h_\lambda) = (A_0 h)_\uln\lambda.
\end{equation}

\begin{lemma}
  \label{lem:op-A}
  The operators $A(\lambda)$ and $A_0(\lambda)$ have the following properties:
  \begin{itemize}
    \item the families $\{A(\lambda): \lambda > 0\}$, $\{A_0(\lambda): \lambda > 0\}$, $\{\lambda\partial_\lambda A(\lambda): \lambda > 0\}$
      and $\{\lambda\partial_\lambda A_0(\lambda): \lambda > 0\}$ are bounded in $\scrL(\dot H^1; L^2)$, with the bound depending on the choice of the function $q(x)$,
    \item for all $h_1, h_2 \in X^1$ and $\lambda > 0$ there holds
      \begin{equation}
        \label{eq:A-by-parts}
        \la A(\lambda)h_1, f(h_1 + h_2) - f(h_1) - f'(h_1)h_2\ra = -\la A(\lambda)h_2, f(h_1+h_2) - f(h_1)\ra,
      \end{equation}
    \item for any $c_0 > 0$, if we choose $c$ in Lemma~\ref{lem:fun-a} small enough, then for all $h \in X^1$ there holds
      \begin{equation}
        \label{eq:A-pohozaev}
        \la A_0(\lambda)h, \Delta h\ra \leq \frac{c_0}{\lambda} \|h\|_{\dot H^1}^2 - \frac{1}{\lambda}\int_{|x| \leq R\lambda}|\grad h(x)|^2 \ud x.
      \end{equation}
    \item assuming \eqref{eq:lambda-bound0} and \eqref{eq:assumption-g}, for any $c_0 > 0$ there holds
    \begin{align}
      \label{eq:L0-A0}
      \|\Lambda_0 \Lambda W_\uln{\lambda(t)} - A_0(\lambda(t))\Lambda W_{\lambda(t)}\|_{L^2} &\leq c_0, \\
      \label{eq:L-A}
      \|\dot \varphi(t) + b(t)\cdot A(\lambda(t))\varphi(t)\|_{L^3} &\leq c_0, \\
      \label{eq:approx-potential}
      \Big|\int\frac 16 \Delta q\big(\frac{x}{\lambda}\big)(f(\varphi + g) - f(\varphi))g\ud x - \int f'(W_\lambda)g^2 \ud x\Big| &\leq c_0C_0^2 \eee^{-3\kappa|t|}.
    \end{align}
    provided that the constant $R$ in the definition of $q(x)$ is chosen large enough.
  \end{itemize}
\end{lemma}
\begin{proof}
  Since $\grad q(x)$ and $\grad^2 q(x)$ are continuous and of compact support, it is clear that $A$ and $A_0$ are bounded operators $\dot H^1 \to L^2$.
  From the invariance \eqref{eq:A-rescale} we see that $A(\lambda)$ and $A_0(\lambda)$ have the same norms as $A$ and $A_0$ respectively.
  For $\lambda \partial_\lambda A(\lambda)$ and $\lambda \partial_\lambda A_0(\lambda)$ the proof is similar. We compute
  $$
  \partial_\lambda A(\lambda) = -\frac{1}{3\lambda^2}\Delta q\big(\frac{x}{\lambda}\big) - \frac{1}{3\lambda^3}x\cdot\grad \Delta q\big(\frac{x}{\lambda}\big)
  - \frac{1}{\lambda^2}x\cdot\grad^2 q\big(\frac{x}{\lambda}\big)\cdot\grad.
  $$
  Since $\grad q(x)$, $\grad^2 q(x)$ and $\grad^3 q(x)$ are continuous and of compact support, boundedness follows.

  In \eqref{eq:A-by-parts} both sides are continuous for the $X^1$ topology, hence we may assume that $h_1, h_2 \in C_0^\infty$.
  We may also assume without loss of generality that $\lambda = 1$.
  Observe that for any $h \in C_0^\infty$ there holds $h\cdot f(h) = 3\cdot F(h)$ and $\grad h\cdot f(h) = \grad F(h)$, hence
  \begin{equation}
    \label{eq:A-by-parts-2}
    \la Ah, f(h)\ra = \int\big(\frac 13 \Delta q\cdot h + \grad q\cdot \grad h\big)f(h)\ud x = \int \Delta q\cdot F(h) + \grad q\cdot \grad F(h)\ud x = 0.
  \end{equation}
  Using this for $h = h_1 + h_2$ and for $h = h_1$, \eqref{eq:A-by-parts} is seen to be equivalent to
  \begin{equation}
    \label{eq:A-by-parts-3}
    \la A h_2, f(h_1)\ra + \la A h_1, f'(h_1)h_2\ra = 0.
  \end{equation}
  Expanding the left side using the definition of $A$ we obtain
  \begin{equation*}
    \begin{aligned}
    \la A h_2, f(h_1)\ra + \la A h_1, f'(h_1)h_2\ra &= \int \frac 13 \Delta q\cdot h_2\cdot f(h_1) + \grad q\cdot \grad h_2\cdot f(h_1)\ud x \\
    &+\int \frac 13 \Delta q\cdot h_1\cdot f'(h_1)\cdot h_2 + \grad q\cdot \grad h_1\cdot f'(h_1)\cdot h_2\ud x
  \end{aligned}
  \end{equation*}
  Integrating by parts the term containing $\grad h_2$ and using the formulas $h_1\cdot f'(h_1) = 2f(h_1)$ and $\grad h_1\cdot f'(h_1) = \grad f(h_1)$, this can be rewritten as
  \begin{equation*}
    \big\la h_2, \frac 13 \Delta q\cdot f(h_1) - \Delta q\cdot f(h_1) - \grad q\cdot\grad f(h_1) + \frac 23 \Delta q\cdot f(h_1) + \grad q\cdot \grad f(h_1)\big\ra = 0,
  \end{equation*}
  which proves \eqref{eq:A-by-parts-3}.

  Inequality \eqref{eq:A-pohozaev} follows easily from \ref{enum:approx}, \ref{enum:convex} and \ref{enum:bilapl},
  once we check the following identity (valid in any dimension $N$, and used here for $N = 6$):
  \begin{equation}
    \label{eq:aux-pohozaev}
    \begin{aligned}
    &\int \Delta h(x)\cdot\big(\frac{1}{2\lambda}\Delta q\big(\frac{x}{\lambda}\big)h(x) + \grad q\big(\frac{x}{\lambda}\big)\cdot \grad h(x)\big)\ud x \\
    &= -\frac{1}{4\lambda}\int(\Delta^2 q)\big(\frac{x}{\lambda}\big)h(x)^2 \ud x - \frac{1}{\lambda}\int\sum_{i, j = 1}^N\partial_{ij}q\big(\frac{x}{\lambda}\big)\partial_i h(x)\partial_j h(x)\ud x.
    \end{aligned}
  \end{equation}
  Without loss of generality we can assume that $\lambda = 1$ (it suffices to replace $q$ by its rescaled version).
  By a density argument, we can also assume that $q, h \in C_0^\infty$ (we use here the fact that $q \in C^{3, 1}$), and \eqref{eq:aux-pohozaev} follows from integration by parts:
  \begin{equation}
    \begin{aligned}
      &\int \frac 12 \Delta h\cdot \Delta q\cdot h + \Delta h\cdot \grad q\cdot \grad h\ud x =
      \int\sum_{i, j = 1}^N \big( \frac 12 \partial_{ii}h\cdot\partial_{jj}q\cdot h + \partial_{ii}h\cdot \partial_j q\cdot \partial_j h \big)\ud x \\
      &= \int-\frac 12 \sum_{i, j}\partial_i h(\partial_{jj}q\partial_i h + \partial_{ijj}q \cdot h)+\sum_i \frac 12 \partial_i((\partial_i h)^2)\partial_i q \\
&+ \sum_{i\neq j}\big(-\frac 12 \partial_j(\partial_i h)^2 \partial_j q - \partial_{ij}q\partial_i h\partial_j h\big)\ud x \\
      &= \int-\frac 12 \sum_{i, j}\big(\partial_{jj}q(\partial_i h)^2 + \frac 12 \partial_{iijj}q\cdot h^2\big) - \frac 12 \sum_i \partial_{ii}q(\partial_i h)^2 \\
&+\frac 12 \sum_{i\neq j}\partial_{jj}q(\partial_i h)^2 - \sum_{i\neq j}\partial_{ij}q\partial_i h\partial_j h\ud x \\
      &= \int -\frac 14 \sum_{i, j}\partial_{iijj}q\cdot h^2 -\sum_{i, j}\partial_{ij}q\partial_i h\partial_j h\ud x.
    \end{aligned}
  \end{equation}
  
  Estimate \eqref{eq:L0-A0} is invariant by rescaling, hence we can assume that $\lambda = 1$. For $|x| \leq R$ we have $A_0\Lambda W(x) = \Lambda_0 \Lambda W(x)$.
  From \ref{enum:gradlap} in Lemma~\ref{lem:fun-a} we get $|A_0\Lambda W(x)| + |\Lambda_0 \Lambda W(x)| \lesssim |x|^{-4}$ for $|x| \geq R$,
  with a constant independent of $R$. Thus $\|\Lambda_0 \Lambda W - A_0 \Lambda W\|_{L^2} \leq c_0$ if $R$ is large enough.

  A similar reasoning yields $\|\Lambda W_\uln\lambda - A(\lambda)W_\lambda\|_{L^3} \ll \lambda^{-1}$ as $R \to +\infty$.
  Since $b(t) \sim \lambda(t)$, this gives
  \begin{equation}
    \label{eq:fidot-1}
    \|\dot \varphi + b A(\lambda)W_\lambda\|_{L^3} \leq \frac{c_0}{3}, \qquad \text{if }R\text{ is large enough.}
  \end{equation}
  From \ref{enum:support} in Lemma~\ref{lem:fun-a} it follows that $\supp (A(\lambda)W_\mu) \subset B(0, \wt R\cdot \lambda)$.
  Since $\|A(\lambda)W_\mu\|_{L^\infty} \lesssim_q \frac{1}{\lambda}$, we have
  \begin{equation}
    \label{eq:fidot-2}
    \|b A(\lambda)W_\mu\|_{L^3} \leq \frac{c_0}{3}, \qquad \text{if }|T_0|\text{ is large enough.}
  \end{equation}
  
  To finish the proof, we have to check that
  \begin{equation}
    \label{eq:fidot-3}
    \big\|b A(\lambda)\big(\chi\cdot \big(\lambda(t)^2 P_{\lambda(t)} + b(t)^2 Q_{\lambda(t)}\big)\big)\big\|_{L^3} \leq \frac{c_0}{3}, \qquad\text{if }|T_0|\text{ is large enough.}
  \end{equation}
  We have the bound $\|(\chi+ |\grad \chi|) P_\lambda\|_{L^3} \lesssim \int_0^{\frac{2}{\lambda}}\frac{1}{r^6}r^5\ud r \lesssim |\log\lambda|$ (similarly with $Q_\lambda$),
  hence $\|(\chi+|\grad \chi|)(\lambda^2 P_\lambda + b^2 Q_\lambda)\|_{L^3} \lesssim (\lambda^2 + b^2)|\log \lambda| \ll 1$ as $|T_0| \to +\infty$.
  We have also $\|\grad(\lambda^2 P_\lambda + b^2 Q_\lambda)\|_{L^3} \lesssim \frac{1}{\lambda}(\lambda^2 + b^2) \ll 1$.
  Since $q$ is smooth and constant at infinity, we have $\big|b\cdot\big(\frac{1}{\lambda}\Delta q\big(\frac{x}{\lambda}\big)+\grad q\big(\frac{x}{\lambda}\big)\big)\big| \lesssim 1$.
  The constant depends on the choice of the function $q$, but this is not a concern here. We obtain
  \begin{equation}
    \begin{aligned}
    &\big\|b A(\lambda)\big(\chi\cdot \big(\lambda(t)^2 P_{\lambda(t)} + b(t)^2 Q_{\lambda(t)}\big)\big)\big\|_{L^3} \lesssim \big|b\cdot\big(\frac{1}{\lambda}\Delta q\big(\frac{x}{\lambda}\big)+\grad q\big(\frac{x}{\lambda}\big)\big)\big|\cdot \\
    &\cdot\big(\|(\chi+|\grad \chi|)(\lambda^2 P_\lambda + b^2 Q_\lambda)\|_{L^3}+ \|\grad(\lambda^2 P_\lambda + b^2 Q_\lambda)\|_{L^3}\big) \ll 1,
  \end{aligned}
  \end{equation}
  hence \eqref{eq:fidot-3}

  Putting together \eqref{eq:fidot-1}, \eqref{eq:fidot-2} and \eqref{eq:fidot-3} we get \eqref{eq:L-A}.

  In order to prove \eqref{eq:approx-potential}, note first that boundedness of $\Delta q$ and \eqref{eq:f-phi-g} yield
  \begin{equation}
    \Big|\int\frac 16 \Delta q\big(\frac{x}{\lambda}\big)(f(\varphi + g) - f(\varphi))g\ud x - \int \frac 16 \Delta q\big(\frac{x}{\lambda}\big)(f'(W_\mu) + f'(W_\lambda))g^2 \ud x\Big| \ll \eee^{-3\kappa|t|}.
  \end{equation}
  Since $\grad q$ is of compact support, we have $\Big| \int\frac 16 \Delta q\big(\frac{x}{\lambda}\big)W_\mu g^2\ud x\Big| \ll \|g\|_{\dot H^1}^2$.
  As in the proof of \eqref{eq:L-A}, on can show that $\big\|\frac 16\Delta q\big(\frac{x}{\lambda}\big)f'(W_\lambda) - f'(W_\lambda)\big\|_{L^3} \to 0$ as $R \to +\infty$.
  This finishes the proof.
\end{proof}
For $t \in [T, T_0]$ we define:
\begin{itemize}
  \item the \emph{nonlinear energy functional}
    $$
    \begin{aligned}
    I(t) &:= \int \frac 12 |\dot g(t)|^2 + \frac 12 |\grad g(t)|^2 -\big(F(\varphi(t) + g(t)) - F(\varphi(t)) - f(\varphi(t))g(t)\big)\ud x \\
      &= E(\bs \varphi(t) + \bs g(t)) - E(\bs \varphi(t)) - \la \vD E(\bs \varphi(t), \bs g(t))\ra,
    \end{aligned}
    $$
  \item the \emph{localized virial functional} $$J(t) := \int \dot g(t)\cdot A_0(\lambda(t))g(t)\ud x,$$
  \item the \emph{mixed energy-virial functional} $$H(t) := I(t) + b(t)J(t).$$
\end{itemize}
From \eqref{eq:point-F2} we have
\begin{equation}
  \label{eq:I-devel}
  \big|I(t) - \frac 12 \la \vD^2 E(\bs \varphi(t))\bs g(t), \bs g(t)\ra\big| \lesssim \|\bs g(t)\|_\cE^3.
\end{equation}
Note that $H(t)$ depends on the function $q(x)$ used in the definition of $A_0(\lambda)$, see \eqref{eq:opA0}.
From the first statement in Lemma~\ref{lem:op-A} we deduce that
\begin{equation}
  \label{eq:J-devel}
  |J(t)| \lesssim_q \|\bs g(t)\|_\cE^2,
\end{equation}
where the constant in the inequality depends on the choice of the function $q(x)$.
Thus \eqref{eq:b-bound0} and \eqref{eq:assumption-g} imply that for $t \leq T_0$ with $|T_0|$ large enough there holds
\begin{equation}
  \label{eq:H-devel}
  \big|H(t) - \frac 12 \la \vD^2 E(\bs \varphi(t))\bs g(t), \bs g(t)\ra\big| \leq c\|\bs g(t)\|_\cE^2,
\end{equation}
with $c > 0$ as small as we wish.

\subsection{Energy estimates via the mixed energy-virial functional}
\label{ssec:energy}
\begin{lemma}
  \label{lem:bootstrap}
  Let $c_1 > 0$. If $C_0$ is sufficiently large, then there exists a function $q(x)$ and $T_0 < 0$ with the following property.
  If $T_1 < T_0$ and \eqref{eq:lambda-bound0}, \eqref{eq:assumption-g}, \eqref{eq:assumption-a} hold for $t \in [T, T_1]$, then for $t \in [T, T_1]$ there holds
  \begin{equation}
    \label{eq:boot-dt}
    H'(t) \leq c_1 \cdot C_0^2 \cdot \eee^{-3\kappa|t|}.
  \end{equation}
\end{lemma}
This lemma is the key step in proving Proposition~\ref{prop:bootstrap}. We will postpone its slightly technical proof.

\begin{proof}[Proof of Proposition~\ref{prop:bootstrap} assuming Lemma~\ref{lem:bootstrap}]
  We first show \eqref{eq:bootstrap-l}. From \eqref{eq:mod} and \eqref{eq:assumption-g} we obtain
  \begin{equation}
    \label{eq:bootstrap-l-1}
    |\mu(t) - 1| = |\mu(t) - \mu(T)| \lesssim \int_{-\infty}^t C_0\cdot\eee^{-\frac 32 \kappa_0|\tau|}\ud \tau \lesssim C_0\cdot \eee^{-\frac 32 \kappa_0|t|}.
  \end{equation}

  Again from \eqref{eq:mod} and \eqref{eq:assumption-g} we have $|\lambda'(t) - b(t)| \lesssim C_0 \cdot \eee^{-\frac 32 \kappa|t|}$.
  Multiplying by $b'(t) = \kappa^2 \cdot\frac{\lambda(t)}{\mu(t)^2} \sim \eee^{-\kappa|t|}$, cf. \eqref{eq:b-def} and \eqref{eq:lambda-bound0}, we obtain
  $\big|\dd t\big(b(t)^2 - \kappa^2 \cdot\frac{\lambda(t)^2}{\mu(t)^2}\big)\big| \lesssim C_0\cdot \eee^{-\frac 52 \kappa|t|}$.
  Since $b(T) = \kappa \cdot\lambda(T)$ and $\mu(T) = 1$, this yields $\big|b(t)^2 - \kappa^2 \cdot\frac{\lambda(t)^2}{\mu(t)^2}\big| \lesssim C_0 \cdot\eee^{-\frac 52 \kappa|t|}$.
  But $b(t) + \kappa\cdot\frac{\lambda(t)}{\mu(t)} \sim \eee^{-\kappa|t|}$, see \eqref{eq:lambda-bound0} and \eqref{eq:b-bound0}, hence
  \begin{equation}
    \label{eq:b-lambda}
    \big|b(t) - \kappa \cdot\frac{\lambda(t)}{\mu(t)}\big| \lesssim C_0\cdot \eee^{-\frac 32\kappa|t|}.
  \end{equation}
  Bound \eqref{eq:bootstrap-l-1} implies that $\big|\frac{\lambda(t)}{\mu(t)}-\lambda(t)\big| \ll \eee^{-\frac 32 \kappa|t|}$, thus \eqref{eq:b-lambda} yields
  $|\lambda'(t) - \kappa \cdot \lambda(t)| \lesssim C_0 \cdot\eee^{-\frac 32 \kappa|t|}$.
  Integrating and using $\lambda(T) = \frac{1}{\kappa}\eee^{-\kappa|T|}$ we obtain \eqref{eq:bootstrap-l}.


  We turn to the proof of \eqref{eq:bootstrap-g}.  
  The initial data at $t = T$ satisfy $\|\bs g(T)\|_\cE \lesssim \eee^{-\frac 32 \kappa|T|}$, thus \eqref{eq:H-devel} implies that $H(T) \lesssim \eee^{-3\kappa|T|}$.
  If $C_0$ is large enough, then integrating \eqref{eq:boot-dt} we get $H(t) \leq c\cdot C_0^2\cdot \eee^{-3\kappa|t|}$, with a small constant $c$.
  Now \eqref{eq:H-devel} implies
  \begin{equation}
    \label{eq:D2E-coer}
    \la \vD^2 E(\bs \varphi(t))\bs g(t), \bs g(t)\ra \lesssim c\cdot C_0^2\cdot \eee^{-3 \kappa |t|},\qquad \text{with }c\text{ small.}
  \end{equation}
  Since $\|\bs \varphi(t) - \bs W_{\lambda(t)}\|_\cE$ is small, Lemma~\ref{lem:bulles-coer} together with
  \eqref{eq:g-orth}, \eqref{eq:assumption-a}, \eqref{eq:a-1m} and \eqref{eq:a-reste} yields
  \begin{equation}
    \|\bs g\|_\cE^2 \lesssim (cC_0^2 + 1)\eee^{-3\kappa|t|},
  \end{equation}
  Eventually enlarging $C_0$, we obtain \eqref{eq:bootstrap-g}, if $c$ in \eqref{eq:D2E-coer} is taken sufficiently small.
\end{proof}
\begin{proof}[Proof of Lemma~\ref{lem:bootstrap}]
  In this proof we say that a term is negligible if its contribution is $\leq c\cdot C_0^2 \cdot \eee^{-3\kappa |t|}$.
  We write $\text{Value}_1 \simeq \text{Value}_2$ if $|\text{Value}_1 - \text{Value}_2| \leq c\cdot C_0^2\cdot \eee^{-3\kappa|t|}$.
  The order of choosing the parameters is the following: we will first choose $q(x)$ independently of $C_0$, then $C_0$, which may depend on $q(x)$,
  and finally $|T_0|$.
  \paragraph{\textbf{Step 1} (Derivative of the energy functional)}
  Using the definition of $I(t)$, the conservation of energy, formulas \eqref{eq:psi}, \eqref{eq:g-diff} and self-adjointness of $\vD^2 E(\bs \varphi)$ we compute:
  \begin{equation*}
    \begin{aligned}
      I'(t) &= 0 - \la \vD E(\bs \varphi), \partial_t \bs \varphi\ra-\la \vD^2 E(\bs \varphi)\partial_t \bs \varphi, \bs g\ra - \la \vD E(\bs \varphi), \partial_t \bs g\ra \\
      &= -\la \vD E(\bs \varphi), J\circ\vD E(\bs \varphi) + \bs \psi\ra - \la J\circ\vD E(\bs \varphi) + \bs \psi, \vD^2 E(\bs \varphi)\bs g\ra \\
        &- \la \vD E(\bs \varphi), J\circ(\vD E(\bs \varphi + \bs g) - \vD E(\bs \varphi)) - \bs \psi\ra \\
      &=-\la \vD^2 E(\bs \varphi)\bs \psi, \bs g\ra - \la \vD E(\bs \varphi), J\circ(\vD E(\bs \varphi + \bs g) - \vD E(\bs \varphi) - \vD^2 E(\bs \varphi)\bs g)\ra \\
      &= \la (\Delta + f'(\varphi))\psi, g\ra - \la \dot \psi, \dot g\ra - \la \dot \varphi, f(\varphi + g) - f(\varphi) - f'(\varphi)g\ra.
    \end{aligned}
  \end{equation*}
  The first term is $\lesssim C_0\eee^{-3\kappa|t|}$, due to \eqref{eq:psi-lin} and \eqref{eq:assumption-g}, hence it is negligible (by enlarging $C_0$ if necessary).
  Inequality \eqref{eq:psi1} implies that the second term can be replaced by $-\frac{b}{\lambda}(\lambda'-b)\la \Lambda_0 \Lambda W_\uln\lambda, \dot g\ra$,
  which in turn can be replaced by $-b(\lambda'-b)\la A_0(\lambda) \Lambda W_\uln\lambda, \dot g\ra$, thanks to \eqref{eq:L0-A0}.
  From \eqref{eq:L-A} we infer that the third term can be replaced by $b\cdot \la A(\lambda)\varphi, f(\varphi + g) - f(\varphi) - f'(\varphi)g\ra$.
  Using formula \eqref{eq:A-by-parts} with $h_1 = \varphi$ and $h_2 = g$ we obtain
  \begin{equation}
    \label{eq:dtI}
      I'(t) \simeq -b(\lambda'-b)\cdot \la A_0(\lambda)\Lambda W_\uln\lambda, \dot g\ra - b\cdot\la A(\lambda)g, f(\varphi + g) - f(\varphi)\ra.
  \end{equation}
  \paragraph{\textbf{Step 2} (Derivative of the virial functional)}
  We compute:
  \begin{equation}
    \label{eq:dtJ}
    \begin{aligned}
      (bJ)'(t) &= b'\int\dot g\cdot A_0(\lambda)g\ud x + b\lambda'\int\dot g \cdot \partial_\lambda A_0(\lambda)g\ud x \\
      &+ b\int \dot g\cdot A_0(\lambda)(\dot g - \psi)\ud x + b\int (\Delta g + f(\varphi + g) - f(\varphi) - \dot \psi)\cdot A_0(\lambda) g\ud x.
    \end{aligned}
  \end{equation}
  The first two terms are negligible thanks to Lemma~\ref{lem:op-A}.
  Consider the third term on the right in \eqref{eq:dtJ}. An integration by parts yields $\int \dot g\cdot A_0(\lambda)\dot g\ud x = 0$.
  Since $A_0(\lambda): \dot H^1 \to L^2$ is a bounded operator, from \eqref{eq:psi0} we see that
  $$
  b\int\dot g\cdot A_0(\lambda)\psi\ud x \simeq -b(\lambda'-b)\cdot \la A_0(\lambda)\Lambda W_\uln\lambda, \dot g\ra.
  $$
  Consider the forth term on the right in \eqref{eq:dtJ}. The term $b\int \dot \psi\cdot A_0(\lambda) g\ud x$ is negligible.
  Using \eqref{eq:A-pohozaev} and the fact that $A_0(\lambda) = A(\lambda) + \frac{1}{6\lambda}\Delta q(\frac{\cdot}{\lambda})$ we get
  $$
  \begin{aligned}
  &b\int (\Delta g + f(\varphi + g) - f(\varphi))\cdot A_0(\lambda)g\ud x \leq b\cdot\la A(\lambda)g, f(\varphi + g) - f(\varphi)\ra \\
  &-\frac{b}{\lambda}\int_{|x|\leq R\lambda}|\grad g|^2 \ud x + \frac{b}{\lambda}\int \frac 16\Delta q\big(\frac{\cdot}{\lambda}\big)(f(\varphi + g) - f(\varphi))g\ud x + cC_0^2 \eee^{-3\kappa|t|},
\end{aligned}
  $$
  with a small constant $c$.
  Putting everything back together and using \eqref{eq:approx-potential} we get
  \begin{equation}
    \label{eq:dtJ-2}
    \begin{aligned}
    (bJ)'(t) &\leq b(\lambda'-b)\cdot \la A_0(\lambda)\Lambda W_\uln\lambda, \dot g\ra + b\cdot\la A(\lambda)g, f(\varphi + g) - f(\varphi)\ra \\
    &+ \frac{b}{\lambda}\Big(-\int_{|x|\leq R\lambda}|\grad g|^2 \ud x + \int f'(W_\lambda)g^2\ud x\Big) + cC_0^2\eee^{-3\kappa|t|}.
    \end{aligned}
  \end{equation}

  \paragraph{\textbf{Step 3} (Localized coercivity)}
  
  Taking the sum of \eqref{eq:dtI} and \eqref{eq:dtJ} we obtain
  \begin{equation}
    H'(t) \leq \frac{b}{\lambda}\Big(-\int_{|x|\leq R\lambda}|\grad g|^2 \ud x + \int f'(W_\lambda)g^2\ud x\Big) + cC_0^2\eee^{-3\kappa|t|}.
  \end{equation}
  Recall that $\big|\big\la \frac{1}{\lambda}\cY_\uln\lambda, g\big\ra\big| \lesssim |a_1^+| + |a_1^-|$, hence \eqref{eq:coer-lin-coer-2} (after rescaling)
  together with \eqref{eq:g-orth}, \eqref{eq:assumption-a} and \eqref{eq:a-1m} imply that
  \begin{equation}
    -\int_{|x|\leq R\lambda}|\grad g|^2 \ud x + \int f'(W_\lambda)g^2\ud x \lesssim (cC_0^2 + 1)\eee^{-3\kappa|t|},
  \end{equation}
  with $c > 0$ as small as we wish (by taking $R$ large enough). Enlarging $C_0$ if necessary we arrive at \eqref{eq:boot-dt}.
\end{proof}

\subsection{Shooting argument and passing to a limit}
\label{ssec:limit}
We are ready to give a proof of the main result of the paper, following a well-known scheme introduced in \cite{Merle90} and \cite{Martel05}.
\begin{proof}[Proof of Theorem~\ref{thm:deux-bulles}]
  ~
  \paragraph{\textbf{Step 1}} 
  Let $T_n$ be a decreasing sequence converging to $-\infty$. For $n$ large and
  $$a_0 \in \cA := \big[-\frac 12 \eee^{-\frac 32 \kappa|T_n|}, \frac 12 \eee^{-\frac 32 \kappa|T_n|}\big],$$
    let $\bs u_n^{a_0}(t): [T_n, T_+) \to \cE$ denote the solution of \eqref{eq:nlw0} with initial data \eqref{eq:data-at-T}. We will prove that
      there exists $a_0$ such that $T_+ > T_0$ and for $\bs u = \bs u_n^{a_0}$ inequalities \eqref{eq:bootstrap-g}, \eqref{eq:bootstrap-l} hold for $t \in [T_n, T_0]$.

      Suppose that this is not the case. For each $a_0 \in \cA$, let $T_1 = T_1(a_0)$ be the last time such that \eqref{eq:bootstrap-g} and \eqref{eq:bootstrap-l}
      hold for $t \in [T_n, T_1)$. Since $\{\bs \varphi(t): t \in [T_n, T_1]\}$ is a compact set, Corollary~\ref{cor:quitte-compact} implies that
      $T_+ > T_1$. Suppose that $|a_1^+(T_1)| < \eee^{-\frac 32 \kappa|T_1|}$. Then Proposition~\ref{prop:bootstrap} would imply that \eqref{eq:bootstrap-g}
      and \eqref{eq:bootstrap-l} hold on some neighborhood of $T_1$, contradicting its definition. Thus $a_1^+(T_1) = \eee^{-\frac 32 \kappa|T_1|}$
      or $a_1^+(T_1) = -\eee^{-\frac 32 \kappa|T_1|}$. Let $\cA_+ \subset \cA$ be the set of $a_0 \in \cA$ which lead to $a_1^+(T_1) = \eee^{-\frac 32 \kappa|T_1|}$
      and let $\cA_- \subset \cA$ be the set of $a_0 \in \cA$ which lead to $a_1^+(T_1) = -\eee^{-\frac 32 \kappa|T_1|}$.
      We have proved that $\cA = \cA_+ \cup \cA_-$. We will show that $\cA_+$ and $\cA_-$ are open sets, which will lead to a contradiction since $\cA$ is connected.

      Let $a_0 \in \cA_+$. This implies that there exists the first $T_2$ such that $a_1^+(T_2) > \frac 34 \eee^{-\frac 32 \kappa|T_2|}$.
      Hence for a solution $\wt{\bs u}(t)$ corresponding to $\wt a_0$ close to $a_0$ we will have (by continuity of the flow) $\wt a_1^+(T_2) > \frac 34 \eee^{-\frac 32 \kappa|T_2|}$
      and $|\wt a_1^+(t)| < \eee^{-\frac 32 \kappa|t|}$ for $t \in [T_n, T_2]$. Suppose that $\wt a_0 \in \cA_-$. Hence there exists the first $T_3 > T_2$ such that
      $\wt a_1^+(T_3) = \frac 34 \eee^{-\frac 32 \kappa|T_3|}$. Estimate \eqref{eq:a-1p} yields $\wt a_1^+(T_3) \gtrsim \eee^{-\frac 12 \kappa|T_3|} \gg \eee^{-\frac 32 \kappa|T_3|}$,
      which is a contradiction. Hence $\cA_+$ is open and analogously $\cA_-$ is open.
      
  \paragraph{\textbf{Step 2}}
  Call $\bs u_n$ the solution found in Step 1. From \eqref{eq:bootstrap-g}, \eqref{eq:bootstrap-l} and \eqref{eq:taille-correct} we deduce that there exists
  a constant $C_1 > 0$ independent of $n$ such that
  \begin{equation}
    \label{eq:bound-seq}
    \|\bs u_n(t) - (W_{\frac{1}{\kappa}\eee^{-\kappa|t|}} - W, -\eee^{-\kappa|t|}\Lambda W_{\uln{\frac{1}{\kappa}\eee^{-\kappa|t|}}})\|_\cE \leq C_1 \cdot \eee^{-\frac 32 \kappa|t|},\qquad \text{for }t\in [T_n, T_0].
  \end{equation}
  The sequence $\bs u_n(T_0)$ is bounded in $\cE$, hence its subsequence (which we still denote $\bs u_n$) has a weak limit $\bs u_0$.
  Let $\bs u(t)$ we the solution of \eqref{eq:nlw0} with the initial data $\bs u(T_0) = \bs u_0$. Let $T < T_0$.
  In view of \eqref{eq:bound-seq}, for large $n$ the sequence $\bs u_n$ satisfies the assumptions of Corollary~\ref{cor:faible} on the time interval $[T, T_0]$,
  hence $\bs u_n(T) \wto \bs u(T)$. Passing to the weak limit in \eqref{eq:bound-seq} finishes the proof.  
\end{proof}
\begin{remark}
  Note that only the instability component $a_1^+(t)$ is treated via a topological argument, whereas $a_2^+(t)$ is estimated directly.
  This depends heavily on the (incidental) fact that the bootstrap bound $\eee^{-\frac 32 \kappa|t|}$ is asymptotically smaller than $\eee^{-\nu|t|}$.
  Were it not the case, we would have to use a topological argument based on the Brouwer fixed point theorem, as in the work of C\^ote, Martel and~Merle \cite{CMM11}.
\end{remark}

\section{Bubble-antibubble for the radial critical Yang-Mills equation}
\label{sec:ym}

\subsection{Notation}
In this section we denote $\|v\|_{L^2(r_1 \leq r \leq r_2)}^2 := 2\pi\int_{r_1}^{r_2}|v(r)|^2 \udr$.
If $r_1$ or $r_2$ is not precised, then it should be understood that $r_1 = 0$, resp. $r_2 = +\infty$.
The corresponding scalar product is denoted $\la v, w\ra := 2\pi\int_0^{+\infty}v(r)\cdot w(r)\udr$.

Recall that $\|v\|_{\cH}^2 := 2\pi \int_0^{+\infty}\big(|\partial_r v(r)|^2 + |\frac 2r v(r)|^2 \big)\udr$.
A change of variables shows that $v(r) \in \cH \Leftrightarrow v(\eee^x) \in H^1(\bR)$, in particular $\|v\|_{L^\infty} \lesssim \|v\|_\cH$
(this change of variables, very helpful in proving coercivity lemmas, can be found in \cite{GKT07}).
Another useful way of understanding the space $\cH$ is to consider the transformation $\wt v(\eee^{i\theta}r) := \eee^{2i\theta}v(r)$,
which is an isometric embedding of $\cH$ in $\dot H^1(\bR^2; \bR^2)$, whose image is given by $2$-equivariant functions in $\dot H^1(\bR^2; \bR^2)$.
Let $\cH^*$ be the dual space of $\cH$ for the pairing $\la\cdot, \cdot\ra$.
The embedding just described identifies $\cH^*$ with the $2$-equivariant distributions in $\dot H^{-1}(\bR^2; \bR^2)$.

We denote $X^0 := L^2 \cap \cH$ and $X^1 := \{v \in \cH: \partial_r v \in \cH \text{ and }\frac 1r v \in \cH\}$.
The generators of the $\cH$-critical and the $L^2$-critical scale change will be denoted respectively $\Lambda := r\partial_r$ and $\Lambda_0 := 1 + r\partial_r$.
\subsection{Linearized equation and formal computation}
\label{ssec:linearise-ym}
Linearizing $-\partial_r^2 u - \frac 1r\partial_r u +\frac{4}{r^2}u(1-u)(1-\frac 12 u)$ around $u = W$ we obtain the operator
\begin{equation}
  L := -\partial_r^2 - \frac 1r \partial_r - \frac{1}{r^2}(2-6(W(r)-1)^2) = -\partial_r^2 - \frac 1r \partial_r + \frac{1}{r^2}(4-6\Lambda W).
\end{equation}

We study solutions behaving like $\bs u(t) \simeq -\bs W + \bs W_{\lambda(t)}$ with $\lambda(t) \to 0$ as $t \to -\infty$.
As in Subsection~\ref{ssec:formal}, we expand
\begin{equation}
  \bs u(t) = -\bs W + \bs U_{\lambda(t)}^{(0)} + b(t)\cdot \bs U^{(1)}_{\lambda(t)} + b(t)^2 \cdot\bs U^{(2)}_{\lambda(t)},
\end{equation}
with $b(t)  = \lambda'(t)$, $\bs U^{(0)} := (W, 0)$ and $\bs U^{(1)} := (0, -\Lambda W)$.
This gives
\begin{equation}
  \label{eq:dttu-ym}
  \partial_t^2 u(t) = -b'(t)(\Lambda W)_\uln{\lambda(t)} + \frac{b(t)^2}{\lambda(t)}(\Lambda_0 \Lambda W)_\uln{\lambda(t)} + \text{lot}.
\end{equation}
Let us restrict our attention to the region $r \leq \sqrt{\lambda(t)}$. We will see that the region $r \geq \sqrt{\lambda(t)}$
will not have much influence on the dynamics of our system. For $r \leq \sqrt\lambda$ we have $W \ll W_\lambda$, hence
\begin{equation}
  4u(1-u)(1-\frac 12 u) \simeq 4W_\lambda(1-W_\lambda)(1-\frac 12 W_\lambda) + (-4 + 6\Lambda W_\lambda)W + \text{lot}.
\end{equation}
Since $\partial_r^2 W +\frac 1r\partial_r W \simeq \frac{4}{r^2}W + \text{lot}$ for $r \leq \sqrt\lambda$, we get
\begin{equation}
  \partial_r^2 u + \frac 1r \partial_r u - \frac{4}{r^2}u(1-u)(1-\frac 12 u) = -\frac{b^2}{\lambda}(LU^{(2)})_\uln\lambda - \frac{1}{r^2}6W_\lambda\cdot W + \text{lot}.
\end{equation}
We can further simplify this using the fact that $W(r)\simeq 2r^2$:
\begin{equation}
  \partial_r^2 u + \frac 1r \partial_r u - \frac{4}{r^2}u(1-u)(1-\frac 12 u) = -\frac{b^2}{\lambda}(LU^{(2)})_\uln\lambda - 12 \Lambda W_\lambda + \text{lot},
\end{equation}
thus \eqref{eq:dttu-ym} yields
\begin{equation}
  LU^{(2)} = -\Lambda_0 \Lambda W + \frac{\lambda}{b^2}(b' - 12\lambda)\Lambda W.
\end{equation}
As in Subsection~\ref{ssec:formal}, we find that the best choice of the formal parameter equations is
\begin{equation}
  \label{eq:formal-param-ym}
  \lambda'(t) = b(t), \qquad b'(t) = \kappa^2 \lambda(t), \qquad\text{with }\kappa := 2\sqrt 3.
\end{equation}
\begin{remark}
  The main term of the interaction is \emph{exactly} cancelled by the term $-b'\Lambda W_\uln\lambda$ for our choice of the parameters.
  We have seen that it is not the case for the power nonlinearity and in the next section we will see that it is not the case either
  for the critical equivariant wave map equation.
\end{remark}

\subsection{Bounds on the error of the ansatz}
\label{ssec:error-ym}
Fix $\cZ \in C_0^\infty((0, +\infty))$ such that
\begin{equation}
  \label{eq:Z-pos-ym}
\int_0^{+\infty} \cZ(r)\cdot \Lambda W(r)\frac{\vd r}{r} > 0.
\end{equation}
By a direct computation we find $L\big(\frac{r^4}{(1+r^2)^2}\big) = \frac{r^2(3-r^2)}{(1+r^2)^3} = -\Lambda_0 \Lambda W$.
Adding a suitable multiple of $\Lambda W(r)$, we obtain a rational function $Q(r)$ such that
\begin{equation}
  \label{eq:profil-ym}
  LQ = -\Lambda_0 \Lambda W,\qquad \int \cZ(r)\cdot Q(r)\frac{\vd r}{r} = 0,\qquad Q(r) \sim r^2\text{ as }r\to 0,\ Q(r) \sim 1\text{ as }r\to +\infty.
\end{equation}

For $\lambda, \mu$ satisfying \eqref{eq:lambda-init0} and \eqref{eq:lambda-bound0} (naturally with $\kappa = 2\sqrt 3$) we define
the approximate solution by the formula
\begin{equation}
  \label{eq:phi-def-ym}
  \begin{aligned}
    \varphi(t) &:= -W_{\mu(t)} + W_{\lambda(t)} + S(t), \\
    \dot \varphi(t) &:= -b(t)\Lambda W_\uln{\lambda(t)},
\end{aligned}
\end{equation}
where
\begin{align}
  \label{eq:b-def-ym}
  b(t) &:= \eee^{-\kappa|T|} + \kappa^2 \int_T^t \frac{\lambda(\tau)}{\mu(\tau)^2}\ud \tau,\qquad&\text{for }t \in [T, T_0], \\
  S(t) &:= \chi\cdot b(t)^2 Q_{\lambda(t)},\qquad&\text{for }t\in [T, T_0]. \label{eq:S-def-ym}
\end{align}
From \eqref{eq:profil-ym} we obtain
\begin{equation}
  \label{eq:taille-Q-ym}
  \|\chi\cdot Q_\uln\lambda\|_{\cH} \lesssim \Big(\int_0^{2/\lambda} 1\cdot \udr\Big)^\frac 12 + \frac{1}{\lambda}\cdot \Big(\int_0^{2/\lambda} ((1+r)^{-1})^2 \udr\Big)^\frac 12 \lesssim \sqrt{|t|}\cdot \eee^{\kappa|t|},
\end{equation}
which implies that
\begin{equation}
  \label{eq:taille-correct-ym}
  \|S(t)\|_{\cH} \ll \eee^{-\frac 32 \kappa|t|}.
\end{equation}

Since $\cZ$ has compact support, \eqref{eq:profil-orth} implies, for sufficiently small $\lambda$,
\begin{equation}
  \label{eq:S-orth-ym}
  \int \cZ_\uln{\lambda}\cdot S(t)\frac{\vd r}{r} = 0.
\end{equation}

We denote $f(u) := -4u(1-u)(1-\frac 12 u)$ and
\begin{equation}
  \label{eq:psi-ym}
  \begin{aligned}
  \bs \psi(t) &= (\psi(t), \dot \psi(t)) := \partial_t  \bs \varphi(t) - \vD E(\bs \varphi(t))  \\
  &= \big(\partial_t \varphi(t) - \dot \varphi(t), \partial_t \dot \varphi(t) - (\partial_r^2 \varphi(t) +\frac 1r\partial_r \varphi(t)+ \frac{1}{r^2}f(\varphi(t)))\big).
\end{aligned}
\end{equation}
We have $f'(u) = 2-6(u-1)^2$.
\begin{remark}
  \label{rem:hidden-lin}
By a direct computation, $f'(W) = -4 + 6\Lambda W$. Thus the potential term of the linearized operator
contains a non-localized part $-4$ and a localized part $6\Lambda W$. These terms need to be treated in different ways.
This is a known issue coming from the fact that $f(u)$ is not really the nonlinearity, as it ``hides'' the linear part near the stable equilibria:
$f(u) \simeq -4u$ near $u = 0$ and $f(u) \simeq 4(2-u)$ near $u = 2$. Sometimes it is convenient to subtract the linear part from $f$,
but here we work simultaneously near $u = 0$ and near $u = 2$, so probably it will be simpler to keep $f$ as it is.
A similar remark could be made in the case of the equivariant wave map equation.
\end{remark}
\begin{lemma}
  \label{lem:psi-ym}
  Suppose that for $t \in [T, T_0]$ there holds $|\lambda'(t)| \lesssim \eee^{-\kappa |t|}$ and $|\mu'(t)| \lesssim \eee^{-\kappa|t|}$. Then
  \begin{align}
    \|\psi(t) - \mu'(t)\frac{1}{\mu(t)}\Lambda W_{\mu(t)}+ (\lambda'(t) - b(t))\frac{1}{\lambda(t)}\Lambda W_{\lambda(t)} \|_{\cH} &\lesssim \eee^{-\frac 32 \kappa |t|}
    , \label{eq:psi0-ym} \\
    \|\dot \psi(t) - \frac{b(t)}{\lambda(t)}(\lambda'(t) - b(t))\Lambda_0 \Lambda W_\uln{\lambda(t)}\|_{L^2} &\lesssim \eee^{-\frac 32 \kappa|t|}, \label{eq:psi1-ym} \\
    \|(-\partial_r^2 - \frac 1r \partial_r - \frac{1}{r^2}f'(\varphi(t)))\psi(t)\|_{\cH^*} &\lesssim \eee^{-\frac 32 \kappa|t|}. \label{eq:psi-lin-ym}
  \end{align}
\end{lemma}
\begin{proof}
  The proof of \eqref{eq:psi0-ym} is the same as the proof of \eqref{eq:psi0}.

In order to prove \eqref{eq:psi1-ym}, we treat separately the regions $r \leq \sqrt\lambda$ and $r \geq \sqrt\lambda$.
We will show that
\begin{equation}
  \label{eq:psi1-nonlin-ym}
  \big\|\frac{1}{r^2}\big(f(\varphi) - f(W_\lambda) + 2\frac{r^2}{\mu^2}f'(W_\lambda) - b^2 f'(W_\lambda)Q_\lambda\big)\big\|_{L^2(r \leq \sqrt\lambda)} \lesssim \eee^{-\frac 32 \kappa|t|}.
\end{equation}
Since $f$ is a polynomial of degree $3$, we have
\begin{equation}
  f(\varphi) = f(W_\lambda) + f'(W_\lambda)(-W_\mu + S) + \frac 12 f''(W_\lambda)(-W_\mu + S)^2 + \frac 16 f'''(W_\lambda)(-W_\mu + S)^3.
\end{equation}
We treat all the terms one by one.
We have $\big|W_\mu(r) - 2\frac{r^2}{\mu^2}\big| \lesssim r^4$.
Since $|f'(W_\lambda)| \lesssim 1$, this implies
\begin{equation}
  \|\frac{1}{r^2}f'(W_\lambda)(W_\mu - 2r^2)\|_{L^2} \lesssim \Big(\int_0^{\sqrt\lambda} r^4\cdot \udr\Big)^{\frac 12}\lesssim \eee^{-\frac 32 \kappa|t|}.
\end{equation}
Since $r$ is small, we have $S(r) = b^2 Q_\lambda(r)$, and the corresponding term is subtracted in \eqref{eq:psi1-nonlin-ym}.

Next, notice that $|W_\mu| + |S| \lesssim r^2$. Since $|f''(W_\lambda)| \lesssim 1$, this implies
\begin{equation}
  \|\frac{1}{r^2}f''(W_\lambda)(-W_\mu + S)^2\|_{L^2} \lesssim \Big(\int_0^{\sqrt\lambda}\big(\frac{1}{r^2}\cdot r^4\big)^2 \udr\Big)^\frac 12 \sim \lambda^\frac 32 \lesssim \eee^{-\frac 32 \kappa|t|}.
\end{equation}
The last term is estimated in a similar way. This finishes the proof of \eqref{eq:psi1-nonlin-ym}.

By a direct computation $\big|\big(\partial_r^2 + \frac 1r \partial_r\big)W_\mu - \frac{8}{\mu^2}\big|\lesssim r^2$, hence
\begin{equation}
  \label{eq:psi1-lin-ym}
  \big\|\big(\partial_r^2 + \frac 1r \partial_r\big)W_\mu - \frac{8}{\mu^2}\big\|_{L^2(r\leq \sqrt\lambda)} \lesssim \eee^{-\frac 32 \kappa|t|}.
\end{equation}
From \eqref{eq:psi1-nonlin-ym}, \eqref{eq:psi1-lin-ym} and the definition of $\varphi(t)$ we have
\begin{equation}
  \label{eq:psi1-all-ym}
  \big\|\big((\partial_r^2 + \frac 1r \partial_r)\varphi + \frac{1}{r^2}f(\varphi)\big)-\big(-\frac{2}{\mu^2}f'(W_\lambda) - \frac{b^2}{\lambda}(LQ)_\uln\lambda - \frac{8}{\mu^2} \big)\big\|_{L^2(r\leq \sqrt\lambda)} \lesssim \eee^{-\frac 32 \kappa|t|}.
\end{equation}
Using \eqref{eq:profil-ym} and the relation $4+f'(W) = 6\Lambda W$ we can rewrite this as
\begin{equation}
  \label{eq:psi1-all2-ym}
  \big\|\big((\partial_r^2 + \frac 1r \partial_r)\varphi + \frac{1}{r^2}f(\varphi)\big)-\big( -\frac{12}{\mu^2}\Lambda W_\lambda + \frac{b^2}{\lambda}\Lambda_0 \Lambda W_\uln\lambda \big)\big\|_{L^2(r\leq \sqrt\lambda)} \lesssim \eee^{-\frac 32 \kappa|t|},
\end{equation}
which is equivalent to \eqref{eq:psi1-ym}, restricted to the region $r \leq \sqrt\lambda$.

Consider the region $r \geq \sqrt\lambda$. Developping $f$ at $2-W_\mu$ we get
\begin{equation}
  f(\varphi) = f(2-W_\mu) + f'(2-W_\mu)(W_\lambda - 2 + S) + \frac 12 f''(2- W_\mu)(W_\lambda - 2 + S)^2 + \frac 16 f'''(2 - W_\mu)(W_\lambda - 2 + S)^3.
\end{equation}
From this and the relations $f(2-W_\mu) = -f(W_\mu)$, $f'(2-W_\mu) = f'(W_\mu)$, we obtain a pointwise bound
\begin{equation}
  \label{eq:psi1-2-nonlin-ym}
  |f(\varphi) + f(W_\mu) + f'(W_\mu)(2-W_\lambda)| \lesssim |S| + |2-W_\lambda|^2.
\end{equation}
Since $|S(r)| \lesssim b^2$, we have
\begin{equation}
  \label{eq:psi1-2-nonlin2-ym}
\|\frac{1}{r^2}S\|_{L^2(r\geq \sqrt\lambda)} \lesssim b^2 \Big(\int_{\sqrt\lambda}^{+\infty}r^{-4}\udr\Big)^\frac 12 \lesssim \eee^{-\frac 32 \kappa|t|}.
\end{equation}
There holds $|2-W_\lambda(r)|^2 \lesssim \frac{\lambda^4}{r^4}$, hence
\begin{equation}
  \label{eq:psi1-2-nonlin3-ym}
\|\frac{1}{r^2}|2-W_\lambda(r)|^2\|_{L^2(r\geq \sqrt\lambda)} \lesssim \lambda^4 \Big(\int_{\sqrt\lambda}^{+\infty}r^{-12}\udr\Big)^{\frac 12} \lesssim \eee^{-\frac 32 \kappa|t|}.
\end{equation}
Since $|W_\mu - 2| \lesssim r^2$, there holds $|f'(W_\mu) + 4| \lesssim r^2$. We also have $|2-W_\lambda| \lesssim \frac{\lambda^2}{r^2}$ and $|(2-W_\lambda) - \frac{2\lambda^2}{r^2}| \lesssim \frac{\lambda^4}{r^4}$, hence
\begin{equation}
\label{eq:psi1-2-nonlin4-ym}
\begin{aligned}
&\big\|\frac{1}{r^2}f'(W_\mu)(2-W_\lambda) + \frac{8\lambda^2}{r^4}\big\|_{L^2(r\geq \sqrt\lambda)} \lesssim \big\|\frac{\lambda^2}{r^2} + \frac{\lambda^4}{r^6}\big\|_{L^2(r\geq \sqrt\lambda)} \\
&\lesssim \lambda^2 \Big(\int_{\sqrt\lambda}^{+\infty}r^{-4}\udr\Big)^{\frac 12} + \lambda^4 \Big(\int_{\sqrt\lambda}^{+\infty}r^{-12}\udr\Big)^{\frac 12} \lesssim \eee^{-\frac 32 |\kappa|t}.
\end{aligned}
\end{equation}
Inserting \eqref{eq:psi1-2-nonlin2-ym}, \eqref{eq:psi1-2-nonlin3-ym} and \eqref{eq:psi1-2-nonlin4-ym} into \eqref{eq:psi1-2-nonlin-ym} we obtain
\begin{equation}
  \label{eq:psi1-2-nonlin5-ym}
  \big\|f(\varphi) + f(W_\mu) - \frac{8\lambda^2}{r^4}\big\|_{L^2(r\geq \sqrt\lambda)} \lesssim \eee^{-\frac 32 \kappa|t|}.
\end{equation}
A direct computation shows that $(\partial_r^2 + \frac 1r \partial_r)W_\lambda(r) = -\frac{8\lambda^2}{r^4} + O(\frac{\lambda^4}{r^6} + \frac{\lambda^8}{r^{10}})$, hence
\begin{equation}
  \label{eq:psi1-2-nonlin6-ym}
  \big\|(\partial_r^2 + \frac 1r \partial_r)W_\lambda + \frac{8\lambda^2}{r^4}\big\|_{L^2(r\geq \sqrt\lambda)} \lesssim \eee^{-\frac 32 \kappa|t|}.
\end{equation}
We have $\partial_rS = \chi'b^2\cdot Q_\lambda + \chi b^2\cdot (Q_\lambda)'$ and $\partial_r^2 S = \chi'' b^2\cdot Q_\lambda + 2\chi' b^2\cdot (Q_\lambda)' + \chi b^2\cdot (Q_\lambda)''$.
There holds $|Q| \lesssim 1$, $|Q'| \lesssim \frac 1r$ and $|Q''| \lesssim \frac{1}{r^2}$,
which implies $|Q_\lambda| \lesssim 1$, $|(Q_\lambda)'| \lesssim \frac 1r$ and $|(Q_\lambda)''| \lesssim \frac{1}{r^2}$.
This gives
\begin{equation}
  \label{eq:Delta-S-ym}
  \begin{gathered}
    \|\chi b^2\cdot(Q_\lambda)''\|_{L^2(r\geq \sqrt\lambda)} + \big\|\frac 1r\chi b^2\cdot(Q_\lambda)'\big\|_{L^2(r\geq \sqrt\lambda)} \lesssim b^2\Big(\int_{\sqrt\lambda}^{+\infty} r^{-4}\udr\Big)^\frac 12 \lesssim \frac{b^2}{\sqrt\lambda} \lesssim \eee^{-\frac 32\kappa|t|}, \\
    \|\chi' b^2\cdot(Q_\lambda)'\|_{L^2(r\geq \sqrt\lambda)} +\| \chi''b^2 \cdot Q_\lambda\|_{L^2(r\geq \sqrt\lambda)} + \big\|\frac 1r \chi'b^2\cdot Q_\lambda\big\|_{L^2(r\geq \sqrt\lambda)} \lesssim b^2 \ll \eee^{-\frac 32 \kappa|t|},
  \end{gathered}
\end{equation}
hence $\|(\partial_r^2 + \frac 1r \partial_r)S\|_{L^2(r\geq \sqrt\lambda)} \lesssim \eee^{-\frac 32 \kappa|t|}$.
Together with \eqref{eq:psi1-2-nonlin5-ym} and \eqref{eq:psi1-2-nonlin6-ym} this proves that
\begin{equation}
  \|(\partial_r^2 + \frac 1r \partial_r)\varphi + \frac{1}{r^2}f(\varphi)\|_{L^2(r\geq \sqrt\lambda)} \lesssim \eee^{-\frac 32 \kappa|t|}.
\end{equation}
Since $\|\Lambda W_\uln\lambda\|_{L^2(r\geq \sqrt\lambda)} + \|\Lambda_0 \Lambda W_\uln\lambda\|_{L^2(r\geq \sqrt\lambda)} \lesssim (\int_{1/\sqrt\lambda}^{+\infty}\frac{1}{r^4}\udr)^{\frac 12}
\lesssim \sqrt\lambda$, the other terms appearing in \eqref{eq:psi1-ym} are $\lesssim \eee^{-\frac 32 \kappa|t|}$. This finishes the proof of \eqref{eq:psi1-ym}.

The proof of \eqref{eq:psi-lin-ym} is very similar to the proof of \eqref{eq:psi-lin}, hence we will just indicate the differences.
Since $\|W_\lambda\|_{L^\infty} + \|W_\mu\|_{L^\infty} \lesssim 1$, we have $|f'(-W_\mu + W_\lambda) - f'(W_\lambda)| \lesssim W_\mu$
and $|f'(-W_\mu + W_\lambda) - f'(W_\mu)| = |f'(W_\mu + (2-W_\lambda)| - f'(W_\mu)| \lesssim 2- W_\lambda$.
Next, we check that $\|\frac{1}{r^2}W_\mu\cdot \Lambda W_\lambda\|_{L^1} \lesssim \eee^{-\frac 32 \kappa|t|}$ (recall that $\cH\subset L^\infty$, hence $L^1(\bR^2) \subset \cH^*$).
To do this, we consider separately $r \leq 1$ and $r \geq 1$:
\begin{gather}
  \|\frac{1}{r^2}W_\mu\cdot \Lambda W_\lambda\|_{L^1(r \leq 1)} \lesssim \|\Lambda W_\lambda\|_{L^1(r \leq 1)} = \lambda^2 \|\Lambda W\|_{L^1(r\leq 1/\lambda)} \lesssim \lambda^2 |\log \lambda| \ll \eee^{-\frac 32 \kappa|t|}, \\
  \|\frac{1}{r^2}W_\mu\cdot \Lambda W_\lambda\|_{L^1(r \geq 1)} \lesssim \|\Lambda W_\lambda\|_{L^\infty(r \geq 1)} = \big|\Lambda W\big(\frac{1}{\lambda}\big)\big| \sim \lambda^2 \ll \eee^{-\frac 32 \kappa|t|}.
\end{gather}
Finally, we check that $\|\frac{1}{r^2}(2-W_\lambda) \cdot \Lambda W_\mu\|_{L^1} \lesssim \eee^{-\frac 12 \kappa|t|}$, again dividing into $r \leq 1$ and $r \geq 1$:
\begin{gather}
  \|2-W_\lambda\|_{L^1(r \leq 1)} = \big\|\frac{1}{1 + (r/\lambda)^2}\big\|_{L^1(r \leq 1)} = \lambda^2 \big\|\frac{1}{1+r^2}\big\|_{L^1(r\leq 1/\lambda)} \sim \lambda^2|\log \lambda| \ll \eee^{-\frac 32 \kappa|t|}, \\
  \|2-W_\lambda\|_{L^1(r \geq 1)} = \frac{2\lambda^2}{1+\lambda^2} \ll \eee^{-\frac 32 \kappa|t|}.
\end{gather}
This allows to conclude, since $\|f'(\varphi) - f'(-W_\mu + W_\lambda)\|_{L^\infty} \lesssim \|S\|_{L^\infty} \lesssim \eee^{-\frac 32 \kappa|t|}$.
\end{proof}

\subsection{Modulation}
\label{ssec:mod-ym}
Having defined the approximate solution $\bs \varphi(t)$, we will now analyse exact solutions close to $\bs \varphi(t)$.
The initial data are
\begin{equation}
  \label{eq:data-at-T-ym}
  \bs u(T) = \big({-}W + W_{\frac{1}{\kappa}\eee^{-\kappa|T|}}, -\eee^{-\kappa|T|}\Lambda W_{\uln{\frac{1}{\kappa}\eee^{-\kappa|T|}}}\big)
\end{equation}
(there is no linear instability in the case of the Yang-Mills equation).

Similarly as in Subsection~\ref{ssec:mod}, we choose the modulation parameters $\lambda(t)$ and $\mu(t)$ which verify
\begin{equation}
  \label{eq:orth-ym}
  \int \cZ_\uln\mu\cdot\big(u(t) - (-W_{\mu(t)} + W_{\lambda(t)})\big)\frac{\vd r}{r} = 0, \qquad \int \cZ_\uln\lambda\cdot\big(u(t) - (-W_{\mu(t)} + W_{\lambda(t)})\big)\frac{\vd r}{r} = 0.
\end{equation}
We define $\bs g(t)$ by
\begin{equation}
  \label{eq:g-def-ym}
  \bs u(t) = \bs \varphi(t) + \bs g(t).
\end{equation}
It satisfies, cf. \eqref{eq:g-orth},
\begin{equation}
  \label{eq:g-orth-ym}
  \big\la \frac{1}{\lambda(t)}\cZ_{\uln{\lambda(t)}}, g(t)\big\ra = 0,\qquad \big|\big\la \frac{1}{\mu(t)}\cZ_{\uln{\mu(t)}}, g(t)\big\ra\big| \leq c\cdot \eee^{-\frac 32\kappa|t|}.
\end{equation}

The functions $\lambda(t)$ and $\mu(t)$ are $C^1$ and
\begin{equation}
  \label{eq:mod-ym}
  |\lambda'(t) - b(t)| + |\mu'(t)| \lesssim \|\bs g(t)\|_\cE + c\cdot\eee^{-\frac 32 \kappa|t|}
\end{equation}
with $c > 0$ arbitrarily small, cf. \eqref{eq:mod}.
\subsection{Coercivity}
Recall that $f'(W) = -4+6\Lambda W$. In the next lemma, it is useful to separate these two terms, see Remark~\ref{rem:hidden-lin}.
\label{ssec:coer-ym}
\begin{lemma}
  \label{lem:loc-coer-ym}
  There exist constants $c, C > 0$ such that
  \begin{itemize}[leftmargin=0.5cm]
    \item for all $g \in \cH$ there holds
      \begin{equation}
        \label{eq:coer-lin-coer-1-ym}
        \begin{aligned}
        &\int_0^{+\infty}\big((g')^2+\frac{4}{r^2}g^2\big)\udr - \int_0^{+\infty} \frac{6}{r^2}\Lambda W g^2 \udr \\
      &\geq c\int_0^{+\infty}\big((g')^2 +\frac{4}{r^2}g^2\big) \udr -C\Big(\int_0^{+\infty}\cZ\cdot g\frac{\vd r}{r}\Big)^2,
      \end{aligned}
      \end{equation}
    \item if $r_1 > 0$ is large enough, then for all $g \in \cH$ there holds
      \begin{equation}
        \label{eq:coer-lin-coer-2-ym}
        \begin{aligned}
        &(1-2c)\int_0^{r_1}\big((g')^2+\frac{4}{r^2}g^2\big)\udr + c\int_{r_1}^{+\infty}\big((g')^2+\frac{4}{r^2}g^2\big)\udr \\
        &- \int_0^{+\infty}\frac{6}{r^2}\Lambda W g^2 \udr\geq -C\Big(\int_0^{+\infty}\cZ\cdot g\frac{\vd r}{r}\Big)^2,
      \end{aligned}
      \end{equation}
    \item if $r_2 > 0$ is small enough, then for all $g \in \cH$ there holds
      \begin{equation}
        \label{eq:coer-lin-coer-3-ym}
        \begin{aligned}
        &(1-2c)\int_{r_2}^{+\infty}\big((g')^2+\frac{4}{r^2}g^2\big)\udr + c\int_{0}^{r_2}\big((g')^2+\frac{4}{r^2}g^2\big)\udr \\
        &- \int_0^{+\infty}\frac{6}{r^2}\Lambda W g^2 \udr\geq -C\Big(\int_0^{+\infty}\cZ\cdot g\frac{\vd r}{r}\Big)^2,
      \end{aligned}
      \end{equation}
  \end{itemize}
\end{lemma}
\begin{proof}
  Let $\wt g(x) := g(\eee^x)$ and $\wt\cZ(x) := \cZ(\eee^x)$. One computes that $f'(W(\eee^x)) = -4+6\sech^2(x)$, hence \eqref{eq:coer-lin-coer-1-ym} is equivalent to
  \begin{equation}
    \label{eq:coer-L2-version}
    \int_{\bR}(\wt g')^2 +(4-6\sech^2)\wt g^2\ud x \geq c\int_\bR \big((\wt g')^2 + \wt g^2\big)\ud x-C\big(\int_\bR\wt\cZ\cdot \wt g\ud x\big)^2.
  \end{equation}
  This quadratic form corresponds to the classical operator $-\frac{\vd^2}{\vd x^2} + (4 - 6\sech^2)$, for which $0$ is a simple discrete eigenvalue,
  with the eigenspace spanned by $\sech^2$.
  Decompose $\wt g = a\sech^2 + g_1$, with $\int \sech^2\cdot g_1\ud x = 0$. From the Sturm-Liouville theory we obtain
  \begin{equation}
    \int_{\bR}(\wt g')^2 +(4-6\sech^2)\wt g^2\ud x = \int_{\bR}(g_1')^2 +(4-6\sech^2)g_1^2\ud x \gtrsim \int g_1^2\udr.
  \end{equation}
  Let $\sech^2 = b\wt\cZ + (\sech^2)^\perp$ with $\int \wt \cZ\cdot (\sech^2)^\perp\ud x = 0$. Since $\int \wt \cZ\cdot \sech^2 \ud x > 0$, see \eqref{eq:Z-pos-ym},
  we have $\int_\bR \big((\sech^2)^\perp\big)^2\ud x < \int_\bR (\sech^2)^2\ud x$, hence
  \begin{equation}
    \begin{aligned}
    \int g_1^2 \ud x &= a^2\int(\sech^2)^2 \ud x - 2a\int\sech^2\cdot \wt g\ud x + \int \wt g^2 \ud x \\
    &= a^2\int(\sech^2)^2 \ud x - 2ab\int \wt \cZ\cdot \wt g\ud x - 2a\int(\sech^2)^\perp\cdot \wt g\ud x + \int \wt g^2 \ud x \\
    &\geq c\int \wt g^2\ud x - C\big(\int_\bR\wt\cZ\cdot \wt g\ud x\big)^2,
  \end{aligned}
  \end{equation}
  which implies \eqref{eq:coer-L2-version}.

  With the same change of variable, \eqref{eq:coer-lin-coer-2-ym} and \eqref{eq:coer-lin-coer-3-ym} will follow once we prove that
  \begin{equation}
    \label{eq:coer-L2-version-2}
    \begin{aligned}
    &(1-2c)\int_{|x| \leq R}\big((g')^2+4 g^2\big)\ud x + c\int_{|x| \geq R}\big((g')^2+4 g^2\big)\ud x \\
    &- \int_{\bR}6\sech^2g^2\ud x \geq -C\Big(\int_{\bR}\wt \cZ\cdot g\ud x\Big)^2,
  \end{aligned}
  \end{equation}
  provided that $R$ is large enough. To this end, take $\wt \chi(x) := \chi\big(\frac{2x}{R}\big)$ and $\wt h := \wt \chi\cdot \wt g$.
  Since $\wt \cZ$ has compact support and $R$ is large, we have $\int_{\bR}\wt\cZ\cdot \wt h\ud x = \int_{\bR}\wt\cZ\cdot \wt g\ud x$.
  By a standard integration by parts we get $\int_{\bR}(\wt h')^2 \ud x = \int_{\bR}\wt \chi^2(\wt g')^2\ud x + \int_{\bR}\frac 12\big((\wt\chi')^2-\wt\chi\wt\chi''\big)\wt g^2\ud x$.
  We notice that $|(\wt\chi')^2 -\wt\chi\wt\chi''| \lesssim R^{-2}$ is small, in particular for any $c >0$ there holds
  $\int_{\bR}\chi^2(\wt g')^2\ud x \geq \int_{\bR}(\wt h')^2 \ud x - \frac c2\int\big((\wt g')^2 + 4\wt g^2\big)x\ud x$,
  if $R$ is large enough. Applying \eqref{eq:coer-L2-version} with $\wt h$ instead of $\wt g$ and $3c$ instead of $c$ we obtain
  \begin{equation}
    \begin{aligned}
      &(1-3c)\int_{|x| \leq R}\big((\wt g')^2 + 4\wt g^2\big)\ud x - \int_{\bR}6\sech^2\wt \chi^2\wt g^2 \ud x \\
      &\geq (1-3c)\int_{\bR}\chi^2\big((\wt g')^2+4\wt g^2\big)\ud x - \int_{\bR}6\sech^2\wt \chi^2\wt g^2 \ud x \\
      &\geq (1-3c)\int_{\bR}\big((\wt h')^2+4\wt h^2\big)\ud x -\frac c2\int_{\bR}\big((\wt g')^2+4\wt g^2\big) \ud x - \int_{\bR}6\sech^2\wt h^2\ud x \\
      &\geq -\frac c2 \int_{\bR}\big((\wt g')^2+4\wt g^2\big) \ud x -C\Big(\int \wt \cZ\cdot \wt h\ud x\Big)^2 \geq -\frac c2 \int_{\bR}\big((\wt g')^2+4\wt g^2\big) \ud x- C\Big(\int_{\bR}\wt\cZ\cdot \wt g\ud x\Big)^2.
    \end{aligned}
  \end{equation}
  But $6\sech^2 \leq 6\sech^2\wt \chi^2 + 2c$ if $R$ is large enough, and \eqref{eq:coer-L2-version-2} follows.
\end{proof}
\begin{lemma}
  \label{lem:bulles-coer-ym}
  There exists a constant $\eta > 0$ such that if $\frac{\lambda}{\mu} < \eta$ and $\|\bs U - (-\bs W_{\mu} + \bs W_{\lambda})\|_\cE < \eta$,
  then for all $\bs g \in \cE$ there holds
  \begin{equation*}
    \la \vD^2 E(\bs U)\bs g, \bs g\ra +\Big(\int \frac{1}{\lambda}\cZ_\uln\lambda\cdot g\frac{\vd r}{r}\Big)^2 + \Big(\int \frac{1}{\mu}\cZ_\uln\mu\cdot g\frac{\vd r}{r}\Big)^2 \gtrsim \|\bs g\|_\cE^2.
  \end{equation*}\qed
\end{lemma}
The proof is a modification of the proof of Lemma~\ref{lem:bulles-coer} and will be skipped.

\subsection{Definition of the mixed energy-virial functional}
\label{ssec:bootstrap-ym}
\begin{lemma}
  \label{lem:fun-a-ym}
  For any $c > 0$ and $R > 0$ there exists a function $q(r) = q_{c,R}(r) \in C^{3,1}((0, +\infty))$ with the following properties:
  \begin{enumerate}[label=(P\arabic*)]
    \item $q(r) = \frac 12 r^2$ for $r \leq R$, \label{enum:approx-ym}
    \item there exists $\wt R > 0$ (depending on $c$ and $R$) such that $q(r) \equiv \tx{const}$ for $r \geq \wt R$, \label{enum:support-ym}
    \item $|q'(r)| \lesssim r$ and $|q''(r)| \lesssim 1$ for all $r > 0$, with constants independent of $c$ and $R$, \label{enum:gradlap-ym}
    \item $q''(r) \geq -c$ and $\frac 1r q'(r) \geq -c$, for all $r > 0$, \label{enum:convex-ym}
    \item $(\frac{\vd^2}{\vd r^2} + \frac 1r \dd r)^2 q(r) \leq c\cdot r^{-2}$, for all $r > 0$, \label{enum:bilapl-ym}
    \item $\big|r\big(\frac{q'(r)}{r}\big)'\big| \leq c$, for all $r > 0$. \label{enum:multip-petit-ym}
  \end{enumerate}
\end{lemma}
\begin{proof}
  It suffices to prove the result for $R = 1$ since the function $q_R(r) := R^2 q(\frac rR)$ satisfies
  the listed properties if and only if $q(r)$ does.
  
  First we define $q_0(r)$ by the formula
  \begin{equation}
    \label{eq:q0-ym}
    q_0(r) := \bigg\{
      \begin{aligned}
        &{\textstyle \frac 12 \cdot r^2}\qquad & r\leq 1 \\
        &{\textstyle \frac 12 r^2 + c_1\big(\frac 12 (r-1)^2 -\log(r)\frac{r^2-1}{4}\big)}\qquad & r \geq 1,
      \end{aligned}
    \end{equation}
    with $c_1$ small. A direct computation shows that for $r > 1$ we have
    \begin{equation}
      \label{eq:deriv-q0-ym}
      \begin{gathered}
        q_0'(r) = r\big(1-\frac{c_1\log r}{2}\big) + c_1\big(\frac 34 r - 1 +\frac{1}{4r}\big), \\
        q_0''(r) = \big(1-\frac{c_1 \log r}{2}\big) + c_1\big(\frac 14-\frac{1}{4r^2}\big), \\
        q_0'''(r) = -c_1\frac{r^2-1}{2r^3}, \\
        \Big(\frac{\vd^2}{\vd r^2} + \frac 1r \dd r\Big)^2 q_0(r) = -\frac{c_1}{r^3}.
      \end{gathered}
    \end{equation}
    In particular $q_0(1) = \frac 12$, $q_0'(1) = 1$, $q_0''(1) = 1$ and $q_0'''(1) = 0$, hence $q_0 \in C^{3, 1}$.
    
    Let $R_0 := \eee^{2/{c_1}}$. From \eqref{eq:deriv-q0-ym} it follows that $q_0(r)$ verifies all the listed properties except for \ref{enum:support-ym} for $r \leq R_0$.
    Let $e_j(r) := \frac{1}{j!}r^j\cdot \chi(r)$ for $j \in \{1, 2, 3\}$, where $\chi$ is a standard cut-off function. We define
  \begin{equation}
    \label{eq:q-ym}
    q(r) := \bigg\{
      \begin{aligned}
        &q_0(r)\qquad & r\leq R_0 \\
        &{\textstyle q_0(R_0) + \sum_{j=1}^{3}q_0^{(j)}(R_0)\cdot R_0^j\cdot e_j(-1 + R_0^{-1}r)}\qquad & r \geq R_0.
      \end{aligned}
    \end{equation}
    We will show that $q(r)$ has all the required properties if $c_1$ is small enough.
    Indeed, it is clear that $q(r) \in C^{3,1}((0, +\infty))$.
    It follows from \eqref{eq:deriv-q0-ym} that $|q_0^{(j)}(R_0)| \lesssim c_1 R_0^{2-j}$ for $j \in \{1, 2, 3\}$. For $r > R_0$ and $k \in \{1, 2, 3, 4\}$ we have
    \begin{equation}
      \label{eq:deriv-q-ym}
      q^{(k)}(r) = \sum_{j=1}^3 q_0^{(j)}(R_0)\cdot R_0^{j-k}e_j^{(k)}(-1 + R_0^{-1}r)\quad \Rightarrow\quad |q^{(k)}(r)| \lesssim c_1 R_0^{2-k}.
    \end{equation}
    Since $q(r) \equiv \tx{const}$ for $r \geq 3 R_0$, we obtain \ref{enum:approx-ym}--\ref{enum:multip-petit-ym}.
\end{proof}

We define the operators $A(\lambda)$ and $A_0(\lambda)$ as follows:
\begin{align}
  [A(\lambda)h](r) &:= q'\big(\frac{r}{\lambda}\big)\cdot h'(r), \label{eq:opA-ym} \\
  [A_0(\lambda)h](r) &:= \big(\frac{1}{2\lambda}q''\big(\frac{r}{\lambda}\big) + \frac{1}{2r}q'\big(\frac{r}{\lambda}\big)\big)h(r) + q'\big(\frac{r}{\lambda}\big)\cdot h'(r). \label{eq:opA0-ym}
\end{align}

\begin{lemma}
  \label{lem:op-A-ym}
  The operators $A(\lambda)$ and $A_0(\lambda)$ have the following properties:
  \begin{itemize}[leftmargin=0.5cm]
    \item the families $\{A(\lambda): \lambda > 0\}$, $\{A_0(\lambda): \lambda > 0\}$, $\{\lambda\partial_\lambda A(\lambda): \lambda > 0\}$
      and $\{\lambda\partial_\lambda A_0(\lambda): \lambda > 0\}$ are bounded in $\scrL(\cH; L^2)$, with the bound depending on the choice of the function $q(r)$,
    \item for all $\lambda > 0$ and $h_1, h_2 \in X^1$ there holds
      \begin{equation}
        \label{eq:A-by-parts-ym}
        \begin{gathered}
        \big|\big\la A(\lambda)h_1, \frac{1}{r^2}\big(f(h_1 + h_2) - f(h_1) - f'(h_1)h_2\big)\big\ra +\big\la A(\lambda)h_2, \frac{1}{r^2}\big(f(h_1+h_2) - f(h_1) +4 h_2\big)\big\ra\big| \\
        \leq \frac{c_0}{\lambda}\big((\|h_1\|_\cH^2+1)\|h_2\|_\cH^2 + \|h_2\|_\cH^4\big),
      \end{gathered}
      \end{equation}
      with a constant $c_0$ arbitrarily small,
    \item for all $h \in X^1$ there holds
      \begin{equation}
        \label{eq:A-pohozaev-ym}
        \big\la A_0(\lambda)h, \big(\partial_r^2 + \frac 1r\partial_r - \frac{4}{r^2}\big)h\big\ra \leq \frac{c_0}{\lambda}\|h\|_{\cH}^2 - \frac{2\pi}{\lambda}\int_0^{R\lambda}\big((\partial_r h)^2 + \frac{4}{r^2}h^2\big) \udr,
      \end{equation}
    \item assuming \eqref{eq:lambda-bound0}, for any $c_0 > 0$ there holds
    \begin{gather}
      \label{eq:L0-A0-ym}
      \|\Lambda_0 \Lambda W_\uln{\lambda(t)} - A_0(\lambda(t))\Lambda W_{\lambda(t)}\|_{L^2} \leq c_0, \\
      \label{eq:L-A-ym}
      \|\dot \varphi(t) + b(t)\cdot A(\lambda(t))\varphi(t)\|_{L^\infty} \leq c_0, \\
      \label{eq:approx-potential-ym}
      \begin{aligned}
        \Big|\int_0^{+\infty}\frac 12 \big(q''\big(\frac{r}{\lambda}\big) &+ \frac{\lambda}{r}q'\big(\frac{r}{\lambda}\big)\big)\frac{1}{r^2}\big(f(\varphi + g) - f(\varphi)+4g\big)g \udr \\
        &- \int_0^{+\infty} \frac{1}{r^2}\big(f'(W_\lambda)+4\big)g^2 \udr\Big| \leq c_0C_0^2 \eee^{-3\kappa|t|},
    \end{aligned}
    \end{gather}
    provided that the constant $R$ in the definition of $q(r)$ is chosen large enough.
  \end{itemize}
\end{lemma}
\begin{remark}
  The condition $\partial_r h_1, \frac 1r h_1, \partial_r h_2, \frac 1r h_2 \in \cH$ is required only to ensure that the left hand side of \eqref{eq:A-by-parts} is well defined,
  but it does not appear on the right hand side of the estimate. Note also that in \eqref{eq:A-by-parts-ym}, \eqref{eq:A-pohozaev-ym} and \eqref{eq:approx-potential-ym}
  we extract the linear part of $f$, see Remark~\ref{rem:hidden-lin}.
\end{remark}
\begin{proof}
  The proof of the first point is the same as in Lemma~\ref{lem:op-A}.

  In \eqref{eq:A-by-parts-ym}, without loss of generality we may assume that $h_1, h_2 \in C_0^\infty((0,+\infty))$ and that $\lambda = 1$.
  From the definition of $A(\lambda)$ we have
  \begin{equation}
    \begin{gathered}
    \big\la A(\lambda)h_1, \frac{1}{r^2}\big(f(h_1 + h_2) - f(h_1) - f'(h_1)h_2\big)\big\ra +\big\la A(\lambda)h_2, \frac{1}{r^2}\big(f(h_1+h_2) - f(h_1) + 4h_2\big)\big\ra \\
    = \int_0^{+\infty}q'h_1'\cdot \frac{1}{r^2}\big(f(h_1 + h_2) - f(h_1) - f'(h_1)\big) + q'h_2'\cdot \frac{1}{r^2}\big(f(h_1+h_2)-f(h_1)+4h_2\big)\udr \\
    = \int_0^{+\infty}\frac{q'}{r}\big(\dd r F(h_1 + h_2) - \dd r F(h_1) - h_2\cdot \dd r f(h_1) - f(h_1)\cdot \dd r h_2 + 2\dd r(h_2^2)\big)\ud r \\
    = \int_0^{+\infty} r\big(\frac{q'}{r}\big)'\cdot\frac{1}{r^2}\big(F(h_1 + h_2) - F(h_1) - f(h_1)\cdot h_2 + 2h_2^2\big)\udr.
  \end{gathered}
  \end{equation}
  Using \ref{enum:multip-petit-ym} in Lemma~\ref{lem:fun-a-ym} and the elementary inequality $|F(h_1 + h_2) - F(h_1) - f(h_1)h_2| \lesssim |h_1|^2|h_2|^2 + |h_1|^4$
  we get \eqref{eq:A-by-parts-ym}.
  
  Note that as a part of this computation, we obtain
  \begin{equation}
    \label{eq:A-lin-h-ym}
    \big|\big\la A(\lambda)h, \frac{1}{r^2}h\big\ra\big| \leq \frac{c_0}{\lambda}\int_0^{+\infty}\frac{1}{r^2}h^2\udr.
  \end{equation}
  For $r\leq R\lambda$ there holds $\frac{1}{2\lambda}q''\big(\frac{r}{\lambda}\big) + \frac{1}{2r}q'\big(\frac{r}{\lambda}\big) = \frac{1}{\lambda}$
  and for all $r$, thanks to \ref{enum:convex-ym}, there holds $\frac{1}{2\lambda}q''\big(\frac{r}{\lambda}\big) + \frac{1}{2r}q'\big(\frac{r}{\lambda}\big) \geq -\frac{c_0}{\lambda}$, hence
  \begin{equation}
    \label{eq:lapq-lin-h-ym}
    \big\la \big(\frac{1}{2\lambda}q''\big(\frac{r}{\lambda}\big) + \frac{1}{2r}q'\big(\frac{r}{\lambda}\big)\big)h, \frac{1}{r^2}h\big\ra \geq \frac{2\pi}{\lambda}\int_0^{R\lambda}\frac{1}{r^2}h^2\udr - \frac{c_0}{\lambda}\int_0^{+\infty}\frac{1}{r^2}h^2\udr.
  \end{equation}
  Taking the sum of \eqref{eq:A-lin-h-ym} and \eqref{eq:lapq-lin-h-ym} we obtain
  \begin{equation}
    \label{eq:A0-lin-h-ym}
    \big\la A_0(\lambda)h, \frac{1}{r^2}h\big\ra \geq -\frac{c_0}{\lambda}\int_0^{+\infty}\frac{1}{r^2}h^2\udr+ \frac{2\pi}{\lambda}\int_0^{R\lambda}\frac{1}{r^2}h^2\udr
  \end{equation}
  ($c_0$ has changed, but is still small).

  Using identity \eqref{eq:aux-pohozaev} with $N = 2$ we obtain, cf. the proof of \eqref{eq:A-pohozaev},
      \begin{equation}
        \label{eq:aux-pohozaev-ym}
        \big\la A_0(\lambda)h, \big(\partial_r^2 + \frac 1r\partial_r\big)h\big\ra \leq \frac{c_0}{\lambda}\int_0^{+\infty}(\partial_r h)^2 \udr - \frac{2\pi}{\lambda}\int_{r\leq R\lambda}(\partial_r h)^2 \udr.
      \end{equation}
      Taking the difference of \eqref{eq:aux-pohozaev-ym} and \eqref{eq:A0-lin-h-ym} we obtain \eqref{eq:A-pohozaev}.

  The proofs of \eqref{eq:L0-A0-ym} and \eqref{eq:L-A-ym} are similar to the proofs of \eqref{eq:L0-A0} and \eqref{eq:L-A} respectively.
  Instead of \eqref{eq:fidot-2} we prove that $\|bA(\lambda)W_\mu\|_{L^\infty} \leq \frac{c_0}{3}$, which follows from \ref{enum:support-ym} and \ref{enum:gradlap-ym}.
  We skip the details.

  The proof of \eqref{eq:approx-potential-ym} is close to the proof of \eqref{eq:approx-potential}. Note that it is crucial that $f'(W_\lambda) + 4$ vanishes at infinity.
\end{proof}
For $t \in [T, T_0]$ we define:
\begin{itemize}
  \item the \emph{nonlinear energy functional}
    $$
    \begin{aligned}
      I(t) &:= \int \frac 12 |\dot g(t)|^2 + \frac 12 |\grad g(t)|^2 - \frac{1}{r^2}\big(F(\varphi(t) + g(t)) - F(\varphi(t)) - f(\varphi(t))g(t)\big)\ud x \\
      &= E(\bs \varphi(t) + \bs g(t)) - E(\bs \varphi(t)) - \la \vD E(\bs \varphi(t), \bs g(t))\ra,
    \end{aligned}
    $$
  \item the \emph{localized virial functional} $$J(t) := \int \dot g(t)\cdot A_0(\lambda(t))g(t)\ud x,$$
  \item the \emph{mixed energy-virial functional} $$H(t) := I(t) + b(t)J(t).$$
\end{itemize}
\subsection{Energy estimates via the mixed energy-virial functional}
  \label{ssec:energy-ym}
  The remaining part of the proof is almost identical to Subsection~\ref{ssec:energy}. We will indicate the few differences.

  Instead of \eqref{eq:dtI}, we obtain now
  \begin{equation}
    \label{eq:dtI-ym}
    I'(t) \simeq -b(\lambda'-b)\cdot\la A_0(\lambda) \Lambda W_{\uln\lambda}, \dot g\ra - b\cdot\big\la A(\lambda)g, \frac{1}{r^2}\big(f(\varphi + g) - f(\varphi) + 4g\big)\big\ra.
  \end{equation}
  As in the proof of Lemma~\ref{lem:bootstrap}, we have
  \begin{equation}
    \label{eq:dtJ-ym}
    \begin{aligned}
    &(bJ)'(t) \simeq b\int\dot g\cdot A_0(\lambda)(\dot g - \psi)\udr + b\int\big(\big(\partial_r^2 + \frac 1r\partial_r\big)g \\
&+ \frac{1}{r^2}\big(f(\varphi+g) - f(\varphi)\big) - \dot \psi\big)\cdot A_0(\lambda)g \udr \\
    &= b\int\dot g\cdot A_0(\lambda)(\dot g - \psi)\udr + b\int\big(\big(\partial_r^2 +\frac 1r\partial_r - \frac{4}{r^2}\big)g \\
&+ \frac{1}{r^2}\big(f(\varphi + g) - f(\varphi) + 4g\big) - \dot \psi\big)\cdot A_0(\lambda)g \udr,
  \end{aligned}
  \end{equation}
  where we recognize the terms appearing in \eqref{eq:A-pohozaev-ym} and \eqref{eq:approx-potential-ym}.
  The rest of the proof applies without change. Theorem~\ref{thm:deux-bulles-ym} follows from the argument given in Subsection~\ref{ssec:limit}.

\section{Bubble-antibubble for the equivariant critical wave map equation}
\label{sec:wmap}

\subsection{Notation}
\label{ssec:setting-wmap}

We use similar notation as in Section~\ref{sec:ym}, with a slight modification in the definition of the norm $\cH$:
\begin{equation}
  \label{eq:H-norm-wmap}
  \|v\|_{\cH}^2 := 2\pi\int_0^{+\infty}\big(|\partial_r v(r)|^2 + |\frac kr v(r)|^2 \big)\udr.
\end{equation}
The transformation $\wt v(\eee^{i\theta}r) := \eee^{ki\theta}v(r)$ is an isometric embedding of $\cH$ in $\dot H^1(\bR^2; \bR^2)$,
whose image is given by $k$-equivariant functions in $\dot H^1(\bR^2; \bR^2)$.

\subsection{Linearized equation and formal computation}
\label{ssec:linearise-wmap}
Linearizing $-\partial_r^2 u - \frac 1r\partial_r u + \frac{k^2}{2r^2}\sin(2u)$ around $u = W$ we obtain the operator
\begin{equation}
  L := -\partial_r^2 - \frac 1r \partial_r + \frac{k^2}{r^2}\cos(4\arctan(r^k)) = -\partial_r^2 - \frac 1r \partial_r + \frac{k^2}{r^2}\Big(1-\frac{8}{(r^k+r^{-k})^2}\Big).
\end{equation}
It has a one-dimensional kernel spanned by $\Lambda W$. We fix $\cZ \in C_0^\infty((0, +\infty))$ such that
\begin{equation}
  \label{eq:Z-pos-wmap}
\int_0^{+\infty} \cZ(r)\cdot \Lambda W(r)\frac{\vd r}{r} > 0.
\end{equation}

\begin{lemma}
  \label{lem:fredholm-wmap}
  For all $V(r) \in C^\infty((0, +\infty))$ such that $\int_0^{+\infty} \Lambda W(r)\cdot V(r) \udr = 0$, $|V(r)| \lesssim r^k$ for small $r$ and $|V(r)| \lesssim r^{-k}$ for large $r$,
  then there exists a function $U(r) \in C^\infty((0, +\infty))$ such that
  \begin{gather}
    LU = V, \label{eq:fredholm-eq-wmap} \\ 
    |U(r)| \lesssim r^k,\quad |\partial_r U(r)| \lesssim r^{k-1},\quad |\partial_r^2 U(r)| \lesssim r^{k-2}\qquad\text{for $r$ small}, \label{eq:fredholm-rpetit-wmap} \\
    |U(r)| \lesssim r^{-k},\quad |\partial_r U(r)| \lesssim r^{-k-1},\quad |\partial_r^2 U(r)| \lesssim r^{-k-2}\qquad\text{for $r$ large}, \label{eq:fredholm-rgrand-wmap} \\
    \int \cZ(r)\cdot U(r) \frac{\vd r}{r} = 0. \label{eq:fredholm-orth-wmap}
  \end{gather}
\end{lemma}
\begin{proof}
  It is easy to check that the operator $L$ factorizes as follows:
  \begin{equation}
    \label{eq:factor-L-wmap}
    L = -\partial_r^2 - \frac 1r \partial_r + \frac{k^2}{r^2}\Big(1-\frac{8}{(r^k+r^{-k})^2}\Big) = \Big({-}\partial_r - \frac 1r -\frac{\Lambda W'(r)}{\Lambda W(r)}\Big)
    \Big(\partial_r -\frac{\Lambda W'(r)}{\Lambda W(r)}\Big),
  \end{equation}
  hence we can invert it explicitely using twice the variation of constants formula.
  Define $U_1 \in C^\infty((0, +\infty))$ by
  \begin{equation}
    U_1(r) := \frac{1}{r\Lambda W(r)}\int_0^r V(\rho)\Lambda W(\rho)\rho\ud \rho.
  \end{equation}
  It solves the equation $\Big(-\partial_r - \frac 1r -\frac{\Lambda W'(r)}{\Lambda W(r)}\Big)U_1(r) = V(r)$.
  Since $|V(r)| \lesssim r^k$ and $\Lambda W(r) \sim r^k$ for small $r$, we have
  \begin{equation}
    \label{eq:U1-rpetit}
    |U_1(r)| \lesssim r^{-1-k+k+k+1+1} = r^{k+1},\qquad \text{small }r.
  \end{equation}
  From the crucial assumption $\int_0^{+\infty}V(\rho)\Lambda W(\rho)\rho\ud \rho = 0$ we get
  \begin{equation}
    \label{eq:U1-rgrand}
    |U_1(r)| = \Big|\frac{1}{r\Lambda W(r)}\int_r^{+\infty}V(\rho)\Lambda W(\rho)\rho\ud \rho\Big| \lesssim r^{-k+1},\qquad \text{large }r.
  \end{equation}
  From the differential equation we get also $|\partial_r U_1(r)| \lesssim r^k$ for small $r$ and $|\partial_r U_1(r)| \lesssim r^{-k}$ for large $r$.
  Now we define $U \in C^\infty((0, +\infty))$ by the formula
  \begin{equation}
    U(r) := \Lambda W(r)\int_0^r \frac{U_1(\rho)}{\Lambda W(\rho)}\ud \rho.
  \end{equation}
  It solves $\Big(\partial_r -\frac{\Lambda W'(r)}{\Lambda W(r)}\Big)U(r) = U_1(r)$, hence \eqref{eq:factor-L-wmap} yields \eqref{eq:fredholm-eq-wmap}.
  Using \eqref{eq:U1-rpetit} and \eqref{eq:U1-rgrand}, one can check that $|U(r)| \lesssim r^{k+2}$ for small $r$ and $|U(r)| \lesssim r^{-k+2}$ for large $r$.
  The differential equations yield $|\partial_r U(r)| \lesssim r^{k+1}$ and $\partial_r^2 U(r)| \lesssim r^k$ for small $r$, as well as
  $|\partial_r U(r)| \lesssim r^{-k+1}$ and $|\partial_r^2 U(r)| \lesssim r^{-k}$ for large $r$.
  Adding to $U$ a suitable multiple of $\Lambda W$ we obtain \eqref{eq:fredholm-orth-wmap}. Since $|\Lambda W(r)| \lesssim r^k$,
  $|\partial_r \Lambda W(r)| \lesssim r^{k-1}$, $|\partial_r^2 \Lambda W(r)| \lesssim r^{k-2}$ for small $r$ and
  $|\Lambda W(r)| \lesssim r^{-k}$, $|\partial_r \Lambda W(r)| \lesssim r^{-k-1}$, $|\partial_r^2 \Lambda W(r)| \lesssim r^{-k-2}$ for large $r$,
  \eqref{eq:fredholm-rpetit-wmap} and \eqref{eq:fredholm-rgrand-wmap} still hold.
\end{proof}

We study solutions behaving like $\bs u(t) \simeq -\bs W + \bs W_{\lambda(t)}$ with $\lambda(t) \to 0$ as $t \to -\infty$. We expand
\begin{equation}
  \bs u(t) = -\bs W + \bs U_{\lambda(t)}^{(0)} + b(t)\cdot \bs U^{(1)}_{\lambda(t)} + b(t)^2 \cdot\bs U^{(2)}_{\lambda(t)},
\end{equation}
with $b(t)  = \lambda'(t)$, $\bs U^{(0)} := (W, 0)$ and $\bs U^{(1)} := (0, -\Lambda W)$.
As in Subsection~\ref{ssec:linearise-ym}, in the region $r \leq \sqrt\lambda$ we arrive at
\begin{equation}
  \partial_r^2 u + \frac 1r \partial_r u - \frac{k^2}{2r^2}\sin(2u) = -\frac{b^2}{\lambda}(LU^{(2)})_\uln\lambda - \frac{1}{r^2}\cdot\frac{8k^2}{((r/\lambda)^k + (r/\lambda)^{-k})^2}W + \text{lot}.
\end{equation}
Using the fact that $W(r) \sim 2r^k$ for small $r$ we obtain
\begin{equation}
  \partial_r^2 u + \frac 1r \partial_r u - \frac{k^2}{2r^2}\sin(2u) = -\frac{b^2}{\lambda}(LU^{(2)})_\uln\lambda -\frac{16k^2r^{k-2}}{((r/\lambda)^k + (r/\lambda)^{-k})^2}+ \text{lot},
\end{equation}
thus, after rescaling,
\begin{equation}
  LU^{(2)} = -\Lambda_0 \Lambda W + \frac{\lambda}{b^2}\big(b'\Lambda W - \lambda^{k-1}\cdot\frac{16k^2r^{k-2}}{(r^k+r^{-k})^2}\big).
\end{equation}
It is not difficult to check (using for example the residue theorem) that
\begin{equation}
  \int_0^{+\infty} \Lambda W(r)\cdot\frac{16k^2r^{k-2}}{(r^k+r^{-k})^2}\udr =\frac{4k^2}{\pi}\sin\big(\frac{\pi}{k}\big)\cdot\int \Lambda W(r)^2 \udr,
\end{equation}
hence the correct choice (that is, such that Lemma~\ref{lem:fredholm-wmap} allows to invert $L$) of the formal parameter equations is
\begin{equation}
  \label{eq:formal-param-wmap}
  \lambda'(t) = b(t), \qquad b'(t) =  \frac{4k^2}{\pi}\sin\big(\frac{\pi}{k}\big)\lambda(t)^{k-1}= \frac k2\Big(\frac{2\kappa}{k-2}\Big)^k \lambda(t)^{k-1},
\end{equation}
where $\kappa := \frac{k-2}{2}\cdot\big(\frac{8k}{\pi}\sin\big(\frac{\pi}{k}\big)\big)^\frac 1k$.
This system has a solution
\begin{equation}
  \label{eq:param-sol-wmap}
  (\lambda_\tx{app}(t), b_\tx{app}(t)) = \big(\frac{k-2}{2\kappa}(\kappa|t|)^{-\frac{2}{k-2}}, (\kappa|t|)^{-\frac{k}{k-2}}\big),\qquad t \leq T_0 < 0.
\end{equation}

\subsection{Bounds on the error of the ansatz}
\label{ssec:error-wmap}
Let $I = [T, T_0]$ be the time interval, $T \leq T_0 < 0$, $|T_0|$ large. Let $\lambda(t)$ and $\mu(t)$ be $C^1$ functions on $[T, T_0]$ such that
\begin{gather}
  \lambda(T) = \lambda_0,\quad \mu(T) = 1, \label{eq:lambda-init0-wmap} \\
  \frac{k-2}{4\kappa}(\kappa|t|)^{-\frac{2}{k-2}}\leq \lambda(t) \leq \frac{k-2}{\kappa}(\kappa|t|)^{-\frac{2}{k-2}},\qquad \frac 12 \leq \mu(t)\leq 2, \label{eq:lambda-bound0-wmap}
\end{gather}
with $\lambda_0$ satisfying $\big|\lambda_0 - \frac{k-2}{\kappa} (\kappa|T|)^{-\frac{2}{k-2}}\big| \lesssim |T|^{-\frac{5}{2(k-2)}}$.
For technical reasons we cannot fix $\lambda_0 = \frac{k-2}{\kappa} (\kappa|T|)^{-\frac{2}{k-2}}$ as we did in the preceding sections. The value of $\lambda_0$ will be chosen later.

Let $P(r)$ and $Q(r)$ by the functions obtained in Lemma~\ref{lem:fredholm-wmap} for $V(r) = \frac k2\big(\frac{2\kappa}{k-2}\big)^k \Lambda W(r) - \frac{16k^2r^{k-2}}{(r^k+r^{-k})^2}$
and $V(r) = -\Lambda_0 \Lambda W(r)$ respectively.
We define the approximate solution by the formula
\begin{equation}
  \label{eq:phi-def-wmap}
  \begin{aligned}
    \varphi(t) &:= -W_{\mu(t)} + W_{\lambda(t)} + S(t), \\
    \dot \varphi(t) &:= -b(t)\Lambda W_\uln{\lambda(t)}.
\end{aligned}
\end{equation}
where
\begin{align}
  \label{eq:b-def-wmap}
  b(t) &:= \Big(\frac{2\kappa}{k-2}\Big)^\frac k2\lambda_0^\frac k2 + \frac k2\Big(\frac{2\kappa}{k-2}\Big)^k \int_T^t \frac{\lambda(\tau)^{k-1}}{\mu(\tau)^k}\ud \tau,\qquad&\text{for }t \in [T, T_0], \\
  S(t) &:= \frac{\lambda(t)^k}{\mu(t)^k} P_{\lambda(t)} + b(t)^2 Q_{\lambda(t)},\qquad&\text{for }t\in [T, T_0]. \label{eq:S-def-wmap}
\end{align}
The definition of $b(t)$ and \eqref{eq:lambda-bound0-wmap} yield
\begin{equation}
  b(t) \sim |t|^{-\frac{k}{k-2}},
\end{equation}
with a constant depending only on $k$.

Since $P, Q \in \cH$, there holds
\begin{equation}
  \label{eq:taille-correct-wmap}
  \|S(t)\|_\cH \lesssim |t|^{-\frac{2k}{k-2}}.
\end{equation}
Note also that
\begin{equation}
  \label{eq:S-orth-wmap}
  \int \cZ_\uln{\lambda}\cdot S(t)\frac{\vd r}{r} = 0.
\end{equation}

We denote $f(u) := -\frac{k^2}{2} \sin(2u)$ (hence $f'(u) = -k^2 \cos(2u)$) and
\begin{equation}
  \label{eq:psi-wmap}
  \begin{aligned}
  \bs \psi(t) &= (\psi(t), \dot \psi(t)) := \partial_t  \bs \varphi(t) - \vD E(\bs \varphi(t))  \\
  &= \big(\partial_t \varphi(t) - \dot \varphi(t), \partial_t \dot \varphi(t) - (\partial_r^2 \varphi(t) +\frac 1r\partial_r \varphi(t)+ \frac{1}{r^2}f(\varphi(t)))\big).
\end{aligned}
\end{equation}
We point out that usually, in the context of equivariant wave maps, $f(u)$ has the opposite sign. We~chose the sign which is more coherent with the traditional notation for \eqref{eq:nlw0}.
\begin{lemma}
  \label{lem:psi-wmap}
  Suppose that for $t \in [T, T_0]$ there holds $|\lambda'(t)| \lesssim |t|^{-\frac{k}{k-2}}$ and $|\mu'(t)| \lesssim |t|^{-\frac{k}{k-2}}$. Then
  \begin{align}
    \|\psi(t) - \mu'(t)\frac{1}{\mu(t)}\Lambda W_{\mu(t)}+ (\lambda'(t) - b(t))\frac{1}{\lambda(t)}\Lambda W_{\lambda(t)} \|_{\cH} &\lesssim |t|^{-\frac{2k-1}{k-2}}, \label{eq:psi0-wmap} \\
    \|\dot \psi(t) - \frac{b(t)}{\lambda(t)}(\lambda'(t) - b(t))\Lambda_0 \Lambda W_\uln{\lambda(t)}\|_{L^2} &\lesssim |t|^{-\frac{2k-1}{k-2}}, \label{eq:psi1-wmap} \\
    \|(-\partial_r^2 - \frac 1r \partial_r - \frac{1}{r^2}f'(\varphi(t)))\psi(t)\|_{\cH^*} &\lesssim |t|^{-\frac{2k-1}{k-2}}. \label{eq:psi-lin-wmap}
  \end{align}
\end{lemma}
\begin{proof}
  From the definition of $\bs \psi$ we get
  \begin{equation}
    \psi + \mu'\Lambda W_\uln\mu + (\lambda' - b)\Lambda W_\uln\lambda = -k\frac{\lambda^k}{\mu^{k+1}}\mu'P_\lambda + k\frac{\lambda^k}{\mu^k}\lambda'P_\uln\lambda
      - \frac{\lambda^k}{\mu^k} \lambda'\Lambda P_\uln\lambda + 2\lambda b b'Q_\uln\lambda - b^2 \lambda'\Lambda Q_\uln\lambda.
  \end{equation}
  The first term has size $\lesssim |t|^{-\frac{3k}{k-2}} \ll |t|^{-\frac{2k-1}{k-2}}$ in $\cH$. The other terms have size $\lesssim |t|^{-\frac{3k-2}{k-2}} \ll |t|^{-\frac{2k-1}{k-2}}$ in $\cH$.

In order to prove \eqref{eq:psi1-wmap}, we treat separately the regions $r \leq \sqrt\lambda$ and $r \geq \sqrt\lambda$.
First we will show that
\begin{equation}
  \label{eq:psi1-nonlin-wmap}
  \big\|\frac{1}{r^2}\big(f(\varphi) - f(W_\lambda) + 2\frac{r^k}{\mu^k}f'(W_\lambda) - \frac{\lambda^k}{\mu^k} f'(W_\lambda)P_\lambda - b^2 f'(W_\lambda)Q_\lambda\big)\big\|_{L^2(r \leq \sqrt\lambda)} \lesssim |t|^{-\frac{2k-1}{k-2}}.
\end{equation}
We have an elementary pointwise inequality
\begin{equation}
  \label{eq:pointwise-1-wmap}
  |f(\varphi) - f(W_\lambda) - f'(W_\lambda)(-W_\mu + S)| \lesssim |{-}W_\mu + S|^2,
\end{equation}
with a constant depending only on $k$. We have $|W_\mu| \lesssim r^k$ and $|S| \lesssim (b^2 + \lambda^k)\cdot\big(\frac{r}{\lambda}\big)^k \lesssim r^k$,
hence
\begin{equation}
  \big\|\frac{1}{r^2}|{-}W_\mu + S|^2\big\|_{L^2(r\leq \sqrt\lambda)} \lesssim \Big(\int_0^{\sqrt\lambda}\big(\frac{1}{r^2}\cdot r^{2k}\big)^2\udr\Big)^\frac 12 \sim \lambda^\frac{2k-1}{2} \lesssim |t|^{-\frac{2k-1}{k-2}},
\end{equation}
which means that the right hand side in \eqref{eq:pointwise-1-wmap} is negligible. From the well-known fact that $|\arctan(z) - z| \lesssim |z|^3$ for small $z$ we get
$|W_\mu(r) - 2(r/\mu)^k|\lesssim r^{3k}$. This implies
\begin{equation}
  \big\|\frac{1}{r^2}f'(W_\lambda)\big(W_\mu - 2\frac{r^k}{\mu^k}\big)\big\|_{L^2}\lesssim \Big(\int_0^{\sqrt\lambda} r^{6k-4} \udr\Big)^{\frac 12}\lesssim \lambda^\frac{3k-1}{2} \lesssim |t|^{-\frac{3k-1}{k-2}} \ll |t|^{-\frac{2k-1}{k-2}}.
\end{equation}
This proves \eqref{eq:psi1-nonlin-wmap} (see the definition of $S$).

We have $\big|W_\mu'(r) - \frac{2kr^{k-1}}{\mu^k}\big| \lesssim r^{3k-1}$ and $\big|W_\mu''(r) - \frac{2(k^2 - k)r^{k-2}}{\mu^k}\big| \lesssim r^{3k-2}$ for small $r$, which implies
$\big|\big(\partial_r^2+\frac 1r \partial_r\big)W_\mu(r) - \frac{2k^2r^{k-2}}{\mu^k}\big| \lesssim r^{3k-2}$, hence
\begin{equation}
  \label{eq:psi1-lin-wmap}
  \big\|\big(\partial_r^2 + \frac 1r \partial_r\big)W_\mu - \frac{2k^2r^{k-2}}{\mu^k}\big\|_{L^2(r\leq \sqrt\lambda)} \lesssim \Big(\int_0^{\sqrt\lambda}r^{6k-4}\udr\Big)^\frac 12 \ll |t|^{-\frac{2k-1}{k-2}}.
\end{equation}
Since $W_\lambda''(r) + \frac 1r W_\lambda'(r) + \frac{1}{r^2}f(W_\lambda(r)) = 0$, from \eqref{eq:psi1-nonlin-wmap}, \eqref{eq:psi1-lin-wmap} and the definition of $\varphi(t)$ we have
\begin{equation}
  \label{eq:psi1-all-wmap}
  \begin{aligned}
  \big\|\big((\partial_r^2 + \frac 1r \partial_r)\varphi + \frac{1}{r^2}f(\varphi)\big)-\big(&{-}\frac{2k^2 r^{k-2}}{\mu^k} - \frac{2r^{k-2}}{\mu^k}f'(W_\lambda) - \frac{\lambda^{k-1}}{\mu^k}(LP)_\uln\lambda\\ &- \frac{b^2}{\lambda}(LQ)_\uln\lambda \big)\big\|_{L^2(r\leq \sqrt\lambda)} \lesssim \eee^{-\frac 32 \kappa|t|}.
\end{aligned}
\end{equation}
By the definition of $P$ there holds $LP = -\frac{16k^2r^{k-2}}{(r^k + r^{-k})^2} + \frac k2\big(\frac{2\kappa}{k-2}\big)^k\Lambda W = -2k^2 r^{k-2} - 2r^{k-2}f'(W)+ \frac k2\big(\frac{2\kappa}{k-2}\big)^k\Lambda W$ and by the definition of $Q$ there holds $LQ = -\Lambda_0 \Lambda W$, hence we can rewrite \eqref{eq:psi1-all-wmap} as
\begin{equation}
  \big\|\big((\partial_r^2 + \frac 1r \partial_r)\varphi + \frac{1}{r^2}f(\varphi)\big)-\big({-}\frac k2\big(\frac{2\kappa}{k-2}\big)^k\frac{\lambda^{k-1}}{\mu^k}\Lambda W_\uln\lambda + \frac{b^2}{\lambda}\Lambda_0 \Lambda W_\uln\lambda \big)\big\|_{L^2(r\leq \sqrt\lambda)} \lesssim \eee^{-\frac 32 \kappa|t|},
\end{equation}
which is precisely \eqref{eq:psi1-wmap} restricted to $r \leq \sqrt\lambda$, cf. \eqref{eq:dot-psi}.

Consider the region $r \geq \sqrt\lambda$. We have $\varphi = (\pi-W_\mu) + (W_\lambda - \pi) + S$, hence elementary pointwise inequalities yield
\begin{equation}
  |f(\varphi) - f(\pi-W_\mu) - f'(\pi-W_\mu)(W_\lambda - \pi + S)| \lesssim |W_\lambda -\pi + S|^2.
\end{equation}
From this and the relations $f(\pi-W_\mu) = -f(W_\mu)$, $f'(\pi-W_\mu) = f'(W_\mu)$, we obtain a pointwise bound
\begin{equation}
  \label{eq:psi1-2-nonlin-wmap}
  |f(\varphi) + f(W_\mu) + f'(W_\mu)(\pi-W_\lambda)| \lesssim |S| + |\pi-W_\lambda|^2.
\end{equation}
Since $|S(r)| \lesssim b^2 + \frac{\lambda^k}{\mu^k} \sim |t|^{-\frac{2k}{k-2}}$, we have
\begin{equation}
  \label{eq:psi1-2-nonlin2-wmap}
  \|\frac{1}{r^2}S\|_{L^2(r\geq \sqrt\lambda)} \lesssim |t|^{-\frac{2k}{k-2}} \Big(\int_{\sqrt\lambda}^{+\infty}r^{-4}\udr\Big)^\frac 12 \lesssim \frac{|t|^{-\frac{2k}{k-2}}}{\sqrt\lambda} \lesssim |t|^{-\frac{2k-1}{k-2}}.
\end{equation}
There holds $|\pi-W_\lambda(r)|^2 =|\pi - 2\arctan((r/\lambda)^k)|^2 = |2\arctan((\lambda/r)^k)|^2 \lesssim \frac{\lambda^{2k}}{r^{2k}}$, hence
\begin{equation}
  \label{eq:psi1-2-nonlin3-wmap}
  \|\frac{1}{r^2}|\pi-W_\lambda(r)|^2\|_{L^2(r\geq \sqrt\lambda)} \lesssim \lambda^{2k} \Big(\int_{\sqrt\lambda}^{+\infty}r^{-4k-4}\udr\Big)^{\frac 12} \sim \lambda^{2k}\lambda^{-\frac{2k+1}{2}} \sim |t|^{-\frac{2k-1}{k-2}}.
\end{equation}
Recall that $f'(W_\mu) = -k^2(1-8((r/\mu)^k + (r/\mu)^{-k})^{-2})$, hence $|f'(W_\mu) + k^2| \lesssim r^k$. We also have (by a standard asymptotic expansion of $\arctan$)
$|\pi-W_\lambda| \lesssim \frac{\lambda^k}{r^k}$ and $|\pi-W_\lambda - 2\frac{\lambda^k}{r^k}| \lesssim \frac{\lambda^{3k}}{r^{3k}}$, hence
\begin{equation}
\label{eq:psi1-2-nonlin4-wmap}
\begin{gathered}
  \big\|\frac{1}{r^2}f'(W_\mu)(\pi-W_\lambda) + \frac{2k^2\lambda^k}{r^{k+2}}\big\|_{L^2(r\geq \sqrt\lambda)} \lesssim \|\frac{\lambda^k}{r^2} + \frac{\lambda^{3k}}{r^{3k+2}}\|_{L^2(r\geq \sqrt\lambda)} \lesssim |t|^{-\frac{2k-1}{k-2}}.
\end{gathered}
\end{equation}
Inserting \eqref{eq:psi1-2-nonlin2-wmap}, \eqref{eq:psi1-2-nonlin3-wmap} and \eqref{eq:psi1-2-nonlin4-wmap} into \eqref{eq:psi1-2-nonlin-wmap} we obtain
\begin{equation}
  \label{eq:psi1-2-nonlin5-wmap}
  \|\frac{1}{r^2}f(\varphi) + \frac{1}{r^2}f(W_\mu) - \frac{2k^2 \lambda^k}{r^{k+2}}\|_{L^2(r\geq \sqrt\lambda)} \lesssim |t|^{-\frac{2k-1}{k-2}}.
\end{equation}
A direct computation shows that for $r \geq \sqrt\lambda$ there holds
$(\partial_r^2 + \frac 1r \partial_r)W_\lambda(r) = -\frac{2k^2\lambda^k}{r^{k+2}} + O\big(\frac{\lambda^{3k}}{r^{3k+2}}\big)$, hence
\begin{equation}
  \label{eq:psi1-2-nonlin6-wmap}
  \|(\partial_r^2 + \frac 1r \partial_r)W_\lambda + \frac{2k^2\lambda^k}{r^{k+2}}\|_{L^2(r\geq \sqrt\lambda)} \lesssim \lambda^{\frac{3k-1}{2}} \ll |t|^{-\frac{2k-1}{k-2}}.
\end{equation}
The same computation as in \eqref{eq:Delta-S-ym} yields $\|(\partial_r^2 + \frac 1r \partial_r)S\|_{L^2(r\geq \sqrt\lambda)} \lesssim \frac{b^2}{\sqrt\lambda} + \frac{\lambda^k}{\sqrt\lambda} \lesssim |t|^{-\frac{2k-1}{k-2}}$.
Together with \eqref{eq:psi1-2-nonlin5-wmap} and \eqref{eq:psi1-2-nonlin6-wmap} this proves that
\begin{equation}
  \|(\partial_r^2 + \frac 1r \partial_r)\varphi + \frac{1}{r^2}f(\varphi)\|_{L^2(r\geq \sqrt\lambda)} \lesssim |t|^{-\frac{2k-1}{k-2}}.
\end{equation}
Since $\|\Lambda W_\uln\lambda\|_{L^2(r\geq \sqrt\lambda)} + \|\Lambda_0 \Lambda W_\uln\lambda\|_{L^2(r\geq \sqrt\lambda)} \lesssim (\int_{1/\sqrt\lambda}^{+\infty}\frac{1}{r^{2k}}\udr)^{\frac 12}
\lesssim \lambda^\frac{k-1}{2}$, the other terms appearing in \eqref{eq:psi1-wmap} are $\lesssim |t|^{-\frac{3k-3}{k-2}} \ll |t|^{-\frac{2k-1}{k-2}}$. This finishes the proof of \eqref{eq:psi1-wmap}.

The proof of \eqref{eq:psi-lin-wmap} is very similar to the proof of \eqref{eq:psi-lin}, so we will not give all the details.
We have $\big|\frac{\lambda'-b}{\lambda}\big| \lesssim |t|^{-1}$ and $\big|\frac{\mu'}{\mu}\big| \lesssim |t|^{-1}$, hence it suffices to check that
$\big\|\frac{1}{r^2}(f'(\varphi) - f'(W_\lambda))\Lambda W_\lambda\big\|_{\cH^*} \lesssim |t|^{-\frac{k+1}{k-2}}$
and $\big\|\frac{1}{r^2}(f'(\varphi) - f'(W_\mu))\Lambda W_\mu\big\|_{\cH^*} \lesssim |t|^{-\frac{k+1}{k-2}}$,
which boils down to $\big\|\frac{1}{r^2}W_\mu\cdot \Lambda W_\lambda\big\|_{L^1} \lesssim |t|^{-\frac{k+1}{k-2}}$
and $\big\|\frac{1}{r^2}(\pi-W_\lambda)\cdot \Lambda W_\mu\big\|_{L^1} \lesssim |t|^{-\frac{k+1}{k-2}}$, see the proof of \eqref{eq:psi-lin-ym}.
In both cases we treat separately $r \leq 1$ and $r \geq 1$:
\begin{gather}
  \big\|\frac{1}{r^2}W_\mu\cdot \Lambda W_\lambda\big\|_{L^1(r \leq 1)} \|r^{k-2}\Lambda W_\lambda\|_{L^1(r\leq 1)} = \lambda^k \|\Lambda W\|_{L^1(r \leq 1/\lambda)} \lesssim \lambda^k|\log \lambda| \ll |t|^{-\frac{k+1}{k-2}}, \\
  \big\|\frac{1}{r^2}W_\mu\cdot \Lambda W_\lambda\big\|_{L^1(r \geq 1)} \|\Lambda W\|_{L^\infty(r \geq 1)} \lesssim \lambda^k \ll |t|^{-\frac{k+1}{k-2}}, \\
  \begin{aligned}
  \big\|\frac{1}{r^2}(\pi-W_\lambda)\cdot \Lambda W_\mu\big\|_{L^1(r \leq 1)} &\lesssim \big\|r^{k-2}\arctan\big((\lambda/r)^k\big)\big\|_{L^1(r\leq 1)} \\ &= \lambda^k\|r^{k-2}\arctan(r^{-k})\|_{L^1(r\leq 1/\lambda)} \lesssim \lambda^k|\log\lambda| \ll |t|^{-\frac{k+1}{k-2}},
  \end{aligned} \\
  \big\|\frac{1}{r^2}(\pi-W_\lambda)\cdot \Lambda W_\mu\big\|_{L^1(r \geq 1)} \lesssim \|\pi - W_\lambda\|_{L^\infty(r \geq 1)} \lesssim \lambda^k \ll |t|^{-\frac{k+1}{k-2}}.
\end{gather}
\end{proof}

\subsection{Modulation}
\label{ssec:mod-wmap}
As in Subsection~\ref{ssec:mod}, we define $\bs g(t) := \bs u(t) - \bs \varphi(t)$ with modulation parameters $\lambda(t)$ and $\mu(t)$ which satisfy
    \begin{gather}
    \label{eq:g-orth-wmap}
    \big\la \frac{1}{\lambda(t)}\cZ_\uln{\lambda(t)}, g(t)\big\ra = 0, \quad \big|\big\la \frac{1}{\mu(t)}\cZ_\uln{\mu(t)}, g(t)\big\ra\big| \lesssim c|t|^{-\frac{k+1}{k-2}}, \\
    \label{eq:mod-wmap}
    |\lambda'(t)-b(t)| + |\mu'(t)| \lesssim \|\bs g(t)\|_\cE + c|t|^{-\frac{k+1}{k-2}},
  \end{gather}
  with a constant $c$ arbitrarily small.
The initial data are
\begin{equation}
  \label{eq:data-at-T-wmap}
  \bs u(T) = \big({-}W + W_{\lambda_0}, -\big(\frac{2\kappa}{k-2}\big)^\frac k2\lambda_0^\frac k2 \Lambda W_{\uln{\lambda_0}}\big),
\end{equation}
where the value of $\lambda_0$ is close to $\frac{k-2}{2\kappa}(\kappa|T|)^{-\frac{2}{k-2}}$.
The initial value of $\lambda(t)$ in the modulation equations is thus $\lambda(T) = \lambda_0$, in accordance with \eqref{eq:lambda-init0-wmap},
and $b(t)$ is defined by \eqref{eq:b-def-wmap}. Note that $b(T) = \big(\frac{2\kappa}{k-2}\big)^\frac k2\lambda_0^\frac k2$ is close to $(\kappa|T|)^{-\frac{k}{k-2}}$.

The equivalent of Proposition~\ref{prop:bootstrap} can by formulated as follows.
\begin{proposition}
  \label{prop:bootstrap-wmap}
There exist constants $C_0 > 0$ and $T_0 < 0$ with the following property
  Let $T < T_1 < T_0$ and suppose that $\bs u(t) = \bs \varphi(t) + \bs g(t) \in C([T, T_1]; X^1 \times X^0)$
  is a solution of \eqref{eq:wmap0} with initial data \eqref{eq:data-at-T-wmap}
  such that for $t \in [T, T_1]$ condition \eqref{eq:lambda-bound0-wmap} is satisfied and
      \begin{align}
        \|\bs g(t)\|_\cE &\leq C_0 |t|^{-\frac{k+1}{k-2}}. \label{eq:assumption-g-wmap}
      \end{align}
  Then for $t \in [T, T_1]$ there holds
  \begin{align}
    \|\bs g(t)\|_\cE &\leq \frac 12 C_0|t|^{-\frac{k+1}{k-2}}, \label{eq:bootstrap-g-wmap} \\
     |\mu(t) - 1| &\lesssim C_0|t|^{-\frac{3}{k-2}}, \label{eq:bootstrap-mu-wmap} \\
     \big|\lambda'(t) - \big(\frac{2\kappa}{k-2}\big)^\frac k2\lambda(t)^\frac k2\big| &\lesssim C_0 |t|^{-\frac{k+1}{k-2}}. \label{eq:bootstrap-eq-l-wmap}
  \end{align}

\end{proposition}
Note that, unlike in Proposition~\ref{prop:bootstrap}, we do not obtain an improvement of estimates on $\lambda(t)$.
Indeed, we will need a shooting argument in order to control $\lambda(t)$.
\begin{proof}[Proof of Theorem~\ref{thm:deux-bulles-wmap} assuming Proposition~\ref{prop:bootstrap-wmap}]
  We consider the functions
  \begin{equation*}
    \beta_j(t) := \frac{k-2}{2\kappa}(\kappa|t|)^{-\frac{2}{k-2}} + j|t|^{-\frac{5}{2(k-2)}},
\end{equation*}
  for $j \in \{-1, -\frac 12, \frac 12, 1\}$.
  They will serve as barriers for $\lambda(t)$ in the topological argument.

  Let $T_n$ be a decreasing sequence converging to $-\infty$. For $n$ large and $\lambda_0 \in \cA := [\beta_{-1}(T), \beta_{1}(T)]$,
  let $\bs u_n^{\lambda_0}(t): [T_n, T_+) \to \cE$ denote the solution of \eqref{eq:wmap0} with initial data \eqref{eq:data-at-T-wmap}.
    We will prove that there exists $\lambda_0$ such that $T_+ > T_0$ and $\bs u = \bs u_n^{\lambda_0}$ verifies for $t \in [T_n, T_0]$
    the estimates \eqref{eq:bootstrap-g-wmap}, \eqref{eq:bootstrap-mu-wmap} and
    \begin{equation}
      \label{eq:bootstrap-l-wmap}
      \beta_{-1}(t) \leq \lambda(t) \leq \beta_1(t).
    \end{equation}

  Suppose that this is not the case. For each $\lambda_0\in \cA$, let $T_1 = T_1(\lambda_0)$ be the last time
    such that these three estimates hold for $t \in [T_n, T_1]$. Proposition~\ref{prop:bootstrap-wmap} implies that only \eqref{eq:bootstrap-l-wmap}
    can possibly break down. Let $\cA_+$ be the set of $\lambda_0\in \cA$ such that $\lambda(T_1) = \beta_1(T_1)$
    and $\cA_-$ the set of $\lambda_0$ such that $\lambda(T_1) = \beta_{-1}(T_1)$. We will prove that $\cA_+$ and $\cA_-$ are open sets.
    They are also non-empty since $\beta_{-1}(T) \in \cA_-$ and $\beta_1(T) \in \cA_+$.
    The end of the proof is then the same as in Subsection~\ref{ssec:limit}.

    The key step is to show that if for a certain $T_2 \in [T_n, T_1]$ we have $\lambda(T_2) > \beta_{\frac 12}(T_2)$, then $\lambda_0 \in \cA_+$.
    Suppose this is not the case. Then there exists the smallest $T_3 > T_2$ such that $\lambda(T_3) = \beta_{\frac 12}(T_3)$.
    This implies that
    \begin{equation}
      \label{eq:dlambda-leq}
    \lambda'(T_3) \leq \beta_{\frac 12}'(T_3).
  \end{equation}
  On the other hand, \eqref{eq:bootstrap-eq-l-wmap} yields
  \begin{equation}
    \begin{aligned}
    \lambda'(T_3) &\geq \Big(\frac{2\kappa}{k-2}\Big)^\frac k2\big(\lambda(T_3)\big)^\frac k2 - C_0 |T_3|^{-\frac{k+1}{k-2}} = \Big(\frac{2\kappa}{k-2}\Big)^\frac k2 \big(\beta_{\frac 12}(T_3)\big)^\frac k2 - C_0 |T_3|^{-\frac{k+1}{k-2}} \\
    &= |T_3|^{-\frac{k}{k-2}}\big(\kappa^{-\frac{2}{k-2}} + \frac{\kappa}{k-2}|T_3|^{-\frac{1}{2(k-2)}}\big)^\frac k2 - C_0 |T_3|^{-\frac{k+1}{k-2}} \\
    &\geq |T_3|^{-\frac{k}{k-2}}\Big(\kappa^{-\frac{k}{k-2}} + \frac k2\cdot \kappa^{-1} \frac{\kappa}{k-2}|T_3|^{-\frac{1}{2(k-2)}}\Big) - (C+C_0)|T_3|^{-\frac{k+1}{k-2}} \\
    &= (\kappa|T_3|)^{-\frac{k}{k-2}} + \frac{k}{2(k-2)}|T_3|^{-\frac{2k+1}{2(k-2)}} - (C + C_0)|T_3|^{-\frac{k+1}{k-2}}.
  \end{aligned}
  \end{equation}
  Since $\beta_{\frac 12}'(T_3) = (\kappa|T_3|)^{-\frac{k}{k-2}} + \frac{5}{4(k-2)}|T_3|^{-\frac{2k+1}{2(k-2)}}$ and $\frac{k}{2(k-2)} > \frac{5}{4(k-2)}$ for $k \geq 3$,
  \eqref{eq:dlambda-leq} is impossible if $|T_3| > |T_0|$ is large enough.

  Analogously, if for a certain $T_2 \in [T_n, T_1]$ we have $\lambda(T_2) < \beta_{-\frac 12}(T_2)$, then $\lambda_0 \in \cA_-$.
  Hence, if $\lambda_0 \in \cA_+$, then $\lambda(t) \geq \beta_{-\frac 12}(t)$ for $t \in [T_n, T_1]$ and $\lambda(T_1) = \beta_1(T_1)$.
  By continuity, if $|\wt \lambda_0 - \lambda_0|$ is sufficiently small, then the solution corresponding to $\wt \lambda_0$
  will satisfy $\wt \lambda(t) > \beta_{\frac 12}(t)$ at some point, hence $\wt \lambda_0 \in \cA_+$. This proves that $\cA_+$ is open.
  Analogously, $\cA_-$ is open.
\end{proof}

\subsection{Coercivity}
\label{ssec:coer-wmap}
Recall that $f'(W) = -k^2\Big(1- \frac{8}{(r^k + r^{-k})^2}\Big)$.
\begin{lemma}
  \label{lem:bulles-coer-wmap}
  There exist constants $c, C > 0$ such that
  \begin{itemize}[leftmargin=0.5cm]
    \item for all $g \in \cH$ there holds
      \begin{equation}
        \label{eq:coer-lin-coer-1-wmap}
        \begin{aligned}
          &\int_0^{+\infty}\big((g')^2+\frac{k^2}{r^2}g^2\big)\udr - \int_0^{+\infty} \frac{k^2}{r^2}\cdot\frac{8}{(r^k+r^{-k})^2} g^2 \udr \\
      &\geq c\int_0^{+\infty}\big((g')^2 +\frac{k^2}{r^2}g^2\big) \udr -C\Big(\int_0^{+\infty}\cZ\cdot g\frac{\vd r}{r}\Big)^2,
      \end{aligned}
      \end{equation}
    \item if $r_1 > 0$ is large enough, then for all $g \in \cH$ there holds
      \begin{equation}
        \label{eq:coer-lin-coer-2-wmap}
        \begin{aligned}
        &(1-2c)\int_0^{r_1}\big((g')^2+\frac{k^2}{r^2}g^2\big)\udr + c\int_{r_1}^{+\infty}\big((g')^2+\frac{k^2}{r^2}g^2\big)\udr \\
        &- \int_0^{+\infty}\frac{k^2}{r^2}\cdot\frac{8}{(r^k+r^{-k})^2} g^2 \udr\geq -C\Big(\int_0^{+\infty}\cZ\cdot g\frac{\vd r}{r}\Big)^2,
      \end{aligned}
      \end{equation}
    \item if $r_2 > 0$ is small enough, then for all $g \in \cH$ there holds
      \begin{equation}
        \label{eq:coer-lin-coer-3-wmap}
        \begin{aligned}
        &(1-2c)\int_{r_2}^{+\infty}\big((g')^2+\frac{k^2}{r^2}g^2\big)\udr + c\int_{0}^{r_2}\big((g')^2+\frac{k^2}{r^2}g^2\big)\udr \\
        &- \int_0^{+\infty}\frac{k^2}{r^2}\cdot\frac{8}{(r^k+r^{-k})^2} g^2 \udr\geq -C\Big(\int_0^{+\infty}\cZ\cdot g\frac{\vd r}{r}\Big)^2,
      \end{aligned}
      \end{equation}
  \end{itemize}
\end{lemma}
\begin{proof}
A change of variable $\wt g(x) := g(\eee^{x/k})$, $\wt \cZ(x) := \cZ(\eee^{x/k})$
reduces the problem to the study of the quadratic form associated with the classical operator $-\frac{\vd^2}{\vd x^2} + (1-2\sech^2)$
and it suffices to repeat the proof of Lemma~\ref{lem:loc-coer-ym}.
\end{proof}
Lemma~\ref{lem:bulles-coer-ym} applies verbatim to the case under consideration.
\subsection{Bootstrap}
We use the operators $A(\lambda)$ and $A_0(\lambda)$ from Subsection~\ref{ssec:bootstrap-ym}.
We define $I(t)$, $J(t)$ and $H(t)$ by the same formulas as in Subsection~\ref{ssec:bootstrap-ym}.
\begin{lemma}
  \label{lem:op-A-wmap}
  The operators $A(\lambda)$ and $A_0(\lambda)$ have the following properties:
  \begin{itemize}
    \item for all $\lambda > 0$ and $h_1, h_2 \in X^1$ there holds
      \begin{equation}
        \label{eq:A-by-parts-wmap}
        \begin{aligned}
        &\big|\big\la A(\lambda)h_1, \frac{1}{r^2}\big(f(h_1 + h_2) - f(h_1) - f'(h_1)h_2\big)\big\ra \\
        &+\big\la A(\lambda)h_2, \frac{1}{r^2}\big(f(h_1+h_2) - f(h_1) +k^2 h_2\big)\big\ra\big|
        \leq \frac{c_0}{\lambda}\cdot\|h_2\|_\cH^2,
      \end{aligned}
      \end{equation}
      with a constant $c_0$ arbitrarily small,
    \item for all $h \in X^1$ there holds
      \begin{equation}
        \label{eq:A-pohozaev-wmap}
        \big\la A_0(\lambda)h, \big(\partial_r^2 + \frac 1r\partial_r - \frac{k^2}{r^2}\big)h\big\ra \leq \frac{c_0}{\lambda}\|h\|_{\cH}^2 - \frac{2\pi}{\lambda}\int_0^{R\lambda}\big((\partial_r h)^2 + \frac{k^2}{r^2}h^2\big) \udr,
      \end{equation}
    \item assuming \eqref{eq:lambda-bound0}, for any $c_0 > 0$ there holds
    \begin{gather}
      \label{eq:L0-A0-wmap}
      \|\Lambda_0 \Lambda W_\uln{\lambda(t)} - A_0(\lambda(t))\Lambda W_{\lambda(t)}\|_{L^2} \leq c_0, \\
      \label{eq:L-A-wmap}
      \|\dot \varphi(t) + b(t)\cdot A(\lambda(t))\varphi(t)\|_{L^\infty} \leq c_0|t|^{-1}, \\
      \label{eq:approx-potential-wmap}
      \begin{aligned}
        \Big|\int_0^{+\infty}\frac 12 \big(q''\big(\frac{r}{\lambda}\big) &+ \frac{\lambda}{r}q'\big(\frac{r}{\lambda}\big)\big)\frac{1}{r^2}\big(f(\varphi + g) - f(\varphi)+k^2g\big)g \udr \\
        &- \int_0^{+\infty} \frac{1}{r^2}\big(f'(W_\lambda)+k^2\big)g^2 \udr\Big| \leq c_0C_0^2 |t|^{-\frac{2k+2}{k-2}},
    \end{aligned}
    \end{gather}
    provided that the constant $R$ in the definition of $q(r)$ is chosen large enough.
  \end{itemize} \qed
\end{lemma}
The proof is almost identical to the proof of Lemma~\ref{lem:op-A-ym} and we will skip it.

\begin{lemma}
  \label{lem:bootstrap-wmap}
  Let $c_1 > 0$. If $C_0$ is sufficiently large, then there exists a function $q(x)$ and $T_0 < 0$ with the following property.
  If $T_1 < T_0$ and \eqref{eq:lambda-bound0-wmap}, \eqref{eq:assumption-g-wmap} hold for $t \in [T, T_1]$, then for $t \in [T, T_1]$ there holds
  \begin{equation}
    \label{eq:boot-dt-wmap}
    H'(t) \leq c_1 \cdot C_0^2 |t|^{-\frac{3k}{k-2}}.
  \end{equation}
\end{lemma}
\begin{proof}
  We follow the lines of the proof of Lemma~\ref{lem:bootstrap}. We have
  $$
  I'(t) = \big\la \big(\partial_r^2 + \frac 1r \partial_r +\frac{1}{r^2}f'(\varphi)\big)\psi, g\big\ra - \la \dot \psi, g\ra - \big\la \dot \varphi, \frac{1}{r^2}(f(\varphi + g) - f(\varphi) - f'(\varphi)g)\big\ra.
  $$
The first term is $\lesssim C_0|t|^{-\frac{3k}{k-2}}$, hence negligible (by enlarging $C_0$ if necessary).
  Inequality \eqref{eq:psi1-wmap} implies that the second term can be replaced by $-\frac{b}{\lambda}(\lambda'-b)\la \Lambda_0 \Lambda W_\uln\lambda, \dot g\ra$,
  which in turn can be replaced by $-b(\lambda'-b)\la A_0(\lambda) \Lambda W_\uln\lambda, \dot g\ra$, thanks to \eqref{eq:L0-A0-wmap}.
  From \eqref{eq:L-A-wmap} we infer that the third term can be replaced by $b\cdot \big\la A(\lambda)\varphi, \frac{1}{r^2}(f(\varphi + g) - f(\varphi) - f'(\varphi)g)\big\ra$
  (indeed, $\big\|\frac{1}{r^2}(f(\varphi + g) - f(\varphi) - f'(\varphi)g)\big\|_{L^1} \lesssim \int_0^{+\infty}\frac{1}{r^2}g^2 \udr \lesssim \|g\|_\cH^2$).
  Using formula \eqref{eq:A-by-parts-wmap} with $h_1 = \varphi$ and $h_2 = g$ we obtain
  \begin{equation}
    \label{eq:dtI-wmap}
    I'(t) \simeq -b(\lambda'-b)\cdot \la A_0(\lambda)\Lambda W_\uln\lambda, \dot g\ra - b\cdot\big\la A(\lambda)g, \frac{1}{r^2}(f(\varphi + g) - f(\varphi) + k^2 g)\big\ra.
  \end{equation}

  As in the proof of Lemma~\ref{lem:bootstrap}, we obtain
  \begin{equation}
    \begin{aligned}
    (bJ)'(t) &\simeq b(\lambda' - b)\cdot \la A_0(\lambda) \Lambda W_\uln\lambda, \dot g\ra \\
    &+ b\int\big(\big(\partial_r^2 +\frac 1r\partial_r - \frac{k^2}{r^2}\big)g + \frac{1}{r^2}\big(f(\varphi + g) - f(\varphi) + k^2 g\big)\big)\cdot A_0(\lambda)g \udr,
\end{aligned}
\end{equation}
hence
\begin{equation}
  \begin{aligned}
  H'(t) &\simeq b\int(\big(\partial_r^2 +\frac 1r\partial_r - \frac{k^2}{r^2}\big)g\big)\cdot A_0(\lambda)g \udr \\ 
  &- \frac{b}{\lambda}\int_0^{+\infty}\frac 12 \big(q''\big(\frac{r}{\lambda}\big) + \frac{\lambda}{r}q'\big(\frac{r}{\lambda}\big)\big)\frac{1}{r^2}\big(f(\varphi + g) - f(\varphi)+k^2g\big)g \udr.
\end{aligned}
\end{equation}
and the conclusion follows from \eqref{eq:A-pohozaev-wmap}, \eqref{eq:approx-potential-wmap} and \eqref{eq:coer-lin-coer-2-wmap}.
\end{proof}

\begin{proof}[Proof of Proposition~\ref{prop:bootstrap-wmap}]
  We first show \eqref{eq:bootstrap-mu-wmap}. From \eqref{eq:mod-wmap} and \eqref{eq:assumption-g-wmap} we obtain
  \begin{equation}
    \label{eq:bootstrap-l-1-wmap}
    |\mu(t) - 1| = |\mu(t) - \mu(T)| \lesssim \int_{-\infty}^t C_0|t|^{-\frac{k+1}{k-2}}\ud t \lesssim C_0|t|^{-\frac{3}{k-2}}.
  \end{equation}

  Again from \eqref{eq:mod-wmap} and \eqref{eq:assumption-g-wmap} we have $|\lambda'(t) - b(t)| \lesssim C_0|t|^{-\frac{k+1}{k-2}}$.
  Multiplying by $b'(t) = \frac k2\cdot\big(\frac{2\kappa}{k-2}\big)^k\cdot \frac{\lambda(t)^{k-1}}{\mu(t)^k} \sim |t|^{-\frac{2k-2}{k-2}}$,
  cf. \eqref{eq:b-def-wmap} and \eqref{eq:lambda-bound0-wmap}, we obtain
  $\big|\dd t\big(b(t)^2 - \big(\frac{2\kappa}{k-2}\big)^k\frac{\lambda(t)^k}{\mu(t)^k}\big)\big| \lesssim C_0|t|^{-\frac{3k-1}{k-2}}$.
  Since $b(T) = \big(\frac{2\kappa}{k-2}\big)^\frac k2 \cdot\lambda(T)^\frac k2$ and $\mu(T) = 1$, 
  this yields $\big|b(t)^2 - \big(\frac{2\kappa}{k-2}\big)^k\frac{\lambda(t)^k}{\mu(t)^k}\big| \lesssim C_0|t|^{-\frac{2k+1}{k-2}}$.
  But $b(t) + \big(\frac{2\kappa}{k-2}\big)^\frac k2\frac{\lambda(t)^\frac k2}{\mu(t)^\frac k2} \sim |t|^\frac{k}{k-2}$, see \eqref{eq:lambda-bound0} and \eqref{eq:b-bound0}, hence
  \begin{equation}
    \label{eq:b-lambda-wmap}
    \big|b(t) - \big(\frac{2\kappa}{k-2}\big)^\frac k2\frac{\lambda(t)^\frac k2}{\mu(t)^\frac k2}\big| \lesssim C_0|t|^{-\frac{k+1}{k-2}}.
  \end{equation}
  Bound \eqref{eq:bootstrap-l-1-wmap} implies that $\big|\frac{\lambda(t)^\frac k2}{\mu(t)^\frac k2}-\lambda(t)^\frac k2\big| \ll |t|^{-\frac{k+1}{k-2}}$, thus \eqref{eq:b-lambda-wmap} yields
  $|\lambda'(t) - \big(\frac{2\kappa}{k-2}\big)^\frac k2 \lambda(t)^\frac k2| \lesssim C_0|t|^{-\frac{k+1}{k-2}}$, which is \eqref{eq:bootstrap-eq-l-wmap}.


  We turn to the proof of \eqref{eq:bootstrap-g-wmap}.
  From \eqref{eq:taille-correct-wmap} the initial data at $t = T$ satisfy
  $\|\bs g(T)\|_\cE \lesssim |T|^{-\frac{2k}{k-2}} \ll |T|^{-\frac{k+1}{k-2}}$, thus $H(T) \lesssim |T|^{-\frac{2k+2}{k-2}}$.
  If $C_0$ is large enough, then integrating \eqref{eq:boot-dt-wmap} we get $H(t) \leq c\cdot C_0^2|t|^{-\frac{2k+2}{k-2}}$, with a small constant $c$.
  Increasing $C_0$ if necessary and using the coercivity of $H$, we obtain \eqref{eq:bootstrap-g-wmap}.
\end{proof}

\appendix

\section{Cauchy theory}
\label{sec:cauchy}
\subsection{Persistence of regularity}
\label{ssec:persistence}

\begin{proposition}
  \label{prop:persistence}
  Let $\bs u: (T_-, T_+) \to \cE$ be the solution of \eqref{eq:nlw0} with the initial condition $\bs u(t_0) = \bs u_0$.
  If $\bs u_0 \in X^1 \times H^1$, then $\bs u \in C((T_-, T_+); X^1\times H^1) \cap C^1((T_-, T_+); \cE)$. \qed
\end{proposition}
  The proof is classical, see \cite[Chapter 5]{Cazenave03} for more general results in the case of the nonlinear Schr\"odinger equation.
  Analogous results hold in the case of equations \eqref{eq:ym0} and \eqref{eq:wmap0}.

  We recall also the following fact from the Cauchy theory of \eqref{eq:nlw} in the energy space. For $I \subset \bR$
  we denote $S(I) = L^\frac{2(N+1)}{N-2}\big(I; L^\frac{2(N+1)}{N-2}\big)$ the Strichartz norm on the time interval $I$.
  For $t \in \bR$ we denote $S(t)$ the linear wave propagator (this traditional collision of notation should not lead to misunderstanding).
  \begin{proposition}
    \label{prop:small-global}
    There exists $\eta > 0$ such that if $\|\bs u_0\|_\cE \leq \eta$, then the solution of \eqref{eq:nlw0} with the initial condition $\bs u(t_0) = \bs u_0$
    is global and
    \begin{equation}
      \sup_{t \in \bR}\|\bs u(t)\|_\cE \lesssim \inf_{t \in \bR}\|\bs u(t)\|_\cE.
    \end{equation}
    Moreover, $\|u(t)\|_{S(\bR)} < +\infty$ and there exists $\bs v_+ \in \cE$ such that
    \begin{equation}
      \lim_{t \to +\infty} \|\bs u(t) - S(t)\bs v_+\|_\cE = 0
    \end{equation}
    (analogously for $t \to -\infty$).
  \end{proposition}
\subsection{Profile decomposition and consequences}
\label{ssec:profile}
For details about the nonlinear profile decomposition for the critical wave equation we refer to \cite{BaGe99} (the defocusing case), \cite{DKM1} (dimension $N = 3$) and
\cite{Rod14} (any dimension).
For the reader's convenience we recall the following result \cite[Proposition 2.3]{Rod14}.
\begin{proposition}
  \label{prop:profile}
  Let $\bs u_{0, n}$ be a bounded sequence in $\cE$ admitting a profile decomposition with profiles $\bs U^j\lin$ and parameters $\lambda_{j, n}$, $t_{j, n}$.
  Let $\bs U^j$ be the corresponding nonlinear profiles and let $\theta_n$ be a sequence of positive numbers. Assume that
  \begin{equation}
    \forall j\geq1, n\geq 1,\ \frac{\theta_n - t_{j, n}}{\lambda_{j, n}} < T_+(\bs U^j)\quad\text{and}\quad \limsup_{n\to+\infty}\|U^j\|_{S\big(\frac{-t_{j, n}}{\lambda_{j, n}}, \frac{\theta_n - t_{j, n}}{\lambda_{j, n}}\big)} < +\infty.
  \end{equation}
  Let $\bs u_n$ be the solution of \eqref{eq:nlw0} with the initial data $\bs u_n(0) = \bs u_{0, n}$.
  Then for $n$ sufficiently large $\bs u_n$ is defined on $[0, \theta_n]$ and
  \begin{equation}
    \label{eq:nonlin-profile}
    \bs u_n(t) = \sum_{j = 1}^{J}\bs U^j\Big(\frac{t-t_{j,n}}{\lambda_{j,n}}\Big)_{\lambda_{j,n}} + \bs w_n^J(t) + \bs r_n^J(t),\qquad \text{for all }J\in \bN\text{ and }t\in [0, \theta_n],
  \end{equation}
  with $\lim_{J\to+\infty}\limsup_{n\to+\infty}\sup_{t\in[0, \theta_n]}\|\bs r_n^J\|_\cE = 0$.
  An analogous statement holds for $\theta \leq 0$.
\end{proposition}
For a corresponding result for the Yang-Mills equation and the equivariant wave maps, see \cite{JiKe15p}.
\begin{corollary}
  \label{cor:quitte-compact}
  There exists a constant $\eta > 0$ such that the following holds.
  Let $\bs u: [t_0, T_+) \to \cE$ be a maximal solution of \eqref{eq:nlw0} with $T_+ < +\infty$.
  Then for any compact set $K \subset \cE$ there exists $\tau < T_+$ such that $\dist(\bs u(t), K) > \eta$ for $t \in [\tau, T_+)$.
\end{corollary}
\begin{proof}
  Suppose for the sake of contradiction that there exists a sequence $T_n \to T_+$ such that for $\bs u_{0,n} := \bs u(T_n)$ there holds $\dist(\bs u_{0,n}, K) \leq \eta$,
  hence $\bs u_{0,n} = \bs k_n + \bs b_n$ with $\bs k_n \in K$ and $\|\bs b_n\|_\cE \leq \eta$. By taking a subsequence we can assume that $\bs k_n \to \bs k_0 \in K$
  and that the sequence $\bs b_n$ admits a profile decomposition, in particular $\bs b_n \wto \bs b_0 \in \cE$. This implies that $\bs u_{0,n}$
  admits a profile decomposition with the first profile $\bs U^1 = \bs k_0 + \bs b_0$ (with parameters $\lambda_{1, n} = 1$ and $t_{1, n} = 0$)
  and all the other profiles small in the energy norm. Proposition~\ref{prop:profile} leads to a contradiction.
\end{proof}
\begin{lemma}
  \label{lem:small-to-zero}
  Let $\bs U$ be a solution of \eqref{eq:nlw} such that $\|\bs U(t)\|_\cE$ is small (for any, hence for all $t$). Let $t_n$, $\lambda_n$ be a sequence
  of parameters such that
  \begin{equation}
    \label{eq:appendix-param-orth}
    |\log(\lambda_n)| + \frac{|t_n|}{\lambda_n} \to +\infty.
  \end{equation}
  Then for all $t \in \bR$ there holds
  \begin{equation}
    \label{eq:appendix-weak}
    \bs U\Big(\frac{t-t_n}{\lambda_n}\Big)_{\lambda_n} \wto 0 \qquad \text{in }\cE.
  \end{equation}
\end{lemma}
\begin{proof}
  It is sufficient to prove \eqref{eq:appendix-weak} for a subsequence of any subsequence, thus we can assume that $\frac{t - t_n}{\lambda_n} \to t_0 \in [-\infty, +\infty]$.

  Suppose first that $t_0 \in (-\infty, +\infty)$, hence $\bs U\big(\frac{t-t_n}{\lambda_n}\big) \to \bs U(t_0)$.
  Extracting again a subsequence we can assume that $\lambda_n \to \lambda_0 \in [0, +\infty]$.
  Then \eqref{eq:appendix-param-orth} implies that $\lambda_0 = 0$ or $\lambda_0 = +\infty$, and in both cases we get \eqref{eq:appendix-weak}.

  Suppose that $t_0 = +\infty$ (the case $t_0 = -\infty$ is the same). Let $\bs V_+ \in \cE$ be such that $\lim_{t \to +\infty}\|\bs U(t) - S(t)\bs V_+\|_\cE = 0$.
  Again, we can assume that $\lambda_n \to \lambda_0 \in [0, +\infty]$. It is well known that $S(t) \bs V_+ \wto 0$ as $t \to +\infty$,
  which settles the case $\lambda_0 \in (0, +\infty)$.
  In the case $\lambda_0 = 0$ or $\lambda_0 = +\infty$, notice that $S\big(\frac{t-t_n}{\lambda_n}\big)$ is a Fourier multiplier,
  hence the Fourier support of $\big(S\big(\frac{t-t_n}{\lambda_n}\big)\bs V_+\big)_{\lambda_n}$ is concentrated in an annulus of characteristic size $\lambda_n^{-1}$,
  which implies the weak convergence to $0$ on the Fourier side.
\end{proof}
\begin{corollary}
  \label{cor:faible}
  There exists a constant $\eta > 0$ such that the following holds.
  Let $K \subset \cE$ be a compact set and let $\bs u_n: [T_1, T_2] \to \cE$ be a sequence of solutions of \eqref{eq:nlw0}
  such that
  \begin{equation}
    \label{eq:proche-K}
    \dist(\bs u_n(t), K) \leq \eta,\qquad \text{for all }n\in \bN\text{ and }t \in [T_1, T_2].
  \end{equation}
  Suppose that $\bs u_n(T_1) \wto \bs u_0 \in \cE$. Then the solution $\bs u(t)$ of \eqref{eq:nlw0} with the initial condition $\bs u(T_1) = \bs u_0$
  is defined for $t \in [T_1, T_2]$ and
  \begin{equation}
    \label{eq:weak}
    \bs u_n(t) \wto \bs u(t),\qquad \text{for all }t \in [T_1, T_2].
  \end{equation}
\end{corollary}
\begin{proof}
  It suffices to prove \eqref{eq:weak} for a subsequence of any subsequence. Hence, we may assume that $\bs u_{0,n} := \bs u_n(T_1)$ admits a profile decomposition.
  As in the proof of Corollary~\ref{cor:quitte-compact}, we show that all the profiles except for $\bs U^1 = \bs u_0$ are small in the energy norm.
  Moreover, for $j > 1$ the sequence of parameters $(\lambda_{j,n}, t_{j,n})$ is pseudo-orthogonal to $(\lambda_{1,n}, t_{1,n}) = (1, 0)$, which means precisely that
  $(\lambda_n, t_n) = (\lambda_{j,n}, t_{j,n})$ verifies \eqref{eq:appendix-param-orth}.
  Proposition~\ref{prop:profile} implies \eqref{eq:weak}.
\end{proof}
\begin{remark}
  Corollaries~\ref{cor:quitte-compact} and \ref{cor:faible} hold for the Yang-Mills and wave map equations, with the same proofs.
\end{remark}
\begin{remark}
  An important point of both results is that $\eta$ is independent of $K$.
  Corollary~\ref{cor:quitte-compact} states that a blow-up cannot happen at a small distance from a compact set.
  Corollary~\ref{cor:faible} establishes sequential weak continuity of the flow in a neighbourhood of any compact set.
  Without this additional condition weak continuity is expected to fail, a counterexample being provided by type~II blow-up solutions.
\end{remark}
\begin{remark}
  Corollaries~\ref{cor:quitte-compact} and \ref{cor:faible} are crucial ingredients of the arguments in Subsection~\ref{ssec:limit}.
Using the nonlinear profile decomposition of Bahouri and G\'erard \cite{BaGe99} is nowadays a well-established method of attacking this type of questions in critical spaces.
Note that \cite{BaGe99} gave the first proof of the sequential weak continuity of the flow for the defocusing energy-critical wave equation.
\end{remark}

%
%

\bibliographystyle{plain}
\bibliography{deux-bulles}

\end{document}